\documentclass[10pt]{amsart}

\usepackage{amssymb, amscd}
\usepackage{mathtools}	
\usepackage{tensor}		
\usepackage{bbm}		
\usepackage{tikz-cd}	
	\usetikzlibrary{decorations.pathmorphing} 
\usepackage[all]{xy}
\usepackage{upref} 
\usepackage{lmodern}
\usepackage[margin=3.5cm]{geometry}
\usepackage{MnSymbol}	
\usepackage{enumitem}	
\usepackage{stmaryrd}	
\usepackage{microtype}
\usepackage[draft=false]{hyperref}

\newcommand{\act}{\curvearrowright}

\newcommand{\C}{\mathbb{C}}
\newcommand{\R}{\mathbb{R}}
\newcommand{\Q}{\mathbb{Q}}
\newcommand{\Z}{\mathbb{Z}}

\newcommand*{\abs}[1]{\lvert#1\rvert}

\newcommand{\SL}{\mathrm{SL}}

\newcommand{\comm}[2]{\left[ #1, #2 \right]}	

\newcommand{\G}{\mathcal{G}}

\newcommand{\A}{\mathcal{A}}
\newcommand{\inv}{^{-1}}

\newcommand{\T}{\mathbb{T}}

\newcommand{\module}{\mathcal{L}}

\newcommand{\gpdcomp}[8]{
 	\begin{tikzcd}[ampersand replacement=\&, column sep=large]
			\scriptstyle {#1} 
		\&
			\scriptstyle {#3}
			\arrow[l, bend right, "{\displaystyle #2}"', start anchor={[xshift=7pt]north west}, end anchor={[xshift=-7pt]north east}]
		\&
			\scriptstyle {#5}
			\arrow[l, bend right, "{\displaystyle #4}"'{name=U}, start anchor={[xshift=7pt]north west}, end anchor={[xshift=-7pt]north east}]
		\&
			\arrow[l, white, dash, bend right, "{\displaystyle \vphantom{#6}=\vphantom{#4}}"'{black}, start anchor={[xshift=7pt]north west}, end anchor={[xshift=-7pt]north east}]
			\scriptstyle {#1}	
		\&
			\scriptstyle {#7}
			\arrow[l, bend right, "{\displaystyle #6}"'{name=D}, start anchor={[xshift=7pt]north west}, end anchor={[xshift=-7pt]north east}]
	\end{tikzcd}
}

\newcommand{\mat}[1]{\left[ \begin{smallmatrix} #1 \end{smallmatrix}\right]} 
\newcommand{\Mat}[1]{\left[ \begin{matrix} #1\end{matrix}\right]} 
\renewcommand{\and}{\quad\textup{and}\quad}

\newcommand{\der}[2]{\tfrac{\partial #1}{\partial #2}}
\renewcommand{\d}{\,\mathrm{d}}					

\renewcommand{\sup}[2]{\underset{#1}{\mathsf{sup}}{\;#2}}	
\newcommand{\imaginary}{i}					

\newcommand{\ket}[1]{\left\vert #1\right\rangle}
\newcommand{\bra}[1]{\left\langle #1\right\vert} 
\newcommand{\Hilbertspace}{W}

\newcommand{\homotopy}{Y}
\newcommand{\slope}{\left(\begin{smallmatrix} \theta \\ 1 \end{smallmatrix}\right)}
\newcommand{\incl}{ \iota }
\newcommand{\Schw}{\mathcal{S}}		
\newcommand{\normS}[3]{\left\Vert #1 \right\Vert^{\Schw}_{\left( #2 , #3 \right)}} 
\newcommand{\normSn}[3]{\left\Vert #1 \right\Vert^{\Schw_{ n }}_{\left( #2 , #3 \right)}} 
\newcommand{\normSt}[3]{\left\Vert #1 \right\Vert^{\Schw_{ 2 }}_{\left( #2 , #3 \right)}}

	\DeclareMathOperator{\Ad}{\mathsf{Ad}}
	\DeclareMathOperator{\dom}{Dom}
	\let\del\partial
		
	\newcommand{\norm}[1]{\left\|#1\right\|}
	
	\newcommand{\inner}[2]{\left\langle #1\,\vert\, #2 \right\rangle}
	\newcommand{\iinner}[2]{\left\llangle #1 \,\big\vert\, #2\right\rrangle}

\newcommand{\Cstar}{\mathrm{C^{*}}}

\newcommand{\B}{\mathcal{B}}

\newcommand{\LL}{\mathcal{L}}

\newcommand{\Bound}{\LL}

\newcommand{\Hilb}{\mathcal{E}}

\newcommand{\K}{\mathrm{K}}
\newcommand{\Dirac}{D}
\newcommand{\Dudelta}{\widehat{\Delta}}

\newcommand{\KK}{\mathrm{KK}}

\newcommand{\RKK}{\textup{RKK}}

\newcommand{\twist}{\tau}

\newcommand{\e}[1]{\mathsf{e}^{ 2\pi i #1}}
\newcommand{\me}[1]{\mathsf{e}^{ -2\pi i #1}}

\newcommand{\pr}{\textup{pr}}

\newcommand{\id}{\mathrm{id}}

\numberwithin{equation}{section}

\theoremstyle{theorem}
\newtheorem{theorem}[equation]{Theorem}
\newtheorem{lemma}[equation]{Lemma}
\newtheorem{proposition}[equation]{Proposition}
\newtheorem{corollary}[equation]{Corollary}

\theoremstyle{definition}
\newtheorem{definition}[equation]{Definition}

\theoremstyle{remark}
\newtheorem{remark}[equation]{Remark}

\begin{document}

\title{Transversals, duality, and irrational rotation}

\author{Anna Duwenig}
\email{aduwenig@uvic.ca}
\address{Department of Mathematics and Statistics\\
 University of Victoria\\
 PO BOX 3045 STN CSC\\
 Victoria, B.C.\\
 Canada V8W 3P4}

\author{Heath Emerson}
\email{hemerson@math.uvic.ca}
\address{Department of Mathematics and Statistics\\
 University of Victoria\\
 PO BOX 3045 STN CSC\\
 Victoria, B.C.\\
 Canada V8W 3P4}


\date{\today}

\thanks{This research was supported by an NSERC Discovery grant.}

\begin{abstract}
An early result of Noncommutative Geometry was 
Connes' observation in the 1980's that the Dirac-Dolbeault cycle for the \(2\)-torus \(\T^2\), which 
induces a Poincar\'e self-duality for \(\T^2\), can be 
`quantized' to give a spectral triple and a K-homology class in 
\(\KK_0(A_\theta\otimes A_\theta, \C)\) providing the co-unit for a
 Poincar\'e self-duality for the irrational rotation algebra \(A_\theta\) for any \(\theta\in \R\setminus \Q\). 
 This spectral triple has been extensively studied since. Connes'
proof, however, relied on a K-theory computation and does not supply a representative 
cycle for the unit of this duality. Since such representatives are vital in applications of 
duality, we supply such a cycle in unbounded form in this article. Our approach is 
to construct, for any non-trivial element \(g\) of the modular group, a finitely generated projective 
module \(\module_g\) over \(A_\theta \otimes A_\theta\) by using a reduction-to-a-transversal 
argument of Muhly, Renault, and Williams, applied to a pair of Kronecker foliations 
along lines of slope \(\theta\) and \(g(\theta)\), using the fact that these flows are transverse to each other. 
We then compute Connes' dual of 
\([\module_g]\) for \(g\) upper triangular, and prove that we obtain
 an invertible in \(\KK_0(A_\theta, A_\theta)\), represented by what one might regard as a 
noncommutative bundle of Dirac-Schr\"odinger operators. An application of \(\Z\)-equivariant 
Bott Periodicity proves that twisting the module by the family gives the requisite spectral cycle for 
the unit, thus proving self-duality for \(A_\theta\) with both unit and co-unit represented by 
spectral cycles.

\end{abstract}

\maketitle


\section{Introduction}

The (irrational) rotation algebra \(A_\theta\) is the crossed-product C*-algebra 
\(C(\T)\rtimes_\theta \Z\) associated to 
a rotation \(z\mapsto e^{2\pi i \theta} z\) of the circle by 
an (irrational) angle \(\theta\). The complex coordinate \(V (z) = z\) on \(\T\) and the 
generator \(U\) of the group action in the crossed-product, are a pair of unitaries in 
\(A_\theta\) which generate it as a C*-algebra, and satisfy the relation 
\[ VU = e^{2\pi i \theta} UV.\]
In particular, when \(\theta = 0\) we obtain the commutative C*-algebra \( C(\T^2)\) of 
continuous functions on the \(2\)-torus, and accordingly \(A_\theta\) is often called the 
`noncommutative torus.'

 Compact spin\(^c\)-manifolds such as \(\T^2\) exhibit duality in \(\KK\).  Two 
C*-algebras 
 \(A\) and \(B\) are \emph{dual in \(\KK\)} 
if there exists a pair of classes 
\[ \Delta \in \KK_0(A\otimes B, \C), \quad \widehat{\Delta} \in \KK_0(\C, 
B\otimes A)\]
satisfying the \emph{zig-zag equations}: 
\begin{equation*}
	(1_A \otimes \widehat{\Delta}) \otimes_{A\otimes B\otimes A} (\Delta \otimes 1_A)
	= 1_{A},
	\qquad
	(\widehat{\Delta} \otimes 1_{B}) \otimes_{B\otimes A \otimes B} (1_{B} \otimes \Delta)
	= 1_{B}.
\end{equation*}
We will refer to the class 
\(\widehat{\Delta}\) as the \emph{unit}, and \(\Delta\) as the \emph{co-unit} of the 
duality, with reference to the theory of adjoint functors. A cup-cap operation 
using \(\Delta\) determines a map 
\[ \Delta\cup  \text{\textvisiblespace}\,\colon \KK_i(D_1, A_\theta \otimes D_2) \cong \KK_i (A_\theta\otimes D_1, D_2),\]
for any pair \(D_1, D_2\) of separable C*-algebras, and it can be checked that 
\(\widehat{\Delta}\) provides a similar map which inverts it, because of the zig-zag 
equations.

If \(X\) is a compact spin\(^c\)-manifold, then the diagonal embedding 
\(\delta \colon X \to X\times X\) 
has a normal bundle \(\nu\) with canonical \(\K\)-orientation 
has a Thom class \(\xi \in 
\K^{-n}(\nu)\). Using a 
tubular neighbourhood embedding \(\nu \subseteq X\times X\), we can extend 
the Thom cycle and class to zero outside the neighbourhood, yielding a \(\K\)-theory 
class for \(X\times X\) that is equal by definition to 
\(\widehat{\Delta}\in \KK_{+n} \left(\C, C(X\times X)\right)\), and which is supported 
 in an arbitrarily small neighbourhood of the diagonal \(X\subset X\times X\). 
This construction determines the unit for a self-duality for \(X\). 

The co-unit \(\Delta \in \KK_{-n} (C(X\times X), \C)\) in this duality is 
represented, analytically, by the 
Dirac cycle for \(X\), consisting of the Dirac operator acting on 
\(L^{2}\)-spinors on \(X\). This gives a cycle for \(\KK_{-n} (C(X), \C)\), and 
pulling it back by the *-homomorphism 
\( C(X\times X) \to C(X)\) of restriction to the diagonal results in a cycle for 
 \(\KK_{-n}(C(X\times X), \C)\).

In the 80's, Connes suggested that there might be C*-algebras which behave in 
some sense like `noncommutative manifolds,' and one possible 
way in which this might happen would be if there were examples of C*-algebras
arising in geometric situations, which 
 exhibit \(\KK\)-duality. He pointed out  that 
the Dirac cycle for the \(2\)-torus can be adapted slightly so as to give a cycle and 
class 
\(\Delta_{\theta} \in \KK_0(A_{\theta} \otimes A_{\theta}, \C)\)
even for the noncommutative \(A_\theta\)'s. Connes 
verified by direct computation that it induces duality (see 
\cite{Connes:Gravity} and \cite{Connes:NCG}). 
There are now several other examples of C*-algebras with KK-theoretic
duality: groupoid C*-algebras arising from 
hyperbolic dynamical systems (\cite{KaPut:dual-SFT} and \cite{KaPut:dual-SFT2}), 
 crossed-products by actions of Gromov 
hyperbolic groups on their boundaries 
\cite{Em:bdActions}, and   
orbifold C*-algebras \cite{EEK}.  In some cases, the Baum-Connes conjecture boils down to 
a form of duality between a group C*-algebra and its classifying space, 
and some of these special cases are studied in \cite{W}. Duality has seemed to be 
quite a successful notion in Noncommutative Geometry.

Connes' remark about \(A_\theta\) was that the class \(\Delta_\theta\) built from the 
Dolbeault operator on \(\T^2\) induces a self-duality for \(A_\theta\) because 
the induced intersection form 
\[ \K_*(A_\theta) \times \K_*(A_\theta) \longrightarrow \K_*(A_\theta\otimes A_\theta) \stackrel{\langle \cdot , \Delta_\theta \rangle }{\longrightarrow} \Z\]
can be computed directly and is non-degenerate. Connes computed a formula for the 
unit \(\widehat{\Delta}_\theta\) in terms of known \(\K\)-theory generators for \(A_\theta\), 
but it 
has 
no obvious representative cycle. 
Duality for C*-algebras is  interesting when it is implemented by explicit cycles. Cycles
 lead to applications
(for example to `noncommutative Lefschetz fixed-point formulas' \cite{Em:Lefschetz}.)

 The purpose of this article is to remedy this situation and describe a geometrically 
defined spectral (that is, unbounded) 
 cycle for \(\K_0(A_\theta\otimes A_\theta)\) representing \(\widehat{\Delta}_\theta\). 
 We do this by showing that Connes' class \(\Delta_\theta\) and our class satisfy 
 the first zig-zag equation, and 
 thus comprise a self-duality with co-unit Connes' \(\Delta_\theta\), which implies 
 that our co-unit equals his \(\widehat{\Delta}_\theta\). Since we are using purely geometric 
 methods, in principal our methodology should apply to more general situations, as we 
 discuss at the end of this introduction.

Our definition of \(\widehat{\Delta}_\theta\) 
 involves two ingredients. The first is the construction of an invertible 
morphism 
\[ \twist_b \in \KK_0(A_\theta, A_\theta),\]
for any \(b\in \Z\), which we call the `\(b\)-twist,' and which is represented by applying descent 
\[ \KK^\Z_0(C(\T), C(\T)) \to \KK_0(C(\T)\rtimes_\theta \Z, C(\T)\rtimes_\theta \Z) = \KK_0(A_\theta, A_\theta)\]
to the class of a certain, quite simple \(\Z\)-equivariant bundle of Dirac-Schr\"odinger operators 
\( \frac{\partial}{\partial r} + r\) over the circle \(\T\). 
The \(b\)-twist has the features of acting as multiplication by the matrix 
\( \begin{bsmallmatrix} 1 & b \\ 0 & 1\end{bsmallmatrix}\) on \(\K_0(A_\theta)\) with the standard 
identification \(\K_0(A_\theta) \cong \Z^2\), and acting as the identity on \(\K_1(A_\theta)\).
In particular, \(\twist_b\) is not represented by any \emph{automorphism} of \(A_\theta\), since 
automorphisms act as the identity on \(\K_0\).

The point at which Bott Periodicity enters into our proof, resides in the fact we show that the morphisms 
\(\{\twist_b \}_{b\in \Z}\) form a cyclic group in the invertibles in \(\KK_0(A_\theta, A_\theta)\).
We also show that the \(b\)-twist agrees, under the Dirac 
\(\KK_1\)-equivalence 
\[ A_\theta = C(\T)\rtimes_\theta \Z \cong_{\KK_1} C_0(\R\times \T)\rtimes_\theta \Z 
\sim C(\T^2)\]
with the class of the linear automorphism of \(\T^2\) given by matrix multiplication 
by \( \begin{bsmallmatrix} 1 & b \\ 0 & 1\end{bsmallmatrix}\). See 
Theorem \ref{theorem:keyindexreduction}. The \(b\)-twist provides the operator which is going to enter into our cycle and the second 
ingredient of the construction determines the module. 

Let \(\B_\theta\) and \(\B_{\theta+b}\) denote the transformation groupoids 
corresponding to Kronecker flows on \(\T^2\) along lines of slope \(\theta\) and 
\(\theta+b\). 
Since \(\B_\theta \) and \(\B_{\theta+b}\) are transverse, the restriction of the groupoid \( \B_\theta \times \B_{\theta+b}\) to the diagonal \(\T^2\) in its unit space \(\T^2\times \T^2\) is \'etale.
A well-known construction of 
Muhly, Renault, and Williams \cite{MRW:Grpd} provides an explicit strong Morita equivalence 
between the restricted groupoid and \(\B_\theta\times \B_{\theta+b}\), and hence with 
\((\T\rtimes_\theta \Z) \, \times \, (\T\rtimes_{\theta + b} \Z)\), and then with 
\((\T\rtimes_\theta \Z)\, \times \, (\T\rtimes_\theta \Z)\). 

We obtain a strong Morita equivalence between the unital C*-algebra of the restricted groupoid and  \(A_\theta\otimes A_\theta\). Since the former is 
\'etale, the strong Morita equivalence 
bimodule is finitely generated projective as an \(A_\theta\otimes A_\theta\)-module. Let 
\( [\module_b]\in \KK_0(\C, A_\theta \otimes A_\theta)\) be its class.

The main result of this article is: 

\begin{theorem}
The class \(\Delta_\theta\) of Connes, and \( \widehat{\Delta}_\theta := 
(1_{A_\theta} \otimes \twist_{-b})_* ([\module_b])\)  for $b>0$ are the co-unit and unit of a 
\(\KK\)-self-duality for \(A_\theta\). 
\end{theorem}

The description of \(\widehat{\Delta}_\theta\) given in the theorem 
leads to an explicit unbounded 
representative of \(\widehat{\Delta}_\theta\) in the form of a self-adjoint operator on a 
Hilbert module -- a kind of `quantized' Thom class for the diagonal embedding 
\(\T^2_\theta \to \T^2_\theta \times \T^2_\theta\). See Theorem \ref{thm:cycle-rep-Dudelta}
for the exact statement.

\section{Preliminaries}
\subsection{Irrational rotation on the circle}
In this paper, we are mainly interested in a class of group actions, but we will use 
groupoid methods prolifically.

	Irrational rotation on the circle \(\T\) is given by
	 the $\Z$-action \(n \mapsto \alpha_n\) where
	 $\alpha_{n} ([x]) = 
	[x+n\theta]$, \([x]\in \T := \R/\Z\). The action determines a 
	 transformation groupoid $\A_{\theta}:= \T\rtimes_{\theta} \Z$ with composition rules
	\[
		\gpdcomp{
			\left[x\right]
		}{
			\left(\left[x\right],\,n\right)
		}{
			\left[x-n\theta\right]
		}{
			\left(\alpha_{-n}\left[x\right]\,m\right)
		}{
			\left[(x-n\theta)-m\theta\right]
		}{
			\left(\left[x\right],\, n+m\right)
		}{
			\left[x-(n+m)\theta\right]
		}{
			0
		}
	\]
	Inverses are given by $\left(\left[x\right],\,n\right)^{-1} = \left(\alpha_{-n}\left[x\right],\,-n \right)$.

	The \emph{irrational rotation algebra} $A_{\theta}$ is the groupoid C*-algebra of this groupoid. Equivalently, \(A_{\theta}\) is the crossed-product 
	\[ A_{\theta} := C^{*}(\A_{\theta}) \cong C(\T)\rtimes_{\theta} \Z.\]

As is well-known, the
 irrational rotation algebra is the universal C*-algebra $A_{\theta}$ generated by two unitaries $U, V$ subject to the relation
	$
		VU = e^{2\pi \imaginary \theta} UV.
	$
	Note that 
	\begin{equation}\label{eq:defn-mfA}
		\mathfrak{A}
		 := 
		\left\{
			\sum_{n,m\in\mathbb{Z}} a_{n,m} V^n U^m
			\;|\;
			(a_{n,m})_{n,m} \in S(\Z^{2})
		\right\}
	\end{equation}
	is a dense subalgebra, where $(a_{n,m})_{n,m} \in \mathcal{S}(\Z^{2})$ if and only if for all $k\in\Z^{+}$, 
	\[
		\sup{n,m}
		{
			\left\{
				\left(
					\abs{n}^k
					+
					\abs{m}^k
				\right)
				\abs{a_{n,m}}
			\right\}
		}
		< \infty.
	\]
In the crossed product picture, $V$ corresponds to the generator of 
$C(\mathbb{T})$ and $U$ to the generator of $\mathbb{Z}$.

As such, \(A_\theta\) is sometimes referred to as the \emph{noncommutative torus,} since 
the C*-algebra \(C(\T^2)\) of continuous functions on the \(2\)-torus, is generated by 
two \emph{commuting} unitaries \(U, V\) (namely, the coordinate projections).

\subsection{Poincar\'{e} duality}
	
	A
	\(\KK\)-theoretic Poincar{\'e} duality between two C*--algebras 
	\(A\) and \(B\), determines an isomorphism between the 
	\(\K\)-theory groups of \(A\) and the \(\K\)-homology groups of \(B\). An important motivating example comes from smooth manifold theory: If \(X\) is a smooth compact manifold, then it is a result of Kasparov that 
	\(C(X)\) is Poincar{\'e} dual to \(C_0(TX)\), where 
	\(TX\) is the tangent bundle. The Poincar{\'e} duality isomorphism sends the \(\K\)-theory class defined by the symbol of an elliptic operator, to the \(\K\)-homology class of the operator.
	
	If \(X\) carries a spin\(^c\)-structure, \emph{i.e.}\ a \(\K\)-orientation on its tangent bundle, then \(C_0(TX)\) is \(\KK\)-equivalent to \(C(X)\) by the Thom isomorphism, and 
	so \(C(X)\) has a \emph{self-duality} of a dimension shift of \(\dim X\). A basic example is \(X = \T^2\). 
	
Duality in this sense is an example of an adjunction of functors, and is, like with adjoint functors in general, determined by two classes, usually called the 
the \emph{unit} and \emph{co-unit}, here denoted 
	\(\widehat{\Delta}\) and \(\Delta\) respectively.

		\begin{definition}
			We say that two (nuclear, separable, unital) C*-algebras $A, B$ are \textit{Poincar{\'e} dual} (with dimension shift of zero) 
			if there exist $\Delta\in \KK_0
			(A\otimes B, \mathbb{C})$ and $\widehat{\Delta}\in \KK_0
			(\mathbb{C},B\otimes A)$ which satisfy the following so-called \emph{zig-zag equations},
			\begin{equation}\label{eq:zigzag}
			\begin{split}
				\widehat{\Delta}\otimes_B \Delta 
				:=&
				(1_{A} \otimes \widehat{\Delta})
				\otimes_{A\otimes B\otimes A}
				(\Delta \otimes 1_{A})
				=
				1_A \in \KK
				(A, A)
				\quad\and 
				\\
				\widehat{\Delta}\otimes_A \Delta 
				:=&
				(\widehat{\Delta} \otimes 1_{B})
				\otimes_{B\otimes A\otimes B}
				(1_{B} \otimes \Delta)
				= 1_B \in \KK
				(B, B).
			\end{split}
			\end{equation}
			We call $(\Delta,\widehat{\Delta})$ (\textit{Poincar{\'e}\textup{)} duality pair}.
		\end{definition}

The co-unit \(\Delta \in \KK_0(A\otimes B, \C)\), for example, determines a 
	cup-cap product map 
	\begin{equation}
	\label{equation:connesdualitymap}
	\Delta \cup \text{\textvisiblespace}\,\colon \KK_*(D_1, B\otimes D_2) \to \KK(A\otimes D_1, D_2),\;\;\;
	\Delta \cup f := (1_{A}\otimes_\C f) \otimes_{A \otimes B} \Delta.
	\end{equation}	
	The unit can be similarly used to define a system of maps dual to the above, and 
	some manipulations show that the maps are inverse if the zig-zag equations hold.

There are now a number of examples of Poincar{\'e} dual pairs of C*-algebras: 
see \cite{EEK}, \cite{KaPut:dual-SFT}, \cite{Em:bdActions}, \cite{KaPut:dual-SFT2}. The first \emph{noncommutative} example, a Poincar{\'e} duality between the irrational rotation algebra \(A_\theta\), is due to Connes (see \cite{Connes:NCG}) and is the primary interest of this article.

Although we have not included it in the definition, one hopes to find explicit \emph{cycles} for the classes \(\Delta\) and \(\widehat{\Delta}\) in a Poincar{\'e} duality. Connes has defined a cycle \cite[p.\ 604]{Connes:NCG}
whose class \(\Delta_\theta \in \KK_0(A_\theta \otimes A_\theta , \C)\) 
determines the duality for the irrational rotation algebra \(A_\theta\) alluded to above, but 
the formula he gave for the dual class 
\(\widehat{\Delta}_\theta\in\KK_0(\C , A_\theta\otimes A_\theta) = \K_0(A_\theta \otimes A_\theta)\) was of the type 
\( \widehat{\Delta} = x\otimes_\C y + x'\otimes_\C y' + \ldots \,\), where \(x,x'\in \K_*(A)\) and 
\(y, y'\in \K_*(B)\), and \(\otimes_\C\) refers to the external product in \(\KK\); this does not specify a cycle, but a class. It is 
this missing cycle, representing \(\widehat{\Delta}_\theta\), 
that this article aims to supply. 
	
Connes' class \(\Delta_{\theta} \in \KK_0(A_\theta\otimes A_\theta, \C)\)  can be defined in the following way. 

\begin{lemma}\label{def:Delta}
	On $L^{2} := L^{2} (\T\times\Z)$, define
	\[	
		\omega_{1} , \omega_{2} \colon C(\T)\to \B(L^{2})
		\quad\and\quad
		u, v \colon \Z^d \to \mathcal{U }(L^{2})
	\]
	for $f\in C(\T), k\in \Z, \xi\in L^{2} (\T ), \mathsf{e}_n\in \ell^{2} (\Z)$ by
	\[
	\begin{tikzcd}
		\omega_{1} (f) \left(\xi\otimes \mathsf{e}_n \right)
		 := 
		\left(\alpha_{-n} (f)\cdot\xi\right)\otimes \mathsf{e}_n 
		\\[-25pt]
	\omega_{2} (f) \left(\xi\otimes \mathsf{e}_n \right)
		 := 
		(f\cdot\xi)\otimes \mathsf{e}_n 
	\end{tikzcd}
	\quad\and\quad
	\begin{tikzcd}
		u_k \left(\xi\otimes \mathsf{e}_n \right)
 := 
		\xi\otimes \mathsf{e}_{k+n} 
		\\[-25pt]
		v_k \left(\xi\otimes \mathsf{e}_n \right)
		 := 
		(k.\xi)\otimes \mathsf{e}_{n-k}, 
	\end{tikzcd}
	\]
	where $k.\xi = \xi\circ\alpha_{-k}$ for $\xi$ in the subspace $C(\T) \subseteq L^{2} (\T)$.
	
	Then the pairs $(\omega_{1} , u )$ and $(\omega_{2} , v )$ are covariant for 
	$(C(\T), \alpha , \Z)$ and hence induce representations of $A_{\theta}$ on $L^{2}$. Moreover, these two representations commute and thus give a representation $\pi$ of $A_{\theta} \otimes A_{\theta}$ on $L^{2}$, so we obtain an unbounded cycle 
	\[
	 (L^{2}\oplus L^{2}, \pi\oplus\pi, d_{\Delta})
		\in
		\Psi (A_{\theta} \otimes A_{\theta}, \C)
	\]
	where 
	\[
		d_{\Delta} := 
		\begin{bmatrix} 0& \Dirac_\Z - \imaginary \Dirac_{\T} \\ \Dirac_\Z + \imaginary \Dirac_{\T} & 0\end{bmatrix}
	\]
	with
	\[
		\Dirac_{\Z} (V^m U^n) = 2\pi n \cdot V^m U^n,
		\quad
		\Dirac_{\T} (V^m U^n) = 2\pi m \cdot V^m U^n,\quad
		\textit{and }
		\dom (\Dirac_{\Z})=\dom (\Dirac_{\T})=\mathfrak{A}
		.
	\]
\end{lemma}
For the definition of $\mathfrak{A}$, see Equation~\eqref{eq:defn-mfA}.

\begin{definition}
\label{def:fc}
We let \(\Delta_\theta \in \KK_0(A_\theta \otimes A_\theta , \C)\) be the class of the 
cycle described in Lemma \ref{def:Delta}. 
\end{definition}

\section{Pairs of transverse Kronecker flows}

	The \emph{Kronecker flow} on the \(2\)-torus \(\T^{2}\) for angle \(\theta\) is given by
	 the $\R$-action on $\T^{2} = \R^{2} / \Z^{2} $ defined by
	\[
		\beta_{t} \mat{ x \\ y } = \mat{ x+t\theta \\ y+ t }.
	\]
	Here, we chose translation in the second and scaled translation in the first component in order to make it compatible with the M{\"o}bius transform, see Equation \eqref{eq:Mobius} and Lemma \ref{lem:2-tori-isos}.
 Similar to the case of irrational rotation, the corresponding transformation groupoid
 ${\B_{\theta}} := \T^{2}\rtimes_{\theta} \R$ is defined as:
	\[
		\gpdcomp{
			\mat{ x \\ y }
		}{
			\left(\mat{ x \\ y },\,t\right)
		}{
			\mat{x-t\theta \\ y-t }
		}{
			\left(\mat{ x-t\theta \\ y-t },\,s\right)
		}{
			\mat{ (x-t\theta)-s\theta \\ (y-t)-s }
		}{
			\left(\mat{ x \\ y },\,t+s\right)
		}{
			\mat{ x-(t+s)\theta \\ y-(t+s)}
		}{
			-5
		}
	\]
	In particular, $\left(\mat{ x \\ y },\,t\right)^{-1} = \left(\mat{ x-t\theta \\ y-t },\,-t\right)$. We denote the momentum maps of $\mathcal{B}_{\theta}$ by $s_{\theta}$ and $r_{\theta}$.
Orbits of the Kronecker flow are lines $\vec{x} + t \mat{ \theta \\ 1 }$ in the 2-torus \(\T^{2}\). 
If
	\begin{equation}\label{eq:def:x-axis}
		{X}
		:=
		\left\{
			t \mat{ 1 \\ 0 }
			\,\mid\, t\in \mathbb{R}
		\right\}
		\subseteq
		\mathbb{T}^2
		=
		(\mathcal{B}_{\theta})^{(0)}
	\end{equation}
	denotes the $x$-axis, then the associated reduction groupoid,
	\[
		\mathcal{R}_{\theta}
	 := 
		s_{\theta}\inv ({X})\cap r_{\theta}\inv ({X}) 
		\subseteq \mathcal{B}_{\theta},
	\]
	turns out to be isomorphic to $\mathcal{A}_{\theta}$: an element $(\mat{ x \\ y },\,s)$ is in $\mathcal{R}_{\theta}$ if and only if $[y]=[0]$ and $s\in \mathbb{Z}$, and the map
		\begin{equation}\label{eq:Rth=Ath}
			\begin{tikzcd}
				\mathcal{R}_{\theta} \ar[r] &\mathbb{T}\times \mathbb{Z}\\[-20pt]
				\left(\mat{ x \\ y },\,s\right) \ar[r, mapsto]
				&
				\left([x],\,s\right)
			\end{tikzcd}
		\end{equation}
	is a groupoid isomorphism between $\mathcal{R}_{\theta}$ and $\mathcal{A}_{\theta}$.
	
	In particular, since ${X}$ is closed and meets every orbit, and since the restriction of $\mathcal{B}_{\theta}$'s range and source maps to $s_{\theta}\inv ({X})$ and to $r_{\theta}\inv ({X})$ are open maps onto their image, Example 2.7 in \cite{MRW:Grpd} 
	implies that we have an equivalence $\mathcal{X}_{\theta}$ of groupoids, 
		\[
			\mathcal{X}_{\theta}
			\colon
			\qquad
			\mathcal{B}_{\theta}
			\curvearrowright
			s_{\theta}\inv ({X})	
			\curvearrowleft
			\mathcal{A}_{\theta}.
		\]
	
Instead of reducing $\mathcal{B}_{\theta}$ to its $x$-axis, we could have reduced to a line $t\mat{ q \\ -p }$ of slope $\tfrac{-p}{q}$ for $p,q$ relatively prime, in which case we would have gotten an equivalence between $\mathcal{B}_{\theta}$ and $\mathcal{A}_{M(\theta)}$ where
\begin{equation}\label{eq:Mobius}
	 M(\theta)=\tfrac{ m\theta + n }{p \theta + q} \text{ for } M=\mat{ m & n \\ p & q }\in \S\module_{2}(\mathbb{Z})
\end{equation}
is the M{\"o}bius transform of $\theta$. 
An alternative approach is to change the slope on the foliated torus instead of the rotational angle on the circle, using the following:

\begin{lemma}\label{lem:2-tori-isos}
	For any $M=\mat{m&n\\p&q}$ in $\G\module_{2} (\mathbb{Z})$, the transformation groupoids $\mathcal{B}_{\theta}$ and $\mathcal{B}_{M(\theta)}$ are isomorphic via 
	\begin{align*}
		\varphi_{\theta}^{M}\colon\quad
		\mathcal{B}_{\theta}
		&\longrightarrow 
		\mathcal{B}_{M(\theta)} 
		\\
		\left(\mat{x\\y},\, t\right)
		&\longmapsto
		\left(M\mat{x\\y},\, t(p\theta+q)\right)
	\end{align*}
\end{lemma}

Note that $\varphi_{M(\theta)}^{N}\circ \varphi_{\theta}^{M} = \varphi_{\theta}^{NM}$ for $N$ another such matrix and $\varphi_{\theta}^{\mathbbm{1}_{2}} = \mathrm{id}_{\mathcal{B}_{\theta}}.$ Further, even though $M(\theta)= (-M)(\theta)$, we should note that $\varphi_{\theta}^{M} \neq \varphi_{\theta}^{-M}.$

\begin{definition}
	Let $\mathcal{X}_{\theta}^M$ be the $\mathcal{B}_{M(\theta)}-\mathcal{A}_{\theta}$ equivalence constructed out of $\mathcal{X}_{\theta}$ via $\varphi_{\theta}^{M}$.
\end{definition}

Given two matrices $M,N\in\G\module_{2} (\mathbb{Z})$, then $\mathcal{X}_{\theta}^{M} \times \mathcal{X}_{\theta}^{N}$ is a groupoid equivalence between $\mathcal{B}_{ M(\theta)}\times \mathcal{B}_{ N(\theta)}$ and $\mathcal{A}_{\theta}\times \mathcal{A}_{\theta}$. Moreover, if $ M(\theta)\neq N(\theta)$, the diagonal 
\[
	D_{M,N}^{}:= \left\{ [x,y,x,y] \,\vert\, [x,y]\in\mathbb{T}^2\right\}
	\subseteq \mathbb{T}^{2}\times \mathbb{T}^{2} = (\mathcal{B}_{ M(\theta)}\times\mathcal{B}_{ N(\theta)})^{(0)}
\]
 meets every orbit. Hence, $\mathcal{B}_{ M(\theta)}\times \mathcal{B}_{ N(\theta)}$ is equivalent to the reduction groupoid 
	$$\mathcal{D}_{M,N}^{}:=(\mathcal{B}_{ M(\theta)}\times\mathcal{B}_{ N(\theta)})^{D_{M,N}^{}}_{D_{M,N}^{}}$$
via $r_{\theta}\inv (D_{M,N}^{})$, and all in all we have the following chain of equivalences:	
	\[
		\mathcal{D}_{M,N}^{}
		\,\curvearrowright\,
		r_{\theta}\inv (D_{M,N}^{})
		\,\curvearrowleft\,
		\mathcal{B}_{ M(\theta)}\times\mathcal{B}_{ N(\theta)}
		\,\curvearrowright\,
		\mathcal{X}_{\theta}^{M} \times \mathcal{X}_{\theta}^{N}
		\,\curvearrowleft\,
		\mathcal{A}_{\theta}\times \mathcal{A}_{\theta}
		.
	\]
	Thus, we can construct a Morita equivalence from the C*-algebra of $\mathcal{D}_{M,N}^{}$ to $A_{\theta}\otimes A_{\theta}$. It will turn out that $\mathcal{D}_{M,N}^{}$ is an {\'e}tale groupoid with compact unit space, so its C*-algebra is unital, and the Morita equivalence is actually a right f.g.p.\ module over $A_{\theta}\otimes A_{\theta}$, \emph{i.e.}\ corresponds to a $\K$-theory class.
	
	While this description of the $\K$-theory class is nice and geometric, we will try to find an easier one. To this end, consider the following diagram:
		\begin{equation}
		\begin{tikzcd}[row sep=large, column sep=huge]
			(\mathcal{D}_{M,N}^{})^{(0)}=D_{M,N}^{}
			\ar[r, phantom, "\subseteq"]
			&[-40pt] (\mathcal{B}_{ M(\theta)}\times \mathcal{B}_{ N(\theta)})^{(0)} 
			\ar[r, phantom, "\subseteq"]
			&[-40pt] \mathcal{B}_{ M(\theta)}\times \mathcal{B}_{ N(\theta)}
					\ar[r, leftrightarrow, squiggly, "\mathcal{X}_{\theta}^{M} \times \mathcal{X}_{\theta}^{N}"]
			&	\mathcal{A}_{ \theta}\times \mathcal{A}_{ \theta}
			\\
			? \ar[u, dashed]	\ar[r, phantom, "\subseteq"]
			&
			(\mathcal{B}_{\theta}\times \mathcal{B}_{ \theta})^{(0)}
			\ar[r, phantom, "\subseteq"] 
			\ar[ur, phantom, "\circlearrowleft"']
						 \ar[u]
			& \mathcal{B}_{ \theta}\times \mathcal{B}_{ \theta}
				\ar[r, leftrightarrow, squiggly, "\mathcal{X}_{\theta} \times \mathcal{X}_{\theta}"]
				\ar[u, "\varphi_{\theta}^{M}\times \varphi_{\theta}^{N}"', "\cong"]
				\ar[ur, phantom, "\circlearrowleft"']
			&	\mathcal{A}_{ \theta}\times \mathcal{A}_{ \theta}		
			\ar[u, equal]				
		\end{tikzcd}\label{diag:F}
		\end{equation}
	
	The right-hand square of the diagram commutes since, by definition, the map $\varphi_{\theta}^{M}\times \varphi_{\theta}^{N}$ turns the equivalence $\mathcal{X}_{\theta}\times\mathcal{X}_{\theta}$ into the equivalences $\mathcal{X}_{\theta}^{M}\times\mathcal{X}_{\theta}^{N}$. The middle square commutes since $\varphi_{\theta}^{M}\times \varphi_{\theta}^{N}$ is a groupoid isomorphism, \emph{i.e.}\ it maps unit space to unit space. We want to fill in the bottom left to make the left-hand square commute as well. In other words, the question mark represents the preimage of $D_{M,N}^{}$ under $\varphi_{\theta}^{M}\times \varphi_{\theta}^{N}$, which we compute to be
	\begin{align}\label{eq:def-transversal-Fg}
		\left( \varphi_{M\inv}^{M(\theta)}\times \varphi_{N\inv}^{N(\theta)} \right) (D_{M,N}^{})
		&=
		\left\{
			\left(M\inv \mat{x\\y}, 0, N\inv \mat{x\\y}, 0\right)
			\,\vert\,
			\mat{x\\y}\in\mathbb{T}^2
		\right\}
		.
	\end{align}
	This justifies denoting this subset of $(\mathcal{B}_{\theta}\times \mathcal{B}_{ \theta})^{(0)}$ by $F_{g}^{}$ for $g:=N\inv M$. 
	As far as $\K$-theory is concerned, the f.g.p.\ $A_{\theta} \otimes A_{\theta}$-module constructed out of the bottom row of Diagram~\ref{diag:F},
	\[
		\mathcal{F}_{g}^{}
		:=
		r_{\theta}\inv (F_{g}^{}) \cap s_{\theta}\inv (F_{g}^{})
		\act
		r_{\theta}\inv (F_{g}^{}) 
		\curvearrowleft
		\mathcal{B}_{\theta}\times \mathcal{B}_{\theta}
		\act
		\mathcal{X}_{\theta}\times \mathcal{X}_{\theta}
		\curvearrowleft
		\mathcal{A}_{\theta} \times \mathcal{A}_{\theta}
		=:
		\mathcal{A}
		,
	\]
	is the same as the module constructed from the top row,
	\[
		\mathcal{D}_{M,N}^{}
		=
		r_{\theta}\inv (D_{M,N}^{}) \cap s_{\theta}\inv (D_{M,N}^{})
		\act
		r_{\theta}\inv (D_{M,N}^{})
		\curvearrowleft
		\mathcal{B}_{M(\theta)}\times \mathcal{B}_{N(\theta)}
		\act
		\mathcal{X}_{\theta}^{M} \times \mathcal{X}_{\theta}^{N} 
		\curvearrowleft
		\mathcal{A}
		,
	\]
	by commutativity of the diagram, and since the induced C*-isomorphism between the C*-algebras of $\mathcal{D}_{M,N}^{}$ and $\mathcal{F}_{g}^{}$ is unital. The clear advantage of considering $\mathcal{F}_{g}^{}$ instead of $\mathcal{D}_{M,N}^{}$ is that we only have to deal with the matrix $g= N\inv M$, and not with all $8$ entries of $M$ and $N$.
%
		The inequality $M(\theta)\neq N(\theta)$ (\emph{i.e.} $g(\theta)\neq \theta$), which we needed to construct $\mathcal{D}_{M,N}$, can be rephrased to 
		\begin{equation}\label{eq:def-mu}
			\mu (g) :=({a} \theta + {b}) - ({c} \theta + {d})\theta \neq 0
			\quad\text{ where }
			g = \mat{{a} & {b} \\ {c} & {d}}.
		\end{equation}
		In yet again other words: Either ${a}\neq {d}$ or ${b} \neq{c} \theta^2$.
		
	We can construct the equivalence between $\mathcal{F}_{g}^{}$ and $\A =	\mathcal{A}_{\theta} \times \mathcal{A}_{\theta}$ using $\mathcal{Y}_{g}^{}:=r_{\theta}\inv (F_{g}^{}) $ and $\mathcal{X}:=\mathcal{X}_{\theta} \times \mathcal{X}_{\theta}$ as
	\[
		\mathcal{F}_{g}^{}
		\,\curvearrowright\,
			\mathcal{Y}_{g}^{}
			\ast_{\B}
			\mathcal{X}
		\,\curvearrowleft\,
		\A
		,
	\]
	\emph{c.f.}\ Proposition~\ref{prop:gp-equiv-Zg} for the details in the case where $g$ is upper triangular. This equips $C_{c}(\mathcal{Y}_{g}^{}
			\ast_{\B}
			\mathcal{X})$ with a $C_c(\mathcal{F}_{g}^{})-C_c(\A)$ pre-imprimitivity bimodule structure, 
	which can be completed to a $C^{*}(\mathcal{F}_{g}^{})-C^{*}(\A)$ Morita equivalence bimodule we call $\mathsf{Z}_{g}$. 
	We will prove that \(C^{*}(\mathcal{F}_{g}^{})\) is unital (see Lemma~\ref{lem:mcFg-is-transf-gpd} below), and so since its identity element acts by a compact operator on the bimodule, 
	it is finitely generated projective as a right \(C^{*}(\A)\)-module (see Corollary \ref{cor:Zg-fgp} below), \emph{i.e.}\ $\iota^\ast (\mathsf{Z}_{g})$ defines a class in $\KK_0(\C, C^{*}(\A))$, where $\iota\colon \mathbb{C} \to C^{*}(\mathcal{F}_{g}^{})$ is the unique unital map.
	
	\begin{definition}\label{def:Lg}
	We let 
	\[[\module_{g}] := \iota^\ast (\mathsf{Z}_{g}) \in \KK_0(\C, C^{*}(\A)) = \KK_0(\C, A_{\theta} \otimes A_{\theta})\]
	be the class of the finitely generated projective right \(C^{*}(\A)\)-module 
	constructed from any \(g\in\G\module_{2} (\Z )\) satisfying Equation (\ref{eq:def-mu}).
	\end{definition}

	We end this section with a good description of $\mathcal{X}_{\theta} = s_{\theta}\inv ({X})$ 
	suitable for later computations. 
		
			\begin{lemma}\label{lem:def-Xth}
				Ket $X$ be the $x$-axis in $\mathbb{T}^2 = (\mathcal{B}_{\theta})^{(0)}$ as in Equation \eqref{eq:def:x-axis}. If we use the bijection
					\[
					\begin{tikzcd}
						s_{\theta}\inv ({X})
						\ar[r]
						&
						\mathbb{T}\times\mathbb{R}
						\\[-15pt]
						( \mat{ x \\ y },\, s )
						\ar[r, mapsto]
						&
						\left( [x-s\theta],\, s \right)
						\\[-20pt]
						(\mat{ x +s\theta \\ s },\, s )
						&
						\ar[l, mapsto]
						( [x] ,\, s )
					\end{tikzcd}
					\]
				to identify $\mathcal{X}_{\theta}$ with $\mathbb{T}\times\mathbb{R}$, then $\mathcal{X}_{\theta}$ has the following actions by $\mathcal{B}_{\theta}$ and $\mathcal{A}_{\theta}$:
					\[
					\begin{tikzcd}
						\mathcal{B}_{\theta} \curvearrowright \mathcal{X}_{\theta}:
						&
						\left( \mat{ x +(s+r)\theta \\ s+r } ,\, r\right).( [x] ,\,s )
						=
						( [x] ,\, r+s )
						\\[-15pt]
						\mathcal{X}_{\theta} \curvearrowleft \mathcal{A}_{\theta}:
						&
						( [x] ,\,s ).([x] ,\,k)
						=
						( [ x - k \theta] ,\,s+k)
					\end{tikzcd}
					\]
				where we used the map from Equation~\eqref{eq:Rth=Ath} to identify $s_{\theta}\inv ({X})\cap r_{\theta}\inv ({X}) \cong \mathcal{A}_{\theta}$.
			\end{lemma}
	The proof is straight forward.
	Let us next describe $\mathcal{F}_{g}^{}$:
	one checks
	\[
	r_{\theta}\inv ({F}_{g}^{})
			=
			\{
				( \mat{ x \\ y },t_{1},g\mat{ x \\ y }, t_{2})
				\,\vert\,
				\mat{ x \\ y }\in\mathbb{T}^2, t_{1},t_{2}\in\mathbb{R}
			\}
	\]
	and thus 
	\begin{align*}
		\mathcal{F}_{g}^{}
				&=
				\left\{
					\left( \mat{ x \\ y }, \tfrac{k+l\theta}{\mu(g)},g \mat{ x \\ y },\tfrac{k({c}\theta+{ d })+l({a}\theta+{b})}{\mu(g)}\right)
					\,\vert\,
					k,l\in\mathbb{Z}
				\right\},
	\end{align*}
	where ${\mu(g)} $ is as in Equation~\eqref{eq:def-mu}. In the following, we will write $\mat{ x \\ y } + t \slope := \mat{ x +t\theta\\ y+t }$.
	
	\begin{lemma}\label{lem:mcFg-is-transf-gpd}
		The groupoid $\mathcal{F}_{g}^{}$ is isomorphic to the transformation groupoid of the following $\mathbb{Z}^2$ action on $\mathbb{T}^2$:
		\[
			\begin{tikzcd}
				\mathbb{T}^2 \curvearrowleft \mathbb{Z}^2:
				&
				\mat{ x \\ y }.\left( k,l \right)
				=
				\mat{ x \\ y } + \tfrac{k+l\theta}{\mu(g)} \slope
			\end{tikzcd}
		\]		
		In particular, $\mathcal{F}_{g}^{}$ is {\'e}tale with compact unit space and its $C^*$-algebra $C^{*}(\mathcal{F}_{g}^{})$ is therefore unital.
	\end{lemma}
	
	\begin{proof}
		The map
			\begin{align}
			\begin{split}\label{eq:Fg=Z2T2}
				\mathcal{F}_{g}=r_{\theta}\inv ({F}_{g}^{})\cap s\inv ({F}_{g}^{})
				&\longrightarrow
				\mathbb{T}^{2}\rtimes\mathbb{Z}^{2}
				\\
				\left(\mat{ x \\ y }, \tfrac{k+l\theta}{\mu(g)}, g \mat{ x \\ y },\tfrac{k({c}\theta+{ d })+l({a}\theta+{b})}{\mu(g)}\right)
				&\longmapsto
				( \mat{ x \\ y },k, l)
			\end{split}
			\end{align}
		is an isomorphism of groupoids, where the right-hand side is the alleged transformation groupoid. In particular, the unit space of $\mathcal{F}_{g}^{}$ is $\mathbb{T}^2$ and hence compact.
		
		Since $\mathbb{Z}^2$ is discrete, the transformation groupoid is {\'e}tale, 
		and so its unit space is clopen. Its characteristic function is hence a continuous, compactly supported function on $\mathcal{F}_{g}$ and serves as unit in $C^*(\mathcal{F}_{g})$.
	\end{proof}
	
	\begin{corollary}\label{cor:Zg-fgp}
		The bimodule $\mathsf{Z}_{g}$ is finitely generated projective as a right $C^{*}(\mathcal{A})$-module, so $\module_g=\iota^\ast (\mathsf{Z}_{g})$ defines a class in $\KK (\mathbb{C}, C^{*}(\mathcal{A}))$ where $\iota\colon \mathbb{C} \to C^{*}(\mathcal{F}_{g}^{})$ is the unique unital map.
	\end{corollary}
	
	\begin{proof}
		We have seen that $C^* (\mathcal{F}_{g}^{})$, which acts by compact operators on the Morita bimodule $\mathsf{Z}_{g}$, is unital. Therefore, the operator $\mathrm{id}_{\mathsf{Z}_{g}}$ is $C^*(\mathcal{A})$-compact, which means $\mathsf{Z}_{g}$ is f.g.p.\ by \cite{GVF:NCG}, Proposition 3.9.
	\end{proof}

	\begin{lemma}\label{lem:def-Yg} 
		If we use the bijection
		\begin{align*}
			\mathcal{Y}_{g} = r_{\theta}\inv ({F}_{g}^{})
			&\longrightarrow
			\mathbb{R}^{2}\times\mathbb{T}^{2}\\
			(\mat{ x \\ y }, t_{1}, g\mat{ x \\ y }, t_{2})
			&\longmapsto
			( t_{1}, t_{2}, \mat{ x \\ y })
		\end{align*}
		to identify $\mathcal{Y}_{g}^{} \cong \mathbb{R}^{2}\times\mathbb{T}^{2}$, then the right action by $\mathcal{B}:=\mathcal{B}_{\theta}\times \mathcal{B}_{\theta}$
		on an element $	(t_{1}, t_{2}, \mat{ x \\ y }) \in \mathcal{Y}_{g}^{}$ is given by
		\[
			(t_{1}, t_{2}, \mat{ x \\ y })
			.
			(\mat{ x \\ y }-t_{1}\slope,r_{1}, g\mat{ x \\ y }-t_{2}\slope, r_{2})
			=
			( t_{1}+r_{1}, t_{2}+r_{2},\mat{ x \\ y }).
		\]	
		
		If we further use the bijection in Equation~\eqref{eq:Fg=Z2T2} to identify $\mathcal{F}_{g}^{} \cong \mathbb{T}^{2}\rtimes\mathbb{Z}^{2}$, then the left action of $\mathcal{F}_{g}^{}$ on $\mathcal{Y}_{g}^{}$ is by
		\begin{gather*}
			(\mat{ x \\ y }+\tfrac{k+l\theta}{\mu(g)}\slope,k,l) 
			.
			(t_{1},t_{2},\mat{ x \\ y })
			=
			\left(
				\tfrac{k+l\theta}{\mu(g)}+t_{1},
				\tfrac{k({c}\theta+{ d })+l({a}\theta+{b})}{\mu(g)}+t_{2},
				\mat{ x \\ y }+\tfrac{k+l\theta}{\mu(g)}\slope
			\right).
		\end{gather*}
	\end{lemma}

	Elements of $\mathcal{Y}_{g}\ast\mathcal{X}$, where 
	$\mathcal{X} = \mathcal{X}_{\theta} \times \mathcal{X}_{\theta}$ as before, are given by those $(t_{1}, t_{2}, [x , y] , [v],s_{1} , [w] , s_{2} )$ in $(\mathbb{R}^2 \times\mathbb{T}^{2}) \times(\mathbb{T}\times\mathbb{R}\times\mathbb{T}\times\mathbb{R})$ which satisfy
	\begin{align*}
		&
		s_{\mathcal{Y}} (t_{1}, t_{2}, [x , y])
		=
		r_{\mathcal{X}} ( [v] ,s_{1} , [w], s_{2} )
		\\
		\iff
		&
			[x,y]-t_{1}[\theta,1]
			=
			[v,0]+s_{1}[\theta,1]
			\and
			g[x,y]-t_{2}[\theta,1]
			=		
			[w,0]+s_{2}[\theta,1].
	\end{align*}
	In other words,
	\begin{align*}
			\mat{ x \\ y }
			=
			\mat{ v \\ 0 }+(s_{1}+t_{1})\slope
			=		
			g\inv\left( \mat{ w \\ 0 }+(s_{2}+t_{2})\slope\right).
	\end{align*}
	
	Now, in the balanced $\mathcal{Y}_{g} \ast_{\mathcal{B}} \mathcal{X}$, we identify the following elements of $\mathcal{Y}_{g}\ast\mathcal{X}$:
	\[
		(t_{1}, t_{2}, [x , y], [v] ,s_{1} , [w] , s_{2})	
		\sim
		(t_{1}+t'_{1}, t_{2}+t'_{2},[x , y],[v],s_{1}-t'_{1},[w],s_{2}-t'_{2})
	\]
	for any $t'_{1},t'_{2}\in \mathbb{R}$. We conclude:
	
	\begin{lemma}\label{lem:Yg-ast-X=Zg}
		If we let
		\begin{align*}
			\mathcal{Z}_{g}
			:=&
			\left\{
				(r_{1},r_{2},\mat{ v \\ w })
				\in\mathbb{R}^2 \times \mathbb{T}^2
				\,\vert\,
				g\left(\mat{ v \\ 0 }+r_{1}\slope\right)
				=		
				\mat{ w \\ 0 }+r_{2}\slope
			\right\}
			,
		\end{align*}
		\label{rmk:rough-Why-pa=mc}
		then the following are mutually inverse bijections:
		\[
		\begin{tikzcd}
			\mathcal{Y}_{g}\ast_{\mathcal{B} }\mathcal{X}
			\ar[r]
			&[-25pt]
			\mathcal{Z}_{g}
			\\[-15pt]
			\left[t_{1}, t_{2},\left[x , y\right],\left[v\right],s_{1},\left[w\right],s_{2}\right]
			\ar[r, mapsto]
			&
			(t_{1}+s_{1},t_{2}+s_{2},\mat{ v \\ w })
			\\[-15pt]
			\left[r_{1}, r_{2},\left[v,0\right]+r_{1}\left(\theta,1\right),\left[v\right],0,\left[w\right],0\right]
			&
			\ar[l, mapsto]
			(r_{1},r_{2},\mat{ v \\ w })
		\end{tikzcd}
		\]
	\end{lemma}
	
	We could now describe the groupoid equivalence structure on $\mathcal{Z}_{g}$ which it inherits from 
	$\mathcal{Y}_{g}\ast_{\mathcal{B} }\mathcal{X}$ via the above bijection. Then Theorem~2.8 in~\cite{MRW:Grpd} allows us to complete $C_c (\mathcal{Z}_{g})$ to the Morita equivalence we called $\mathsf{Z}_{g}$, which would yield the formulas for the $\K$-theory class of $\module_{g} = \iota^* (\mathsf{Z}_{g})$. However, we have no need for them in all generality, so we will postpone this until we have restricted to a certain class of matrices $g$.

\section{The b-twist}

Connes' quantized Dolbeault cycle, consisting of \(L^2(\T^2)\) with an appropriate pair of 
commuting representations of \(A_\theta\), and the Dirac-Dolbeault operator 
\[
\begin{bmatrix} 0 & \frac{\partial}{\partial \Theta_1} - i\frac{\partial}{\partial \Theta_2}\\
 \frac{\partial}{\partial \Theta_1} + i \frac{\partial}{\partial \Theta_2}& 0 \end{bmatrix},\]
gives the class \(\Delta_\theta \in \KK_0(A_\theta\otimes A_\theta, \C)\), which is the co-unit of the 
duality we are going to establish.  By 
the general mechanics of \(\KK\), the class \(\Delta_\theta\) determines a map 
\[ \Delta_\theta \cup \text{\textvisiblespace}\,\colon \KK(\C, A_\theta\otimes A_\theta) \to \KK_0(A_\theta, A_\theta),
\; f\mapsto (f\otimes 1_{A_\theta} ) \otimes_{A_\theta^3} (\Delta_\theta \otimes 1_{A_\theta}),
\]
and the first zig-zag equation asserts that, if \(f\in \KK_0(\C, A_\theta \otimes A_\theta)\) is the 
unit for a duality with co-unit \(\Delta_\theta\), then 
\( \Delta_\theta \cup f = 1_{A_\theta}.\)

We are going to show in this article that 
\begin{equation}
\label{equation:yay!}
 \Delta_\theta \cup [\module_g] = \twist_g
 \end{equation}
where \( g= \begin{bsmallmatrix} 1 & b \\ 0 & 1\end{bsmallmatrix}\) is 
upper triangular and non-trivial, \([\module_g]\in 
\KK_0(\C, A_\theta\otimes A_\theta)\) the class of the finitely generated projective 
\(A_\theta\otimes A_\theta\)-module constructed in the last section from the transversals and 
\(g\), and \(\twist_g\) is a certain
 \emph{invertible} in 
\(\KK_0(A_\theta, A_\theta)\) which we describe explicitly first.

Let \(b\in \Z\) be any integer. Equip $C_c (\mathbb{T}\times\mathbb{R})$ with the following structure:
\label{page:Hb-module-structure}
	\begin{alignat*}{2}
			\phi,\psi\in C_c (\mathbb{T}\times\mathbb{R}):
			&&\quad
			\inner{\phi_{1} }{\phi_{2} }_{{C(\mathbb{T})}} ([x])
			&= 
			\int_{\mathbb{R}}
				\overline{\phi_{1}}\phi_{2}
				\left(
					\left[
						x
					\right],
					r
				\right)
			\,\mathrm{ d } r
					,
			\\
			\mathbb{Z} \curvearrowright C_c (\mathbb{T}\times\mathbb{R}):	
			&&\quad
			(l\bullet \phi)
			(\left[x\right],r)
			&= 
			\phi
			\left(
				\left[ 
					x - l \theta
				\right],
				r - l
			\right),
			\\
			{C(\mathbb{T})} \curvearrowright C_c (\mathbb{T}\times\mathbb{R}):
			&&\quad
			(f \bullet \phi ) (\left[x\right],r) 
			&=
			f
				\left(\left[ x + r { b } \right]\right)
					\phi ([x],r),
			\\
			C_c (\mathbb{T}\times\mathbb{R}) \curvearrowleft {C(\mathbb{T})}:
			&&\quad
			(\phi\bullet f) ([x],r) 
			&= 
			\phi(\left[x\right],r)
			f(\left[ x\right])
			.
		\end{alignat*}

	Let $H_{ b }^{\pm}$ be the completion of \(C_c(\T\times \R)\)
	 respect to the pre-inner product given above and set {$H_{ b } := H_{ b }^{+} \oplus H_{ b }^{-}$}. 
	 For $\lambda\in\mathbb{R}^{\times}$, let $d_{\lambda,+}$ be the following, well-known operator on $L^2 (\mathbb{R} )$
	with domain the Schwartz functions $\mathcal{S}(\mathbb{R})$ on \(\mathbb{R}\):
	\[
		d_{ \lambda,+ }
		 :=
		 \lambda
		 \mathsf{M}
 		 +\tfrac{\partial\;}{\partial r}
		\quad\text{with adjoint}\quad
		d_{ \lambda ,-}
		 := 
		 	\lambda
			\mathsf{M}
		 	-
		 	\tfrac{\partial\;}{\partial r}
	\]
Here, $\mathsf{M}$ is the operator that multiplies by the input of the $\mathbb{R}$-component.

\begin{theorem}\label{thm:HdH-is-cycle}
	 If \(\Z\) acts by rotation on \(\T\), and  $\lambda\in\mathbb{R}^{\times}$, then 
	 the pair $({H_{ b }}, \mathrm{id}_{C(\mathbb{T})}\otimes d_{ \lambda })$ with 
	 \[d_{ \lambda } := \mat{0 & d_{ \lambda,-}\\d_{ \lambda,+ } &0}\]
	  is a cycle in $\Psi^\mathbb{Z} \left( C(\mathbb{T}),C(\mathbb{T})\right)$.
\end{theorem}
Recall that the symbol $\mathrm{id}_{C(\mathbb{T})}\otimes d_{ \lambda }$ denotes the closure of the operator $\mathrm{id}_{C(\mathbb{T})}\odot d_{ \lambda }$, which has dense domain
$C(\mathbb{T})\odot \mathcal{S}(\mathbb{R})$. 
The proof of the theorem is on page \pageref{pf:thm:HdH-is-cycle}.
\begin{lemma}\label{lem:d_H-is-nice}
	The operator $ \mathrm{id}_{C(\mathbb{T})}\otimes d_{ \lambda }$ is odd, self-adjoint, regular, and has compact resolvent.
\end{lemma}

\begin{proof}
	By construction, $ \mathrm{id}_{C(\mathbb{T})}\odot d_{ \lambda }$ is odd and symmetric. We compute
	\begin{align}
		d_{ \lambda }^{2}
		&=
		{}
		\Mat{
			\lambda^{2}
			\mathsf{M}^{2}
			-
			\tfrac{\partial^{2}\;}{\partial r^{2}}
			-\lambda
			&
			0
		\\
			0
			&
			\lambda^{2}
			\mathsf{M}^{2}
			-
			\tfrac{\partial^{2}\;}{\partial r^{2}}
			+\lambda
		}	
		,
		\label{eq:dH2}
	\end{align}
	
	Consider the $ L^{2} $-normalized functions
	\begin{align*}
		\psi_0 (r) = \abs{\lambda}^{\frac{1}{2}}\pi^{\frac{1}{4}} \mathsf{e}^{-\abs{\lambda}\frac{r^{2}}{2}}
		\and
		\psi_l = (2l\abs{\lambda})^{-\frac{1}{2}} \cdot (\abs{\lambda}\mathsf{M} - \tfrac{\partial\;}{\partial r}) \psi_{l-1}.
	\end{align*}
	Note that $\psi_{0}$ is a Schwartz function, \emph{i.e.}\ in the domain of $d_{\lambda,\pm}$, and therefore so are all $\psi_{l}$. Moreover, they span a dense subspace of $ L^{2} (\mathbb{R} )$ (\cite{Roe:Elliptic}, Proposition~9.8) and they are eigenfunctions of $\lambda^{2}\mathsf{M}^{2}-\tfrac{\partial^{2}\;}{\partial r^{2}}$ (\cite{Roe:Elliptic}, Lemma~9.6) with corresponding set of eigenvalues
	\[
		\left\{
			(2l+1)\abs{\lambda} \, \colon \, l = 1, 2, \ldots 
		\right\}.
	\]
	We conclude that the operator $d_{ \lambda }^{2} + 1$ has the eigenfuctions
	$
		\psi_{l} \oplus 0
	$ and $
		0 \oplus \psi_{l}
		.$
	Thus, the orthonormal basis $\left\{ \psi_l\oplus 0, 0\oplus \psi_l \,:\, l\in\mathbb{N}_0\right\}$ of $ L^{2} (\mathbb{R} )\oplus L^{2} (\mathbb{R} )$ is in the range of $d_{\lambda}^{2} + 1$, which proves that the range of $\left(\mathrm{id}_{C(\mathbb{T})}\odot d_{ \lambda }\right)^{2}+1$ is dense, so $\mathrm{id}_{C(\mathbb{T})}\otimes d_{ \lambda }$ is self-adjoint and regular. Moreover, $d_{ \lambda }^{2} + 1$ is diagonalizable and its eigenvalues $(2l+1)\abs{\lambda}$ tend to infinity. This shows that $d_{ \lambda }^{2} + 1$ has compact inverse, and that $(( \mathrm{id}_{C(\mathbb{T})}\otimes d_{\lambda})^{2} +1)\inv=1_{C(\mathbb{T})}\otimes \left(d_{ \lambda }^{2} + 1\right)\inv$ is compact as tensor product of compact operators. Thus, $\mathrm{id}_{C(\mathbb{T})}\otimes d_{\lambda}$ has compact resolvent.
\end{proof}

\begin{lemma}
	\label{lem:d_H-is-equivariant}
	The operator $ \mathrm{id}_{C(\mathbb{T})}\otimes d_{ \lambda }$ is 
	almost equivariant, \emph{i.e.}\ for any $n\in\mathbb{Z} $, the operator {$( \mathrm{id}_{C(\mathbb{T})}\otimes d_{ \lambda })-\mathsf{Ad}_n ( \mathrm{id}_{C(\mathbb{T})}\otimes d_{ \lambda })$} on $\dom ( \mathrm{id}_{C(\mathbb{T})}\otimes d_{ \lambda })$ extends to an adjointable operator.
\end{lemma}

\begin{proof}
For $\phi\in C_c (\mathbb{R}\times\mathbb{T} )\subseteq {H_{ b }}^{\pm}$, we have
$
		l\bullet\tfrac{\partial \phi}{\partial r}
=
		\tfrac{\partial\;}{\partial r} (l\bullet\phi),
$	and
	\begin{align*}
		l\bullet\left( \mathsf{M} (-l)\bullet\phi \right) ([x],r)
		&=
		\left( \mathsf{M} (-l).\phi \right) \left([ x - l \theta],r - l \right)
		\notag
		\\
		&=
		\left(r - l\right) \cdot \phi ([x],r)
		.
	\end{align*}
We conclude that
	\begin{align*}
		(\mathrm{id}_{C(\mathbb{T})}\otimes d_{ \lambda })-\Ad_l (\mathrm{id}_{C(\mathbb{T})}\otimes d_{ \lambda })
		=
		\Mat{0 & \lambda\mathsf{M}
		 		 \\\lambda\mathsf{M} &0}
		-
		\Mat{0 & \Ad_l (\lambda\mathsf{M})
		 		\\\Ad_l (\lambda\mathsf{M})
		 		 &0}
		&=
		\Mat{
			0
			&
			\lambda l
		\\
			\lambda l
			&
			0
			}
	\end{align*}
	on the dense subspace $C_{c}(\mathbb{R}\times\mathbb{T})$.
	Thus, for any fixed $l\in\mathbb{Z} $, the operator $( \mathrm{id}_{C(\mathbb{T})}\otimes d_{ \lambda })-\Ad_l ( \mathrm{id}_{C(\mathbb{T})}\otimes d_{ \lambda })$ is adjointable.
\end{proof}

\begin{lemma}
\label{lem:bounded-commutators}
	The subalgebra $\left\{f \in {C(\mathbb{T}_{})} : \comm{ \mathrm{id}_{C(\mathbb{T})}\otimes d_{ \lambda }}{f} \in \mathcal{L} ({H_{ b }})\right\}$ is dense in ${C(\mathbb{T}_{})}$.
\end{lemma}

	\begin{proof}
	We need to figure out for which $f$ the operators
\begin{align*}
		{H_{ b }}^{\pm}
		\supseteq
		C(\mathbb{T})\otimes \mathcal{S}(\mathbb{R})
		\ni
		\phi
		&\longmapsto
		\left(
			\lambda
			\mathsf{M}
			\pm
			\tfrac{\partial\;}{\partial r}			
		\right)
		(f\bullet \phi)
		-
		f\bullet 
		\left(
			\lambda
			\mathsf{M}
			\pm
			\tfrac{\partial\;}{\partial r}
		\right) (\phi)	
	\end{align*}
	are bounded. First, note that 
	\[
		\mathsf{M}
		(f\bullet \phi)
		=
		f\bullet 
		\left(
			\mathsf{M} \phi
		\right),
	\]
	as $f\bullet$ does not change the $\mathbb{R}$-coordinate.
	 Secondly, define $$f_{b}([x],r) := f\left(\left[ x + { b } r \right]\right),$$ so that $f\bullet \phi = f_{b}\cdot \phi$, and
	\begin{align*}
		\tfrac{\partial (f\bullet \phi)}{\partial r}
		-
		f\bullet \tfrac{\partial \phi}{\partial r}
		=
		\tfrac{\partial f_{b}}{\partial r}
		\cdot \phi.
	\end{align*}
	As long as $\tfrac{\partial f_{b}}{\partial r}$ makes sense and is bounded, this is a bounded operator of $\phi$, \emph{i.e.}\ 
	it makes $\comm{ \mathrm{id}_{C(\mathbb{T})}\otimes d_{ \lambda }}{f}$ bounded.
\end{proof}

\begin{proof}[Proof of Theorem \ref{thm:HdH-is-cycle}]\label{pf:thm:HdH-is-cycle}
	 By construction, ${H_{ b }}$ is a graded, equivariant correspondence. 
	As $ \mathrm{id}_{C(\mathbb{T})}\otimes d_{ \lambda }$ only sees the $\mathbb{R}$-component of a function's domain while the right action only sees the $\mathbb{T}$-component, we see that the two commute, which proves linearity. 
	The remaining properties that $(H_{b},\mathrm{id}_{C(\mathbb{T})}\otimes d_{ \lambda })$ has to satisfy in order to be a cycle have already been proven: Lemma~\ref{lem:d_H-is-nice} showed that the operator is self-adjoint regular with compact resolvent, Lemma~\ref{lem:d_H-is-equivariant} showed that it is almost equivariant, and Lemma~\ref{lem:bounded-commutators} showed that the subalgebra of $C(\mathbb{T})$ which commutes with the operator up to bounded operators is dense.
\end{proof}

It follows that $j(({H}_{ b }, \mathrm{id}_{C(\mathbb{T})}\otimes d_{\lambda}))=:(\mathcal{H}_{ b }, D_{\lambda})$ is a cycle in $\Psi( A_{\theta} , A_{\theta} )$, where \(j\) is the descent map on 
cycles \(\Psi^\Z\left(C(\T), C(\T)\right) \to \Psi( C(\T)\rtimes_\theta \Z, C(\T) \rtimes_\theta \Z) = 
\Psi(A_\theta, A_\theta)\). 
For reference, let us explicitly describe the structure of $\mathcal{H}_{ -b }$, which can be constructed using descent and the definition of its lift $H_{-b}$ on page~\pageref{page:Hb-module-structure}:
	\begin{lemma}\label{lem:description-of-mcHb}
	The left $\mathfrak{A}$-action on $C_{c} (\mathbb{Z}\times\mathbb{T}\times\mathbb{R})\subseteq \mathcal{H}_{-b}^{\pm}$ is given by
		\begin{align}\label{eq:left action-on-mcH}
			\bigl(	a ._{\mathcal{H}_{ -b }} \Psi \bigr) (n,[x],r)
			&=
			\sum_{m\in\mathbb{Z}} a ( [x-rb], m) \Psi (n-m,[x-m\theta], r-m ).
		\end{align}
and the right action by
\begin{align}
		(\Psi ._{\mathcal{H}_{ -b }} a) (n,[x],r)
		 	&=
		 	\sum_{m\in\mathbb{Z}}
		 	\Psi (m,[x],r)
		 	a([x-m\theta], n-m)
		 	\label{eq:right-action-on-mcH}
		 	.
\end{align}
	Its \textup(pre-\textup)inner product is given by:
		\begin{align}
		 	\inner{\Psi_{1}}{\Psi_{2}}^{\mathcal{H}_{ -b }}_{A_{\theta}} ([w],l_{2})
		 	&=
		 	\sum_{k_{1}}
		 	\inner{\Psi_{1}(k_{1})}{\Psi_{2}(l_{2}+k_{1})}^{{N}_{g}}_{C(\mathbb{T})}([w+k_{1}\theta])
		 	\notag
		 	\\
		 	&=
		 	\sum_{k_{1}}
		 		\int_{\mathbb{R}}
		 			\overline{\Psi_{1}}(k_{1},[w+k_{1}\theta],r)
		 			\Psi_{2}(l_{2}+k_{1},[w+k_{1}\theta],r)
		 	\,\mathrm{ d } r
		 	.\label{eq:inner-product-on-mcH}
	 \end{align}
	\end{lemma}

\begin{definition}
	\label{definition:btwist}
	The \emph{\(b\)-twist} $\twist_{b}$ is the element of \(\KK_0(A_\theta, A_\theta)\) represented by the descent \((\mathcal{H}_{ b }, D_{1})\) of the 
	\(\Z\)-equivariant unbounded cycle \((H_{ b }, \mathrm{id}_{C(\mathbb{T})}\otimes d_{1})\), a cycle in \(\Psi^\Z( C(\mathbb{T}) , C(\mathbb{T}) )\). 
\end{definition}

\begin{remark}\label{rmk:Hb-db-replaced-by-Hb-dlambda}
Note that since we have proved that \(d_\lambda\) defines an elliptic operator for any real 
\(\lambda \not= 0\), any two of 
 the cycles  $(H_{b},\mathrm{id}_{C(\mathbb{T})}\otimes d_{ \lambda })$ with \(\lambda \) of the 
 same sign, are homotopic to each other in the obvious way. Of course \(d_\lambda\) is not homotopic to \(d_{-\lambda}\), since their (nonzero) Fredholm indices 
 have opposite signs. 
\end{remark}

\begin{remark}

It seems likely that there is a `\(g\)-twist'  cycle 
and class 
\(\twist_g\in \KK_0(A_\theta, A_\theta)\) 
\emph{any} \(g\in \SL_2(\Z)\), an element \(\tau_g\in \KK_0(A_\theta, A_\theta)\) making \eqref{equation:yay!} true
and \(g\mapsto \twist_g\) a group homomorphism \(\SL_2(\Z)\) into the 
invertibles in \(\KK_0(A_\theta, A_\theta)\). Currently, we have only defined the cycle
 for 
upper-triangular \(g\), because the descent apparatus becomes available. That there is 
a functorial construction of \emph{classes} 
\(\tau_g\) from \(\SL_2(\Z)\) is rather easy to see; see the end of the section; the open 
question is whether or not these cycles can be defined using Dirac-Schr\"odinger 
operators. As this question is not 
immediately  material for proving duality, we leave it as a project for the future. 
\end{remark}

The duality result we are proving in this article, like all dualities known to the 
authors, uses Bott Periodicity (specifically in this case, \(\Z\)-equivariant Bott Periodicity) 
at some point in the proof. In our case, it is embedded in the proof of the  
following result.

\begin{theorem}
\label{theorem:thetwistedindextheorem}
The 
twist morphisms \(\{\tau_b\}_{b\in \Z} \in \KK_0(A_\theta, A_\theta)\) form a 
cyclic group of \(\KK\)-equivalences under composition. In particular, 
\[ \tau_{-b} = \tau_b^{-1}\in \KK_0(A_\theta, A_\theta).\]
\end{theorem}

Recall that Kasparov's bivariant category \(\RKK^\Z_*(\R;\, \cdot \,,\cdot\,)\) has objects 
 \(\Z\)-C*-algebras 
and morphisms \(A\to B\) are the elements of the abelian group 
\[ \RKK^\Z_*(\R; A, B),\]
which is the quotient of the set of cycles \((\Hilb, F)\) for \(\KK_*^\Z(C_0(\R)\otimes A, C_0(\R)\otimes B)\) for which the left and right actions of \(C_0(\R)\) on the module \(\Hilb\) are \emph{equal}. Such a cycle can be considered as a family \( ( \Hilb_t, F_t)_{t\in\mathbb{R}}\) of \(\KK_*(A, B)\)-cycles which is essentially equivariant in the sense that, for all \(t\in \R\), any integer \(l\) maps \(\Hilb_t\) to \(\Hilb_{t+l}\) and 
\begin{equation}\label{eq:ess-eq}
	(-l)\circ F_{t+l}\circ l - F_{t}
\end{equation}
is a compact operator on \(\Hilb_t\). 

Let 
\[ p^{*}_{\R}\colon \KK^\Z_*(A, B) \to \RKK^\Z_*(\R; A, B)\]
be Kasparov's inflation map, which (on cycles) associates to a cycle for \(\KK_*(A, B)\) the corresponding constant field of cycles over \(\R\). The inflation map converts analytic problems into topological problems, as we shall see shortly in connection with our own problems.

The following result follows from the Dirac-dual-Dirac method.

\begin{lemma}[see \cite{Em:bdActions}, Theorem 54] 
\(p_{\R}^{*}\) is an isomorphism for all \(A, B\). 
\end{lemma}

We will be setting \(A = B = C(\T)\) in the following, and apply the inflation map to 
the class of the equivariant cycles \( ( H_{b} , \id_{C(\T)}\otimes d_\lambda )\) discussed above. 

\begin{definition}\label{def:hat-tau-b}
The \emph{topological} \(b\)-twist \(\widehat{\tau_b} \in \RKK^\Z_0(\R; C(\T), C(\T))\) is the class of the 
bundle of *-homomorphisms 

\[ \widehat{\tau}^b_t  \colon C(\T) \to C(\T), \quad \widehat{\tau}^b_t  (f) ([x]) := f([x+bt]),\]
The family of automorphisms \(\{\widehat{\tau}^b_t \}_{t\in\mathbb{R}}\) is equivariant if 
the action by $\Z$ on $\R$ is by translation and on $C(\T)$ is by irrational rotation, since
 \(b\) is an integer.  
\end{definition}

Since the Kasparov product of two families of automorphisms in $\RKK^{\mathbb{Z}}_0$ is simply given by composition, we see that the product of $\widehat{\tau}^b$ with $\widehat{\tau}^{b'}$ is exactly $\widehat{\tau}^{b+b'}$. Clearly $\widehat{\tau}^0$ is the identity, and so we conclude that $b\mapsto \widehat{\tau}^b$ is a group homomorphism from $\Z$ to invertibles in $\RKK^{\mathbb{Z}}_{0}(\R; C(\T),C(\T))$  (under composition).

\begin{theorem}
\label{theorem:keyindexreduction}
Let \( ( H_{b}, \id_{C(\T)}\otimes d_\lambda )\) be any of the Dirac-Schr\"odinger 
cycles for \(\KK^\Z_0\left(C(\T), C(\T)\right)\) of Theorem \ref{thm:HdH-is-cycle}, with 
\(\lambda >0\). Then 
\[ p_{\R}^{*}([( H_{b}, \id_{C(\T)}\otimes d_\lambda )]) = \widehat{\tau^b}\in \RKK^\Z_0(\R; C(\T), C(\T)).\]
\end{theorem}

\begin{proof}
As explained at the beginning of this section, $p_{\R}^{*}([( H_{b}, \id_{C(\T)}\otimes d_\lambda )])$ is 
represented by the \emph{constant} bundle of cycles which consists, for each \(t\in \R\), of the Dirac-Schr\"odinger cycle. 

First, we will modify the operator
\begin{equation*}
d_\lambda  = \mat{ 0 & d_{\lambda, -}\\ d_{\lambda , +}& 0 },
	\qquad
	d_{\lambda, \pm} =  \lambda \mathsf{M} \pm  \der{\;}{r}
	,\end{equation*}
on $L^2(\R) \oplus L^2(\R)$ by changing the implicit 
reference point \(t = 0\) in the cycle; 
we do this to turn our constant
family over $\R$, which is essentially $\Z$-equivariant in the sense of Equation \eqref{eq:ess-eq}, into 
a $\Z$-equivariant family. We will then apply an argument of L{\"u}ck-Rosenberg.

If  \( U_{t} \) is a left translation unitary with \(t\in \R\) then 
\[ 
	U_{t} \circ d_{\lambda, +} \circ U_{-t} 
	=  \lambda ( \mathsf{M} -t )  +  \der{\;}{r}
	=:
	d_{\lambda, +}^{t},
\]
and a similar statement holds for \(d_{\lambda, -}\) 
and hence for \(d_\lambda \). We thus obtain an equivariant 
 family of operators \( d^{t}_\lambda \) on \(L^{2}(\R)\oplus L^{2}(\R)\), all unitary conjugates and bounded perturbations of each other since
\[
	d_\lambda  - d_\lambda ^{t}
	=
	d_\lambda  - U_{t} \circ d_\lambda  \circ U_{-t} 
	=\mat{ 0 & \lambda t \\ \lambda t & 0 }.
\]
We now tensor \(d^{t}_\lambda \) by the identity on \(C(\T)\) to obtain a family 
\[ \{ ( H_{b}, \id_{C(\T)} \otimes d_\lambda^{t})\}_{t}\]
of cycles for \(\KK^\Z_0\left(C(\T), C(\T)\right)\), in which only the operator is varying with \(t\in\R\) while the modules \(H_{b}\) stay constant. 
This describes a cycle  
that is a bounded perturbation of the constant cycle which represents \(p_{\R}^{*}[ ( H_{b}, \id_{C(\T)} \otimes d_{\lambda} )]\). In particular,
\begin{equation}\label{eq:bdd-perturbation-of-constant-family}
	p_{\R}^{*}\left( \left[ \left( H_{b}, \id_{C(\T)}\otimes d_{\lambda} \right) \right] \right)
	=
	\left[( H_{b}, \id_{C(\T)} \otimes d_{\lambda}^{t})_{t\in\R}\right]
	\;\in\;
	\RKK^\Z_0 \left(\R; C(\T), C(\T)\right)
	.
\end{equation}
and our new bundle of cycles is \(\Z\)-equivariant on the nose, as a bundle. 

We next describe a homotopy, which we will describe as a family of homotopies parameterized 
by \(t\in \R\). Fix \(t\).

The following is based on arguments of L{\"u}ck and Rosenberg in \cite{LR:Euler}. 
For $\lambda \in [1,+\infty) $, the 
spectrum of the operator 
\[d_\lambda^t :=  \begin{bmatrix} 0 & \lambda ( \mathsf{M} -t ) - \frac{\partial }{\partial r}\\ \lambda ( \mathsf{M} -t ) + \frac{\partial }{\partial r} & 0 \end{bmatrix}\] on $L^{2}(\mathbb{R})\oplus L^{2}(\mathbb{R})$ is given by 
	\[
		\left\{
			(\pm \sqrt{2l+1}) \lambda \, \colon \, l = 0, 1, 2, \ldots 
		\right\}
	,\]
and \(d_{\lambda}^t\) is orthogonally diagonalizable with eigenspaces all of multiplicity \(1\). The kernel of \(d_{\lambda}^t\) is spanned by the unit vector $\psi^t_{0,\lambda} \oplus 0$ where
\begin{equation}\label{eq:def:psi0}
	\psi^t_{0,\lambda} (r) 
	=
	\left( \frac{\lambda}{\sqrt{\pi}}\right)^{\frac{1}{2}} \cdot e^{-\frac{\lambda (r-t)^{2}}{2}},
\end{equation} 
and the Fredholm index of \(d^t_{\lambda}\) is \(1\).

For each \(\lambda\), let \(\pr_{\lambda}^{t}\) be projection to the kernel of \(d_\lambda\). 
Since the minimal nonzero eigenvalue of $d_{\lambda}^t$ has a distance $\sqrt{2 \lambda}$ to the origin, we obtain Part \ref{part:C0(d-infty)} of the following

\begin{lemma}\label{lem:func-calc-on-d-lambda}
With \(d_{\lambda}^t\) as above and \(f(d_{\lambda}^{t}) \in \Bound (L^{2} (\R)^{\oplus 2})\) the operator obtained from \(f\in C_0(\R)\) by functional calculus, 
we have 
\begin{enumerate}[label=\textup{\arabic*)}]
\item \label{part:C0(d-infty)}
$\lim_{\lambda \to +\infty} \norm{ f(d_{\lambda}^{t}) - f(0) \cdot \pr^t_\lambda} = 0.$
\item If \(\chi\in C_b(\R)\) is a normalizing function, and \(\epsilon^t\) is the \textup(Borel measurable\textup) sign function on \(\R\) given by
\[ \epsilon^t (r) := \tfrac{r-t}{\abs{r-t}},\]
acting as a multiplication operator on \(L^{2}(\R)\), then 
\begin{equation}\label{eq:chi-d-lambda-to-eps}
	F_{\lambda, t} := \chi (d_{\lambda}^t) \to 
	\mat{0 & \epsilon^t\\\epsilon^t & 0} \text{ for }\lambda \to +\infty
\end{equation}
in the strong operator topology. 
\item  If \(f\) is a smooth, periodic function on \(\R\), then 
\[
	\lim_{\lambda \to +\infty} \norm{ \comm{F_{\lambda, t}}{ f} } = 0
.\]
\end{enumerate}
\end{lemma}

The proof of the last claim is carried out in \cite{LR:Euler}, p.\ 582-583, and 
the last statement in 
 \cite{Em:NCG}, Chapter 7, Lemma 7.6, or \cite{LR:Euler}, p. 584-586.

Define a family \(\{ \Hilbertspace_{\lambda, t}\}_{\lambda} = \{ \Hilbertspace_{\lambda, t}^{+} \oplus \Hilbertspace_{\lambda, t}^{-}\}_{\lambda \in [1, +\infty]} \) of 
Hilbert spaces by setting \( \Hilbertspace_{\lambda, t}^{-} := L^{2}(\R)\) for all \(\lambda\in [1, + \infty]\), and 
\begin{equation*}
 \Hilbertspace_{\lambda, t}^{+} = \left\{ 
\begin{array}{l l}
 L^{2}(\R)
 & \quad \mbox{if \(1 \le \lambda <\infty\), }\\
 &\\
 L^{2}(\R)\oplus \C & \quad \mbox{if $\lambda = \infty$. }\\ \end{array} \right. 
 \end{equation*}

We let $\delta^t_0 = (0,1) \in W_{\infty, t}^{+} = L^{2}(\R) \oplus \C$. 

To endow this field with a structure of a continuous field, we only need be 
concerned about the point \(\infty\): We declare a section \(\xi^t\) of the field \( \{ \Hilbertspace_{\lambda, t}^{+}\}_{\lambda \in [1, +\infty]}\) 
with value \( f + z\delta_0^t\) at \(\lambda = + \infty\), $ f \in L^{2} (\R)$ and $z \in \mathbb{C}$, to be 
\emph{continuous at infinity} if 
\begin{equation}\label{eq:cty-at-infty}
	\norm{ \xi^t (\lambda) - ( f + z\psi^t_{0,\lambda}) }_{L^{2}(\R)} \to 0\quad \textup{as} \; \lambda \to +\infty,
\end{equation}
where \(\psi^t_{0, \lambda} \in L^{2}(\R)\) is the normalized \(0\)-eigenvector of \(d^t_{\lambda}\) as defined in Equation (\ref{eq:def:psi0}).

We now describe a continuous family of self-adjoint, grading-reversing operators 
\[ F_{\lambda, t} \colon \Hilbertspace_{\lambda, t} \to \Hilbertspace_{\lambda, t}\]
for \(\lambda \in [1, +\infty]\). 
For finite \(\lambda\), set 
\[
	F_{\lambda, t} := \chi (d_{\lambda}^{t}), \quad \textup{where } d_{\lambda}^{t} = \begin{bmatrix} 0 & \lambda ( \mathsf{M} -t ) - \frac{\partial }{\partial r}\\ \lambda ( \mathsf{M} -t ) + \frac{\partial }{\partial r} & 0 \end{bmatrix} 
.\]
This odd, self-adjoint operator has the form 
\[
	F_{\lambda,t} = \begin{bmatrix} 0 & G_{\lambda,t}^{*}\\ G_{\lambda,t} & 0 \end{bmatrix}
\]
for suitable \(G_{\lambda,t}\).

At infinity, we have \(\Hilbertspace_{\infty,t} = ( L^{2}(\R)\oplus \C ) \oplus L^{2}(\R)\) with the first summand \(L^{2}(\R)\oplus \C\) graded even and the second summand \(L^{2}(\R)\) graded odd. We let 
\[
	G_{\infty,t} \colon L^{2}(\R) \oplus \C \to L^{2}(\R)
\]
be multiplication by the sign function \(\epsilon^t\) on the summand \(L^{2}(\R)\), and zero on the \(\C\)-summand. Thus, the operator \(G_{\infty,t}^{*} \colon L^{2}(\R) \to L^{2}(\R) \oplus \C\) is multiplication by \(\epsilon^t\) on \(L^{2}(\R)\), followed by the inclusion into \(L^{2}(\R)\oplus \C\) by zero in the second summand.
The operator \(F_{\infty,t}\) is the odd, self-adjoint operator on \(\Hilbertspace_{\infty,t}\) given by the matrix 
\[ F_{\infty,t} := \begin{bmatrix} 0 & G_{\infty,t}^{*}\\ G_{\infty,t} & 0 \end{bmatrix}.\]
This is the correct choice in order to make $(F_{\lambda,t})_{\lambda}$ a continuous family, \emph{i.e.}\ an adjointable operator on the module of sections, 
 because of (\ref{eq:chi-d-lambda-to-eps}) in Lemma \ref{lem:func-calc-on-d-lambda}. Note that the operator \(L^{2}(\R) \to L^{2}(\R) \) of multiplication by $\epsilon^t$ 
has no kernel. Since, however, \(G_{\infty,t}\) kills the second summand \(\C\) of \(L^{2}(\R) \oplus \C\), the operator \(G_{\infty,t}\) has a \(1\)-dimensional kernel. The cokernel of \(G_{\infty,t}\) is clearly trivial, and therefore \(G_{\infty,t}\) (and \(F_{\infty,t}\)) also has index \(1\).

The family of operators \( \{ F_{\lambda,t}\}_{\lambda \in [1, +\infty]}\) induces an odd, self-adjoint operator \(F_t\) on the sections $\mathcal{E}_{t}$ of the field \( \{ \Hilbertspace_{\lambda, t}\}_{\lambda \in [1, +\infty]}\). In other words, we have constructed a \(\Z/2\)-graded Hilbert \(C([1, +\infty])\)-module and an odd, self-adjoint operator \(F_t\) on \(\Hilb_t\). Further, \(1-F_t^{2}\) is compact: for finite $\lambda$,
\[
	1- F_{\lambda, t}^{2}
	=
	(1- \chi^{2}) (d_{\lambda}^{t})
\]
is compact by Lemma \ref{lem:func-calc-on-d-lambda}. By the same lemma, 
\[
	\norm{(1-F_{\lambda, t}^{2}) - (1- \chi^{2})(0) \cdot \pr_{\lambda, t} }
	=
	\norm{1- F_{\lambda,t}^{2} - \pr_{\lambda,t} }
	\to 0
	\text{ for } 
	\lambda \to \infty.
\] 
As $\mathrm{pr}_{\lambda}^{t}=\ket{\psi_{0, \lambda}^{t}}\bra{\psi_{0, \lambda}^{t}}$ and 
	$1- F_{\infty,t}^{2}
	=
	( 0\oplus 1 )\oplus 0 = \ket{\delta_{0}^{t}}\bra{\delta_{0}^{t}}$ on $
	(L^{2} (\mathbb{R}) \oplus \mathbb{C} )
	\oplus L^{2} (\mathbb{R})
	,$
we see that $1- F_{\lambda, t}^{2}$ is asymptotic to $\ket{\xi}\bra{\xi}$, the rank-one operator corresponding to the continuous section given by
$
	\xi_{\lambda}
	:= 
	\psi_{0, \lambda}^{t}
	\text{ for } \lambda<\infty
$ and $
	\xi (\infty) = \delta_{0}^{t}
	.
$

The definitions above supply a 
homotopy of \(\KK_0(\C, \C)\)-cycles between \( (\Hilbertspace_{\lambda,t}, d_{\lambda}^{t}) = (L^{2}(\R) \oplus L^{2}(\R), d_{\lambda}^{t})\) for any finite \(\lambda\) and any \(t\in \R\), 
on the one hand, and the sum of the cycle \( ( \C\oplus 0, 0)\) with the degenerate cycle
\[
	\left( L^{2}(\R) \oplus L^{2}(\R), 
	\mat{ 0 & \epsilon^t\\ \epsilon^t & 0} \right)
\]
on the other hand. Here, both \(\C\oplus 0\) and \(L^{2}(\R)\oplus L^{2}(\R)\) are \(\Z/2\)-graded with their respective first summand even and second odd, and \(\epsilon^t\) is the sign function as before. 

Further, the homotopy is equivariant for \(\Z\) if one allows the 
real parameter \(t\in \R\) to change with the integer action: translation by \(n\in \Z\) 
conjugates \(d^t_\lambda \) to \(d^{t+n}_\lambda\). 
This means that the 
construction can be carried out in \(\RKK^\Z(\R; \cdots, \cdots)\), as we now show.

Set 
	\[ \Hilb_{\lambda, t}:=  C(\T) \otimes \Hilbertspace_{\lambda, t}
	 \and
	\mathcal{F}_{\lambda,t} := \mathrm{id}_{C(\mathbb{T})}\otimes F_{\lambda,t},\]
endowed with its standard 
right Hilbert \(C(\T)\)-module structure, and carrying the \(\Z/2\)-grading inherited from 
the gradings on \(\Hilbertspace_{\lambda, t}\). On \(\Hilb_{\lambda, t}\) and for \(f \in C(\T)\) 
considered a periodic function on~\(\R\), we let 
\[\nu_{b,\lambda, t}  (f) \in \Bound (\Hilb_{\lambda, t})\]
be the operator defined as follows. 
Set   
\[ f_{b} ([x], r) = f([x + br]),\]
where $b$ is the integer which was fixed in the beginning.
For finite \(\lambda\),  we let 
\(\nu_{\lambda, t}^{\pm} (f)\) act on \(\Hilb^{\pm}_{\lambda, t} = C(\T)\otimes L^2(\R)\) by multiplication by the function \(f_{b }\) 
on \(\T\times \R\). For \(\lambda = \infty\), we let 
\(\nu^+_{b , \infty, t} (f)\) act on \(\Hilb^+_{\infty, t}
 = C(\T) \otimes \left( L^2(\R)\oplus \C \right) = 
C(\T)\otimes L^2(\R) \, \oplus \, C(\T)
\) 
by multiplication by \(f_{b, \lambda}\) on the first factor 
\(C(\T)\otimes L^2(\R)\), and on the second factor \(C(\T)\) by the 
multiplication by the function \(f_{b}^{t} \in C(\T)\), where 
\[ f_{b}^{t} ([x]) :=  f([x+bt]).\]

 For $t\in\mathbb{R}$ and $\lambda\in (0,\infty],$ let
\[
	\homotopy_{\lambda,t}
	:=
	\left( \nu_{ \lambda, t}, \Hilb_{\lambda, t}, \mathcal{F}_{\lambda, t} \right)
 	\and 
 	\homotopy_{\lambda}:= \left\{\homotopy_{\lambda,t}\right\}_{t\in\mathbb{R}}.
\]
These \(\RKK\)-cycles are \(\Z\)-equivariant, and  $\left(\lambda\mapsto\homotopy_{\lambda}\right)$ is a homotopy of $\RKK^{\mathbb{Z}}$-cycles. For any $\lambda\in (0,\infty),$ Equation~\eqref{eq:bdd-perturbation-of-constant-family} yields that $\homotopy_{\lambda}$ is a compact perturbation of the constant family $\left\{\homotopy_{\lambda,0}\right\}_{t\in\mathbb{R}}$, because they arise as the bounded transform of $\{( L^{2}(\mathbb{R})\oplus L^{2}(\mathbb{R}), d_{\lambda}^{t} )\}_{t\in\mathbb{R}}$ resp.\ $\mathrm{pr}_{\mathbb{R}}^{*}(  L^{2}(\mathbb{R})\oplus L^{2}(\mathbb{R}), d_{\lambda} )$ after fibrewise tensoring with the right-Hilbert $C(\mathbb{T})$-bimodule $(\nu , C(\mathbb{T}))$. Thus, $\homotopy_{\lambda}$ and $\left\{\homotopy_{\lambda,0}\right\}_{t\in\mathbb{R}}$ determine the same class in $\RKK^{\mathbb{Z}}$. By definition of the inflation map, 
\(
	\mathrm{pr}_{\mathbb{R}}^{*} \bigl( ( H_{b}, \mathrm{id}_{C(\mathbb{T})}\otimes d_{\lambda})\bigr)
	=
	\left\{\homotopy_{\lambda,0}\right\}_{t\in\mathbb{R}}
\)
for any finite $\lambda$, so we have shown that $\mathrm{pr}_{\mathbb{R}}^{*} \bigl( ( H_{b}, \mathrm{id}_{C(\mathbb{T})}\otimes d_{\lambda})\bigr)$ and $\homotopy_{\lambda}$ determine the same class.

On the other hand, at $\lambda=\infty$, we have that $\homotopy_{\infty}$ is the sum of the topological $b$-twist \(\widehat{\tau}^{b}=\{\widehat{\tau}^{b}_{t} \}_{t\in\mathbb{R}}\), see Definition~\ref{def:hat-tau-b}, and the degenerate $\left( t \mapsto \left( H_{b}, \mat{ 0 & \epsilon^{t} \\ \epsilon^{t} & 0 }\right)\right)$. In particular, $\widehat{\tau}^{b}$ also determines the same class as  $\homotopy_{\lambda}$ in $\RKK^{\mathbb{Z}}$. This concludes our proof of Theorem~\ref{theorem:keyindexreduction}. 
\end{proof}

\begin{proof}[Proof of Theorem~\ref{theorem:thetwistedindextheorem}]
	Since $\mathrm{pr}_{\R}^{*}$ is an isomorphism, it follows from Theorem~\ref{theorem:keyindexreduction} that $b\mapsto \bigl[ ( H_{b}, \mathrm{id}_{C(\T)}\otimes d_{\lambda})\bigr]$ is a group homomorphism from $\Z$ to $\KK^{\mathbb{Z}}_{0}(C(\T),C(\T))$. Using descent, the map
	\[
		b\mapsto
		j \bigl[ ( H_{b}, \mathrm{id}_{C(\T)}\otimes d_{\lambda})\bigr]
		=
		\bigl[ ( \mathcal{H}_{b}, D_{\lambda})\bigr]
		=
		\twist_{b}
	\]
	is a group homomorphism from $\Z$ to $\KK_{0}(A_{\theta},A_{\theta})$, as claimed.
\end{proof}

We conclude this section by computing the action of a \(b\)-twist on \(\K\)-theory. 

A small variant of Kasparov's descent map is a natural map
\begin{equation}\label{eq:Kasparovs-map}
\begin{tikzcd}
	\RKK^\Z_*(\R; C(\T), C(\T) )
	\ar[r, "\lambda^{\Z}"]
	&\RKK_*\left(\T; C(\R\times_{\mathbb{Z}} \T), C(\R\times_{\mathbb{Z}} \T)\right) 
\end{tikzcd}
\end{equation}
which is similar to the usual `descent,' but contains a bimodule construction as well. It is routine to compute. 
\begin{lemma}
\label{lemma:yay}
	Kasparov's map \eqref{eq:Kasparovs-map}	for $*=0$, followed by the forgetful map
	\[\RKK_0\left(\T; C(\R\times_{\mathbb{Z}} \T), C(\R\times_{\mathbb{Z}} \T)\right)\to
	\KK_0\left(C(\R\times_\Z \T), C(\R\times_\Z)\right) \cong \KK_0\left( C(\T^2), C(\T^2)\right),\]
		sends \([\widehat{\tau^b}]\) to the class of the homeomorphism of \(\T^2\)
	 of matrix multiplication by 
	  \(\begin{bmatrix} 1 & b\\ 0 & 1\end{bmatrix}\). 
\end{lemma}

The class \(\beta\in \KK^\Z_1(C_0(\R), \C)\) of the Dirac operator on \(\R\) determines a 
\(\KK_1\)-equivalence \(j(\beta\times 1_{C(\T)} ) \in \KK_1(C(\R\times_\Z \T), A_\theta)\). 
Conjugating by the equivalence determines an 
isomorphism 
\[ \KK_*(A_\theta, A_\theta) \cong \KK_{*}( C(\R\times_\Z \T, C(\R\times_\Z \T).\]
It fits into the right vertical side of a diagram: 
\begin{equation}
\label{cd}
\begin{tikzcd}[column sep = 8em]
	\KK^\Z_*(C(\T), C(\T))
	\ar[d, "p_{\R}^{*}"]
	\ar[r,"j"] 
	& \KK_*(A_{\theta}, A_{\theta}) \ar[d, "\cong"]
	\\
	\RKK^\Z_*(\R; C(\T), C(\T))
	\ar[r, "\mu"]
	& 
	\KK_*(C(\R\times_{\mathbb{Z}} \T), C(\R\times_{\mathbb{Z}} \T)) 
\end{tikzcd}
\end{equation}

Since matrix multiplication on \(\T^2\) by 
 any element of \(\SL_2(\Z)\), induces the identity map on \(\K^0(\T^2)\), and 
 under the standard identification \(\K^1(\T^2) \cong \Z^2\) acts by multiplication by 
 the matrix, we obtain, by Lemma \ref{lemma:yay},
  a corresponding statement about how the \(b\)-twist 
 \(\tau_b \in \KK_0(A_\theta, A_\theta)\) acts. It acts as the identity on 
 \(\K_1(A_\theta)\), and on the identification of \(\K_0(A_\theta)\) with 
 \(\K^1(\T^2) \cong \Z^2\), it acts on combinations of the standard generators by 
 matrix multiplication by \(\mat{1 & b\\0 &1}\).  
 
 We believe that the image of these two generators under Dirac equivalence, are the 
 classes \([1]\in \K_0(A_\theta)\), of the unit, and the class \([p]\in \K_0(A_\theta)\) of the 
 Rieffel--Powers projection, but we do not know of a convenient reference.  In any case, 
 either of these \(\K\)-theory classes 
 are fixed by automorphisms of \(A_\theta\) (since they are unital, and 
 preserve the trace). So the \(b\)-twist
  \(\tau_b\in \KK_0(A_\theta, A_\theta)\) is not represented by an automorphism of \(A_\theta\).

\section{ Poincar\'e duality calculation}

One of the two main technical results of this paper is the following. Let 
\(\Delta_\theta \in \KK_0(A_\theta\otimes A_\theta, \C)\) be the class of 
Definition \ref{def:fc}.

\begin{theorem}\label{thm:j-H-d-H=TheProduct}
	Let \(g = \begin{bsmallmatrix}1 & b\\0 & 1\end{bsmallmatrix}\) for $b\neq 0$ and \(\module_b:= \module_g\). Then 
	\begin{equation}
	\label{equation:wowthatsamazing}
		(1_{A_{\theta}} \otimes [\module_{ b }])
		\otimes_{A_{\theta}^{\otimes 3}}
		(\Delta_{\theta} \otimes 1_{A_{\theta}}) =
		[(\mathcal{H}_{ b },  \tfrac{1}{b} D_{2 \pi b})]\in \KK_0(A_\theta, A_\theta)
		.
	\end{equation}
	In particular, if $b>0$, then this class coincides with \(\twist_b\in \KK_0(A_\theta, A_\theta)\), the \(b\)-twist 
	\textup(Definition \ref{definition:btwist}\textup).
\end{theorem}

We proceed to the proof of Theorem \ref{thm:j-H-d-H=TheProduct}, which is 
fairly long. 

\subsection{Computation of the module in the zig-zag product}

Our goal is to compute 
 $( 1_{A_{\theta}} \otimes [ \module_{g} ] ) \otimes_{A_{\theta}^{\otimes 3}} ( \Delta_{\theta} \otimes 1_{A_{\theta}} ) \in \KK_0(A_\theta, A_\theta)$ for \(g\) upper-triangular, 
 and prove that it equals the class of the \(b\)-twist of Theorem \ref{thm:HdH-is-cycle}.

 In fact, some of the calculations we will do for arbitrary \(g\), since it involves little additional effort and leads to the following observation: only for upper-triangular \(g\), the Hilbert \(A_\theta\)-bimodule 
 involved in the Kasparov product of the left hand side of \eqref{equation:wowthatsamazing} is of the 
 kind one gets from applying descent to an equivariant module (such as the one appearing 
 in our cycle for the \(b\)-twist).

As the module $\module_{g}$ and the $\Cstar$-algebra $ A_{\theta} $ are ungraded, the module underlying this class is comprised of two copies of
\begin{equation*}
	( A_{\theta} \otimes \module_{g})
	 \otimes_{A_{\theta}^{\otimes 3}} 
	\left( L^{2} \otimes A_{\theta} \right)
	,
\end{equation*}
 where $L^{2} = L^2 (\mathbb{T})\otimes \ell^2 (\mathbb{Z})$ as before (see Lemma~\ref{def:Delta}). We initially focus on 
 describing this bimodule. Observe first that one is reduced to computing 
$
	 \module_{g} 
	\otimes_{ A_{\theta} }
	 L^{2} ,
$
where the balancing is over $ A_{\theta} \otimes 1$ acting on the right of $ \module_{g} $, and $ A_{\theta} $ acting on the left of $ L^{2} $ via $\omega_{2} \rtimes v$. This is because the maps
\[
	( A_{\theta} \otimes \module_{g})
	 \otimes_{A_{\theta}^{\otimes 3}} 
	( L^{2} \otimes A_{\theta} )
	\longleftrightarrow
	 \module_{g} \otimes_{ A_{\theta} } L^{2} 
\]
defined on elementary tensors by
\begin{equation}\label{iso:ALgL2AtoLgL2}
\begin{split}
	(a\otimes \Phi) \otimes (f\otimes b)
	&\longmapsto
	\Phi._{\module_{g}} (1 \otimes b) \otimes (\omega_{1} \rtimes u)(a)\left( f\right)
	\\
	(1\otimes \Phi)\otimes (f\otimes 1)
	&\longmapsfrom
	\Phi\otimes f
\end{split}
\end{equation}
are inverse to one another and therefore 
equip the right-hand side with the structure of a right-Hilbert $ A_{\theta} $-bimodule as follows:
\begin{equation}
\label{formula:ModuleStructureOnE0L2}
\begin{split}
	A_{\theta} \act ( \module_{g} \otimes_{ A_{\theta} } L^{2} ):
	&\qquad
	\xi (\Phi\otimes f)
	:= 
	\Phi \otimes (\omega_{1}\rtimes u)(\xi) \left(f\right),
	\\
	 ( \module_{g} \otimes_{ A_{\theta} } L^{2} ) \curvearrowleft A_{\theta} :
	&\qquad
	(\Phi\otimes f) \xi 
	:= 
	\Phi._{ \module_{g}}(1\otimes \xi) \otimes f.
\end{split}
\end{equation}
Moreover, $\module_{g} \otimes_{ A_{\theta} } L^{2} $ has $A_{\theta} \otimes A_{\theta}$-valued inner product given in elementary tensors by
	\begin{align}
		\label{eq:inner-of-EmfA}
		\inner{\Phi \otimes f_{1} }{\Psi \otimes f_{2} }^{ }
		&=
		\iinner{(1\otimes\Phi) \otimes ( f_{1} \otimes 1)}{(1\otimes \Psi) \otimes ( f_{2} \otimes 1)}^{ (1_{A_{\theta}} \otimes \module_{g} ) \otimes_{A_{\theta}^{\otimes 3}} ( L^2 \otimes 1_{A_{\theta}} )}
		\notag
		\\
		&=
		\inner{
			f_{1} \otimes 1
		}{
			\left(
				\inner{
					1\otimes\Phi
				}{
					1\otimes \Psi
				}^{{ A_{\theta}} \otimes \module_{g} }_{A_{\theta}^{\otimes 3}}
			\right)
			\cdot
			(f_{2} \otimes 1)
		}^{ L^2 \otimes { A_{\theta}}}
		,
	\end{align}
	where $\cdot$ denotes, for the moment, the left-action of $A_{\theta}^{\otimes 3}$ on $ L^2 \otimes { A_{\theta}}$. Note that, since we induce this inner product on $ \module_{g} \otimes_{ A_{\theta} } L^{2} $ via the bijection, we do not need to worry about topologies.

\begin{lemma}\label{lem:def-of-mcNg0}
		The maps
		\begin{align}
		\begin{split}\label{formula:general:scr-h-in-EL}
			C_c (\mathcal{Z}_{g})
			&\longrightarrow 
			C_c(\mathcal{Z}_{g})\odot_{\mathfrak{A}}
					\mathfrak{A} 
			\;\subseteq\;
	 \module_{g}\otimes_{ A_{\theta} } L^{2} 
			\\
			\Phi
			&\longmapsto
			\Phi \otimes 	(z^{0}\otimes\varepsilon_{0})
			\\
			\Phi._{\module_{g}}(V^{l } U^{-k } \otimes 1)
			&\longmapsfrom
			\Phi
			\otimes
			(z^{l }\otimes\varepsilon_{k })
		\end{split}
		\end{align}
		are mutually inverse. In particular with the help of Formula \eqref{iso:ALgL2AtoLgL2}, a copy of the space $C_c (\mathcal{Z}_{g})$ is sitting densely inside of
		$( A_{\theta} \otimes \module_{g})
				 \otimes_{A_{\theta}^{\otimes 3}} 
				\left( L^{2} \otimes A_{\theta} \right)$.				
	\end{lemma}
	
	In the above lemma, we write $\mathfrak{A}$ for two things: on the one hand, it denotes the dense subspace of $L^2$ consisting of elements $\sum_{n,m } a_{n,m}z^n\otimes \varepsilon_m$. On the other hand, it denotes the subalgebra of $A_{\theta}$ consisting of elements $\sum_{n,m} a_{n,m}V^n U^m$. In both of these cases, $(a_{n,m})_{n,m}$ is assumed to be of Schwartz decay. Recall also that $\odot$ denotes the algebraic tensor product before completion.
	
	\begin{proof}
	On the right-hand side, the balancing gives us the following equality for $\Phi\in C_c(\mathcal{Z}_{g})$, $f\in \mathfrak{A}\subseteq L^{2} $, and any acting element $\xi\in \mathfrak{A}\subseteq A_{\theta} $:
	\begin{equation*}
		\Phi \otimes (\omega_{2} \rtimes v)(\xi)\left( f\right)
		=
		\Phi ._{ \module_{g}} (\xi \otimes 1) \otimes f
		.
	\end{equation*}
	For $\xi=V^{l_{1}} U^{k_{1}}$ and $f=z^{l_{2}}\otimes \varepsilon_{k_{2}}$, we have
	\[
		(\omega_{2} \rtimes v)(\xi)\left( f\right)
		=
		\lambda^{-k_{1}l_{2}} z^{l_{2}+l_{1}}\otimes \varepsilon_{k_{2}-k_{1}}
		,
	\]
	where $\lambda := \mathrm{e}^{2\pi \imaginary \theta}$. So we have for any choice of $l_{1}, k_{1}\in\mathbb{Z} $:
	\begin{gather*}
		\Phi
		\otimes 
		\left(
			\lambda^{-k_{1}l_{2}}\cdot
			z^{l_{2}+l_{1}}\,
			\otimes
			\varepsilon_{k_{2}-k_{1}}
		\right)
			= 
		(\Phi._{ \module_{g}}(V^{l_{1}} U^{k_{1}} \otimes 1))
		\otimes
		\left(
			z^{l_{2}}
			\otimes
			\varepsilon_{k_{2}}
		\right)
		.
	\end{gather*}
	The case $k_{2} := 0,$ $l_{2} := 0$ and $k_{1}$ replaced by $-k_{1}$ yields:
	\begin{gather*}
		\Phi
		\otimes 
		\left(
			z^{l_{1}}\,
			\otimes
			\varepsilon_{k_{1}}
		\right)
			= 
		(\Phi._{ \module_{g}}(V^{l_{1}} U^{-k_{1}} \otimes 1))
		\otimes
		\left(
			z^{0}
			\otimes
			\varepsilon_{0}
		\right)
		.
	\end{gather*}
	It is now easy to see that the two maps are mutually inverse maps, as claimed.
	\end{proof}

		Formula \eqref{formula:general:scr-h-in-EL} equips
		the left-hand side with the structure of an $\mathfrak{A}-\mathfrak{A}$-right-pre-Hilbert module.
		We let $\mathcal{N}_{g}^{0}$ be its completion and $\mathcal{N}_{g} := \mathcal{N}_{g}^{0}\oplus \mathcal{N}_{g}^{0}$ with the standard even grading. 
		By construction, $\mathcal{N}_{g}$ is (isomorphic to) the $ A_{\theta} - A_{\theta} $-right-Hilbert module underlying $( 1_{A_{\theta}} \otimes [ \module_{g} ] ) \otimes_{A_{\theta}^{\otimes 3}} ( \Delta_{\theta} \otimes 1_{A_{\theta}} )$.	We will now study this $\mathfrak{A}-\mathfrak{A}$-right-pre-Hilbert module in terms of the bimodule structure of~$\module_{g}$.

	\begin{lemma}[The Hilbert bimodule structure of $\mathcal{N}_{g}^{0}$]\label{lem:mathscr-left-action}	\label{lem:MakeInnerEtoInnerH}
		The bimodule structure on $\mathcal{N}_{g}^{0}$ 
		is given on its dense subspace \mbox{$C_c (\mathcal{Z}_{g})$} by
		\begin{equation}
		\label{formula:general-mathscr-h:bimodule}
		\begin{alignedat}{2}
			\mathfrak{A} \curvearrowright C_c (\mathcal{Z}_{g}):
			&&\qquad
			(V^{l_{1}}U^{k_{1}} \square \Phi )
	 &= 
			\lambda^{l_{1}k_{1}}\,\Phi._{ \module_{g} }(V^{l_{1}} U^{-k_{1}} \otimes 1),
			\\
			C_c (\mathcal{Z}_{g}) \curvearrowleft \mathfrak{A}
			:&&\qquad
			(\Phi\square V^{l_{2}}U^{k_{2}})
	 &= 
			\Phi._{ \module_{g} }(1\otimes V^{l_{2}}U^{k_{2}})
			.
		\end{alignedat}
		\end{equation}
		
		For $\Psi$ another compactly supported function on $\mathcal{Z}_{g}$, the \mbox{\textup(pre-\textup)}inner product $\inner{\Phi}{\Psi}^{\mathcal{N}_{g}^{0}}$ with value in $C_c (\mathcal{A}_{\theta})\subseteq A_{\theta}$ is given by 
		\[\inner{\Phi}{\Psi}^{\mathcal{N}_{g}^{0}} ([x],k)=\int_{\mathbb{T}}\inner{\Phi}{\Psi}^{ \module_{g} }([y],0,[x],k)\mathrm{d}\, y,\]
		where $\inner{\Phi}{\Psi}^{ \module_{g} }$ takes values in $C_{c}(\A) = C_{c} (\mathcal{A}_{\theta}	\times \mathcal{A}_{\theta}	)$.	
	\end{lemma}
	
		\begin{proof}
	An element $\Phi\in C_c (\mathcal{Z}_{g})$ corresponds to $\Phi \otimes (z^{0} \otimes \varepsilon_{0})$ in $ \module_{g} \otimes_{ A_{\theta} } L^{2} $, see Formula \eqref{formula:general:scr-h-in-EL}.
	By Formula \eqref{formula:ModuleStructureOnE0L2}, the left action on $ \module_{g} \otimes_{ A_{\theta} } L^{2} $ is given by
	\[
		V^{l_{1}}U^{k_{1}} .\left( \Phi \otimes (z^{0} \otimes \varepsilon_{0})\right)
		=
		\Phi \otimes (\omega_{1}\rtimes u)(V^{l_{1}}U^{k_{1}}) \left(z^{0} \otimes \varepsilon_{0}\right).
	\]
	We compute
	\begin{align*}
		(\omega_{1}\rtimes u)(V^{l_{1}}U^{k_{1}}) \left(z^{0}\otimes\varepsilon_{0}\right)
		=
		\lambda^{l_{1} k_{1}}\,
		z^{l_{1}}\otimes\varepsilon_{k_{1}}
		,
	\end{align*}
	so that
	\[
		V^{l_{1}}U^{k_{1}} .\left( \Phi \otimes (z^{0} \otimes \varepsilon_{0})\right)
		=
		\lambda^{l_{1}k_{1}}\,\Phi\otimes 	(z^{l_{1}}\otimes\varepsilon_{k_{1}}).
	\]
	Similarly, the right action on $ \module_{g} \otimes_{ A_{\theta} } L^{2} $ is given by
	\[
		\left(\Phi \otimes (z^{0} \otimes \varepsilon_{0})\right) .\, V^{l_{2}}U^{k_{2}}
		=
		\Phi._{ \module_{g} }(1\otimes V^{l_{2}}U^{k_{2}}) \otimes (z^{0} \otimes \varepsilon_{0}).
	\]
	The claim about the bimodule structure now follows from Formula \eqref{formula:general:scr-h-in-EL}.
 
 	Next, we turn to the inner product. Because of Equation \eqref{eq:inner-of-EmfA} and Formula \eqref{formula:general:scr-h-in-EL}, the pre-inner product on $\mathcal{N}_{g}^{0}$ is given by
		\begin{align}\begin{split}\label{formula:innerH}
			\inner{\Phi}{\Psi}^{\mathcal{N}_{g}^{0}} 
			&=
			\inner{
				z^{0}\otimes\varepsilon_{0} \otimes 1
			}{
				\left(
					\inner{
						1\otimes\Phi
					}{
						1\otimes \Psi
					}^{{ A_{\theta}} \otimes \module_{g} }_{A_{\theta}^{\otimes 3}}
				\right)
				\cdot
				(z^{0}\otimes\varepsilon_{0} \otimes 1)
			}^{ L^2 \otimes { A_{\theta}}}
			,
		\end{split}
		\end{align}
		where $\cdot$ is, as before, the left-action of $A_{\theta}^{\otimes 3}$ on $ L^2 \otimes { A_{\theta}}$. Since $$\inner{1\otimes\Phi}{1\otimes \Psi}^{{ A_{\theta}} \otimes \module_{g} }_{A_{\theta}^{\otimes 3}} = 1 \otimes\inner{	\Phi}{\Psi}^{ \module_{g} }_{ A_{\theta} \otimes A_{\theta} },$$ 
		let us study Equation \eqref{formula:innerH} for $\inner{
								1\otimes\Phi
							}{
								1\otimes \Psi
							}$ replaced by an elementary tensor $1\otimes a \otimes b$:
		\begin{align*}
			&
			\inner{
				z^{0}\otimes\varepsilon_{0} \otimes 1
			}{
				\bigl(
					1\otimes a\otimes b
				\bigr)
				\cdot
				(z^{0}\otimes\varepsilon_{0} \otimes 1)
			}^{ L^2 \otimes { A_{\theta}}} 
			=
			\inner{
					z^{0}\otimes\varepsilon_{0}
			}{
				\omega_{2}\rtimes v(a) (z^{0}\otimes\varepsilon_{0})
				}^{ L^2 }
			\cdot
			b
			.
		\end{align*}
		
		
		For $a= \sum_{n,m} a_{n,m} V^n U^m $, we have
		\begin{align*}
			\omega_{2}\rtimes v(a) (z^{0}\otimes\varepsilon_{0}) 
			=
			\sum_{n,m} a_{n,m} z^{n}\otimes\varepsilon_{-m}
			,
		\end{align*}
		so that
		\[
			\inner{
					z^{0}\otimes\varepsilon_{0}
			}{
				\omega_{2}\rtimes v(a) (z^{0}\otimes\varepsilon_{0})
				}^{ L^2 } 
			=
			a_{0,0}.
		\]
		All in all, we have for any $([x],k)\in \mathbb{T}\times\mathbb{Z}$:
		\begin{align*}
		\inner{
								z^{0}\otimes\varepsilon_{0} \otimes 1
							}{
								\bigl(
									1\otimes a\otimes b
								\bigr)
								\cdot
								(z^{0}\otimes\varepsilon_{0} \otimes 1)
							}^{ L^2 \otimes { A_{\theta}}}_{A_{\theta}}
					([x],k)
			&=
			a_{0,0}
			\cdot
			b ([x],k)
			\\&=					
			\int_{\mathbb{T}}(a\otimes b)([y],0,[x],k)\mathrm{d}\, y
			.
		\end{align*}
		We bootstrap from the elementary tensor $a\otimes b$ with $a\in \mathfrak{A}$ to a more general element $\zeta \in C_c (\mathbb{T}\times\mathbb{Z}\times\mathbb{T}\times\mathbb{Z})$ with the result
		\[
			\inner{
						z^{0}\otimes\varepsilon_{0} \otimes 1
					}{
						\bigl(
							1\otimes \zeta
						\bigr)
						\cdot
						\,
						(z^{0}\otimes\varepsilon_{0} \otimes 1)
					}^{ L^2 \otimes { A_{\theta}}}
										([x],k)
				=
				\int_{\mathbb{T}}\zeta([y],0,[x],k)\mathrm{d}\, y
		\]
		and so in particular
		\begin{align*}
			\inner{\Phi}{\Psi}^{\mathcal{N}_{g}^{0}}
											([x],k)
			&=
			\inner{
				z^{0}\otimes\varepsilon_{0} \otimes 1
			}{
				\bigl(
					1\otimes \inner{\Phi}{\Psi}^{ \module_{g} }
				\bigr)
				\cdot
				\,
				(z^{0}\otimes\varepsilon_{0} \otimes 1)
			}^{ L^2 \otimes { A_{\theta}}}
			\\
			&=
			\int_{\mathbb{T}}
									\inner{\Phi}{\Psi}^{ \module_{g} }
								([y],0,[x],k)\mathrm{d}\, y
								,
		\end{align*}
		where the last equation follows from Formula \eqref{formula:innerH}.
	\end{proof}

Our goal is to show that for \(g\) upper-triangular, the module \(\mathcal{N}_g\) 
underlying the cup-cap product 
		$
			( 1_{A_{\theta}} \otimes [ \module_{g} ] ) \otimes_{A_{\theta}^{\otimes 3}} ( \Delta_{\theta} \otimes 1_{A_{\theta}} )
		$
		is obtained by applying the descent map to an equivariant module, and to identify this module. 
		Such `descended' modules are completions of $C_c (\mathbb{Z}, N)$ for some right-Hilbert $C(\mathbb{T}) - C(\mathbb{T}) $-bimodule $N$ equipped with a $ \mathbb{Z} $-action. As already mentioned, $\mathcal{N}_{g}^{0}$ is a completion of continuous compactly supported functions on the space $\mathcal{Z}_g$, which for $g=\mat{a&b\\c&d}$ is given by
		\[
			\mathcal{Z}_{\mat{a&b\\c&d}}
			=
			\left\{
				(r_{1},r_{2},\mat{ v \\ w })
				\in\mathbb{R}^2 \times \mathbb{T}^2
				\,\vert\,
				\mat{a(v+r_{1}\theta)+ br_{1} \\ c(v+r_{1}\theta)+ dr_{1} }
				=		
				\mat{ w+r_{2} \theta \\ r_{2} }
			\right\}
			,
		\]
		see Lemma \ref{lem:Yg-ast-X=Zg}. 
		We therefore need to restrict to those $g$ which make $\mathcal{Z}_{g}$ contain a copy of $\mathbb{Z}$. From the above description, we see that this happens exactly when $g$ is upper triangular; then the elements of $\mathcal{Z}_{g}$ have the restriction $\left[dr_{1} \right] = \left[r_{2}\right],$ \emph{i.e.}\ $r_{2}= dr_{1} +k$ for some $k\in \mathbb{Z}$.
			
		Since $g$ was assumed to be in $\mathrm{SL}_{2} (\mathbb{Z})$, $c=0$ implies $d=a$, and $\mu (g)  \neq 0$ (Equation~\eqref{eq:def-mu}) becomes $b\neq 0$. We get		
			\begin{equation}\label{eq:equiv-Zg}
			\begin{split}
				\mathcal{Z}_{\mat{a&b\\0&a}}
				&\cong
				\mathbb{T}\times\mathbb{R}\times\mathbb{Z}
				\\
				(r , r_{2},\mat{ v \\ w })
				&\mapsto 
				([v], r, r_{2} - ar)
				.
			\end{split}
			\end{equation}
		
		The below proposition gives the formulas that $\mathcal{Z}_{g}$ inherits from			$\mathcal{Y}_{g}\ast_{\mathcal{B} }\mathcal{X}$ via the identification from Equation \eqref{eq:equiv-Zg}. It also makes use of Lemma \ref{lem:def-Yg}, which gave a nicer description of the left $\mathcal{F}_{g}$-action on $\mathcal{Y}_{g}$, and of Lemma \ref{lem:def-Xth}, which gave a nicer description of the right $\mathcal{A}$-action on $\mathcal{X}$.
		\begin{proposition}\label{prop:gp-equiv-Zg}
			The $(\mathcal{F}_{g}^{},\mathcal{A})$-equivalence $\mathcal{Z}_{{\mat{a&b\\0&a}}}=\mathbb{T}\times\mathbb{R}\times\mathbb{Z}$ is given by:
			\begin{align*}
				\mathcal{F}_{g} \curvearrowright \mathcal{Z}_{g}:&
				\quad
						\left(\mat{v \\ 0}+(\tfrac{l_{1}+l_{2}\theta}{ b }+r )\slope,l_{1},l_{2} \right)
						.(\left[v\right],r ,k)
			=
						\left(
							\left[v\right],
							\tfrac{l_{1}+l_{2}\theta}{ b }+r ,
							k+ l_{2}				
						\right),
				\\
				\mathcal{Z}_{g} \curvearrowleft \mathcal{A}:&
				\quad
				([v],r ,k).	([v],k_{1},[av+ r { b } +k\theta ],k_{2}) = ([v-k_{1}\theta],r +k_{1},k+k_{2} - a k_{1})
				.
			\end{align*}
		\end{proposition}
		
		Now, we will finally compute the module structure of~$\module_{g}$, but only for matrices $g$ of the above form. This, in turn, will then allow us to give the formulas for the Hilbert module structure of $\mathcal{N}_{g}^{0} \cong ( A_{\theta} \otimes \module_{g}) \otimes_{A_{\theta}^{\otimes 3}} \left( L^{2} \otimes A_{\theta} \right)$.		

Let $ g=\mat{a&b\\0&a} $: recall that $\mathsf{Z}_{g}$ is the Morita equivalence built as completion of $C_c (\mathcal{Z}_{g})$, and by `forgetting' its left-action, we arrived at the right-$A_{\theta}\otimes A_{\theta}$-Hilbert module $\module_{g}=\iota^* (\mathsf{Z}_{g})$. This means that their right Hilbert-module structures coincide, and so according to Theorem~2.8 in~\cite{MRW:Grpd}, the right-$C_{c}(\mathcal{A})$-action on the dense subspace $C_c (\mathcal{Z}_{g})$ of~$\module_{g}$ needs to be defined by 
	\[
		(\Phi._{ \module_{g} } f) (\mathbf{z}) = \smashoperator{\int_{\substack{\text{sensible}\\\nu\in \mathcal{A} }}} \Phi(\mathbf{z}.\nu)f(\nu\inv) \,\mathrm{ d } \nu,
	\] 
	where $\nu$ is ``sensible'' if $\mathbf{z}.\nu$ makes sense. For $\mathbf{z}=([v],r,k)\in \mathcal{Z}_{g}$, this is the case exactly when
		$
			\nu =([v],-k_{1}, [av+ r { b } - k\theta], -k_{2})
		$
		for some $k_{i}\in\mathbb{Z}$,	in which case
		$$
			\mathbf{z}.\nu
			=
	 ([v+k_{1}\theta],r -k_{1},k-k_{2} + a k_{1})
			.
		$$
	
		The inverse of such $\nu$ in $\mathcal{A}_{\theta}\times\mathcal{A}_{\theta}$ is
		$
			\nu\inv
			=
			([v+k_{1}\theta],k_{1},[av+ r { b } +(k_{2}-k)\theta],k_{2})
		$.		
		All in all this means:
		\begin{align}\label{eq:rightActionOnE0-general}
		\begin{split}
			(\Phi._{ \module_{g} } f) ([v],r,k)
			=
			\smashoperator{\sum_{k_{1},k_{2}\in\mathbb{Z}}}
			&\Phi([v+k_{1}\theta],r -k_{1},k-k_{2} + a k_{1})
			\\
			&f([v+k_{1}\theta],k_{1},[av+ r { b } +(k_{2}-k)\theta],k_{2}).
		\end{split}
		\end{align}
		In particular, for $f= V^{l_{1}} U^{k_{1}} \otimes V^{l_{2}} U^{k_{2}}$:
		\begin{align}
		\begin{split}\label{eq:rightActionOnE0}
			(\Phi._{ \module_{g} } V^{l_{1}} U^{k_{1}} \otimes V^{l_{2}} U^{k_{2}}) ([v],r,k)
			=
			&\Phi([v+k_{1}\theta],r -k_{1},k-k_{2} + a k_{1})
			\\
			&
			\mathsf{e}^{2\pi \imaginary l_{1}(v+k_{1}\theta)}\mathsf{e}^{2\pi \imaginary l_{2}(av+ r { b } +(k_{2}-k)\theta)}
			.
		\end{split}
		\end{align}
	
	Now that we have concrete formulas for the right-action on $ \module_{g} $, we can make the structure of $\mathcal{N}_{g}^{0}$ concrete by using Formula \eqref{formula:general-mathscr-h:bimodule}:
	\begin{equation}\label{formula:mathscr-h:bimodule}
	\begin{tikzcd}
 \mathfrak{A}
 \curvearrowright \mathcal{N}_{g}^{0}:
		&[-20pt]
		(V^{l_{1}}U^{k_{1}} \square \Phi )([v],r,k)
		=
		\Phi([v-k_{1}\theta],r +k_{1},k - a k_{1})
		\mathsf{e}^{2\pi \imaginary l_{1}v},
		\\[-15pt]
		\mathcal{N}_{g}^{0} \curvearrowleft \mathfrak{A} :
		&[-20pt]
		(\Phi\square V^{l_{2}}U^{k_{2}})([v],r,k)
		=
		\lambda^{l_{2}(k_{2}-k)}
		\Phi([v],r,k-k_{2})
		\mathsf{e}^{2\pi \imaginary l_{2}(av+ r { b } )}.
	\end{tikzcd}
	\end{equation}	
		We now compare 
		this right-module structure of $\mathcal{N}_{g}^{0}$ to the right-module structure it would have if it came via descent from a suitable (yet to be determined) completion of $C_c(\mathbb{T}\times \mathbb{R} )$: for any $l_{2},k_{2}\in\mathbb{Z} $, $([v],r,k)\in\mathcal{Z}_{g}=\mathbb{T} \times\mathbb{R}\times\mathbb{Z} $, and $\Phi\in C_c (\mathcal{Z}_{g})$, we would need 
	\[
		\lambda^{l_{2}(k_{2}-k)}
		\Phi([v],r,k-k_{2})
		\mathsf{e}^{2\pi \imaginary l_{2}(av+ r { b }) }
		\stackrel{!}{=}
		\left(
			\lambda^{l_{2}(k_{2}-k)}\,\Phi (k-k_{2}) \ast z^{l_{2}}
		\right)
		([v],r)
		.
	\]
	Here, the left-hand side is the right-action by $V^{l_{2}}U^{k_{2}}$ on the function $\Phi$, an element of the dense subspace $C_c(\mathcal{Z}_{g})$ of $\mathcal{N}_{g}^{0}$. The right-hand side is the formula for the right-action by $V^{l_{2}}U^{k_{2}}$ as `prescribed' by descent; notice that $\Phi (k-k_{2})$ is our notation for the function
	\[
		\mathbb{T}\times\mathbb{R} \ni ([v],r)
		\mapsto
		\Phi ([v],r,k-k_{2})
	\]
	in $C_c (\mathbb{T}\times\mathbb{R})$.
	In other words, if we define for $\phi\in C_c (\mathbb{T} \times\mathbb{R} )$ and $f\in C(\mathbb{T})$,
	\begin{equation}\label{formula:right-action-before-descent}
		(\phi \ast f)([v],r)
		=
		\phi([v],r)
		f([av+ r b])
		,
	\end{equation}
	then descent turns this right-action of $C(\mathbb{T})$ on (a completion of) $C_c (\mathbb{T} \times\mathbb{R} )$ into the right-module structure we have on $\mathcal{N}_{g}^{0}$.
	
	For the left-module structure to be coming from descent, we similarly require for any $l_{1}, k_{1}\in\mathbb{Z}$ that
	\[
		\Phi([v-k_{1}\theta],r +k_{1},k - a k_{1})
		\mathsf{e}^{2\pi \imaginary l_{1}v}
		\stackrel{!}{=}
		\left(
			z^{l_{1}} \ast \left( k_{1}.\Phi (k-k_{1}) \right)
		\right)
		([v],r)
		.
	\]

	This shows that we need to have $a=1$, so that we can define for $\phi\in C_c (\mathbb{T} \times\mathbb{R} )$ the action of $k_{1}\in\mathbb{Z}$ and the left-action of $f\in C(\mathbb{T})$ by:
	\begin{equation}\label{formula:ZZ-and-left-action-before-descent}
		\left( k_{1}.\phi \right)([v],r)
		=
		\phi([v-k_{1}\theta],r +k_{1})
		\and
		\left(
			f \ast \phi
		\right)
		([v],r)
		=
		f([v])
		\phi([v],r)
		.	
	\end{equation}

For $g=\mat{1&b\\0&1}$ with $b\in\mathbb{Z}^\times$, the inner products of both~$\module_{g}$ and subsequently of $\mathcal{N}_{g}^{0}$ are now easy to compute. First, the $A_{\theta}\otimes A_{\theta}$-valued inner product of $\module_{g}=\iota^* (\mathsf{Z}_{g})$ is just the inner product of $\mathsf{Z}_g$. Therefore, Theorem~2.8 in~\cite{MRW:Grpd} gives us the following formula for the inner product of two functions $\Phi,\Psi\in C_c (\mathcal{Z}_{g}) \subseteq \module_{g}$ evaluated at $\nu\in\mathcal{A}=\mathcal{A}_{\theta}\times\mathcal{A}_{\theta}$:
\[
	\inner{\Phi}{\Psi}^{ \module_{g} } (\nu)
	=
	\smashoperator{\int_{\substack{\text{sensible}\\\gamma\in\mathcal{F}_{g}}}}
		\overline{\Phi}
		(\gamma.\mathbf{z})
		\Psi
		(\gamma.\mathbf{z}.\nu)
	\,\mathrm{ d } \gamma
	,
\]
where $\mathbf{z}\in \mathcal{Z}_{g}=\mathbb{T}\times\mathbb{R}\times\mathbb{Z}$ is any element such that $\mathbf{z}.\nu$ makes sense in $\mathcal{Z}_{g}$, and $\gamma$ is ``sensible'' if $\gamma.\mathbf{z}$ is defined. According to Proposition \ref{prop:gp-equiv-Zg}, when $\nu=([v],l_{1},[w],l_{2})$, we can take
the element $
	\mathbf{z} = ([v],\tfrac{ w - v }{ b },0)
$
for some choice of representatives $v,w$ of $[v,w]$. For sensible $\gamma\in\mathcal{F}_{g}$, we have 
\begin{align*}
	\gamma.\mathbf{z}
	&=
	\left(
		[v],
		\tfrac{k_{1}+k_{2}\theta + w - v }{ b },
		k_{2}	
	\right),
\end{align*}
where $k_{1},k_{2}\in\mathbb{Z}$ are arbitrary, and then
\begin{align*}
	\gamma.\mathbf{z}.\nu
	=
	\left(
		[v-l_{1}\theta],
		\tfrac{k_{1}+k_{2}\theta + w - v }{ b } + l_{1},
		k_{2}	+l_{2} - l_{1}
	\right)
	.
\end{align*}
All in all:
\begin{align*}
	\inner{\Phi}{\Psi}^{ \module_{g} } ([v],l_{1},[w],l_{2})
	=
	\smashoperator{\sum_{k_{1}, k_{2}\in\mathbb{Z}}}
		\overline{\Phi}
	\left(
		[v],
		\tfrac{k_{1}+k_{2}\theta + w - v }{ b },
		k_{2}	
	\right)
		\Psi
	\left(
		[v-l_{1}\theta],
		\tfrac{k_{1}+k_{2}\theta + w - v }{ b } + l_{1},
		k_{2}	+l_{2} - l_{1}
	\right)
	.
\end{align*}
Now we will use Lemma \ref{lem:MakeInnerEtoInnerH} to compute a formula for $\inner{\Phi}{\Psi}^{\mathcal{N}_{g}}$ where $\Phi, \Psi\in C_c (\mathcal{Z}_{g})\subseteq \mathcal{N}_{g}$:
\begin{align*}
	\inner{\Phi}{\Psi}^{\mathcal{N}_{g}} ([x],l)
	&=
	\int_{\mathbb{T}}
		\sum_{k_{1}, k_{2}\in\mathbb{Z}}
		\overline{\Phi}
	\left(
		[y],
		\tfrac{k_{1}+k_{2}\theta + x - y}{ b },
		k_{2}	
	\right)
		\Psi
	\left(
		[y ],
		\tfrac{k_{1}+k_{2}\theta + x - y}{ b },
		k_{2} + l
	\right)
	\mathrm{d}\, y
\\&
=
	\int_{\mathbb{R}}
	\smashoperator{\sum_{ k \in\mathbb{Z}}}
		\overline{\Phi}
	\left(
		[r],
		\tfrac{k \theta + x - r}{ b },
		k	
	\right)
		\Psi
	\left(
		[r ],
		\tfrac{k \theta + x - r}{ b },
		k	+ l
	\right)
	\mathrm{d}\, r
\\&
=
	\int_{\mathbb{R}}
	\smashoperator{\sum_{ k \in\mathbb{Z}}}
		\overline{\Phi}
	\left(
		[x+k \theta-r],
		\tfrac{r}{ b },
		k
	\right)
		\Psi
	\left(
		[x+k \theta-r ],
		\tfrac{r}{ b },
		k + l
	\right)
	\mathrm{d}\, r
	.
\end{align*}
For this to come from descent, we need
\begin{align*}
	\inner{\Phi}{\Psi}^{\mathcal{N}_{g}}
	([x],l)
	&\stackrel{!}{=}
	\sum_{k}
	\inner{\Phi(k)}{\Psi(k+l)}^{{N}_{g}}_{C(\mathbb{T})}([x+k\theta])
	.
\end{align*}
This is satisfied if we define
\begin{equation}\label{formula:inner-prod-before-descent}
	\inner{\phi}{\psi}^{{N}_{g}^{0}}_{C(\mathbb{T})} ([x])
	 :=
	\int_{\mathbb{R}}
		(\overline{\phi}\psi)
	\left(
		[x-r],
		\tfrac{r}{ b }
	\right)
	\mathrm{d}\, r
	.
\end{equation}
		
In Theorem \ref{thm:N-d-N=Cycle} below, we will sum up what we have found so far, namely the formulas for the lift via descent of the module $\mathcal{N}_{g}^0$.

\subsection{Conclusion of the proof}

\begin{theorem}\label{thm:N-d-N=Cycle}
	Suppose $b\in \mathbb{Z}^{\times}$. We define the structure of an equivariant, right-pre-Hilbert $C(\mathbb{T})$-bimodule on $C_c (\mathbb{T} \times\mathbb{R} )$ by
	\begin{alignat*}{2}
			\phi, \psi\in C_c(\mathbb{T} \times\mathbb{R}):
			&&\qquad
			\inner{\phi}{\psi}_{C(\mathbb{T})} ([x])
			&=
			\int_{\mathbb{R}}
				(\overline{\phi}\psi)
			\left(
				[x-r],
				\tfrac{r}{ b }
			\right)
			\mathrm{d}\, r
			,
			\\
			\mathbb{Z}\curvearrowright C_c (\mathbb{T} \times\mathbb{R} ):
			&&\qquad
			\left( l.\phi \right)([x],r)
					&=
					\phi([x-l\theta],r +l),
			\\
			C(\mathbb{T}) \curvearrowright C_c (\mathbb{T} \times\mathbb{R} ):
			&&\qquad
			\left(
				f \ast \phi
			\right)
			([x],r)
			&=
			f([x])
			\phi([x],r)
			,
			\\
			C_c (\mathbb{T} \times\mathbb{R} ) \curvearrowleft C(\mathbb{T}):
			&&\qquad
			(\phi \ast f)([x],r)
			&=
			\phi([x],r)
			f([x+ r { b }])
	.
	\end{alignat*}
	
	Let ${N}_{b}^{\pm}$ be the completion of $C_c (\mathbb{T} \times\mathbb{R} )$ with respect to this pre-inner product, and let ${N}_{b} := {N}_{b}^{+}\oplus {N}_{b}^{-}$ be standard evenly graded.
	Define the unbounded operator $d_{{N}_{ b },+}\!:\,{N}_{ b }^{+}\to {N}_{ b }^{-}$ 
	by
	\begin{equation}\label{def:dN}
		d_{{N}_{ b },+}
		:= 
		-
	 	\tfrac{1}{ b }	\tfrac{\partial\;}{\partial r}
 		+
	 	\tfrac{\partial\;}{\partial \Theta}
 		-2\pi \mathsf{M}
		,
	\end{equation}
	let $d_{{N}_{ b },-} := d_{{N}_{ b },+}^{*}$ and define
	\begin{equation}
		d_{{N}_{ b }}
		 := 
		\mat{
			0 & d_{{N}_{ b },-}\\
			d_{{N}_{ b },+} & 0
		}
		.
	\end{equation}
	
	Then
	the pair $({N}_{ b }, d_{{N}_{ b }})$ is a cycle in $\Psi^\mathbb{Z} \left( {C(\mathbb{T}_{})},{C(\mathbb{T}_{})}\right)$.
\end{theorem}

\begin{remark}
	 To see why we chose this (pre-)Hilbert module structure, see Formula \eqref{formula:ZZ-and-left-action-before-descent},
	Formula \eqref{formula:right-action-before-descent}, and Formula \eqref{formula:inner-prod-before-descent}.
	To see why we chose this operator, see the proof of Lemma~\ref{lem:KaspProd:Bdd}.
\end{remark}

To prove Theorem \ref{thm:N-d-N=Cycle}, we will check that $({N}_{ b }, b \cdot d_{{N}_{ b }})$ is unitarily equivalent to the equivariant cycle 
$({H}_{ b }, \mathrm{id}_{C(\mathbb{T})}\otimes d_{\lambda})$ of Remark~\ref{rmk:Hb-db-replaced-by-Hb-dlambda} for $\lambda:= 2\pi b \in \mathbb{R}^{\times}$. Recall that we defined \(H_b^{\pm}\) as the completion of 
 $C_c (\mathbb{T}\times\mathbb{R})$ with respect to pre-Hilbert module structure given on page~\pageref{page:Hb-module-structure}, which is also where the definition of $d_{\lambda}$ can be found. Note that the domain of $\mathrm{id}_{C(\mathbb{T})}\otimes d_{\lambda}$ contains, by definition, the subspace $C(\mathbb{T})\odot \mathcal{S}(\mathbb{R})$.

\begin{proof}[Proof of Theorem~\ref{thm:N-d-N=Cycle}]
	Define $w,w\inv \colon \mathbb{T}\times\mathbb{R} \longrightarrow \mathbb{T}\times\mathbb{R}$ by
	\begin{align*}
			w([x],r) 		:= \left(\left[ x + { b } r\right], - r\right)
		\and
			w\inv([x],r)	:= \left(\left[ x + { b } r \right], 	 - r \right)
		,
	\end{align*}
	so that $w\circ w\inv = w\inv\circ w = \mathrm{id}$, and let
	\[
	\begin{tikzcd}
		{H_{ b }^{\pm}}\supseteq C_c (\mathbb{T}\times\mathbb{R}) \ar[rr, "W", shift left=1ex] && \ar[ll, "W\inv", shift left=1ex] C_c (\mathbb{T}\times\mathbb{R})\subseteq {N}_{ b }^{\pm} \\[-15pt]
		W\inv \phi := \sqrt{\abs{b}}\cdot \phi \circ w\inv && W \phi := \frac{1}{\sqrt{\abs{b}}}\cdot \phi \circ w
	\end{tikzcd}
	\]
	It is quickly checked that this induces the claimed structure on $H_{ b }^{\pm}$. 
	
	Let us check that $d_{\lambda,+}$ is induced by $W$, \emph{i.e.}\ 
	\[
		d_{\lambda,+} (\phi)
		\stackrel{!}{=}
		\left(
			W\inv \circ b\cdot d_{N,+} \circ W
		\right)
		\phi
		=
		\left(	
			b\cdot d_{N,+}
			\left(
				\phi \circ w
			\right)
		\right) \circ w\inv,
	\]
	where we abused notation and stopped writing $\mathrm{id}_{C(\mathbb{T})}$.
		
	Note that, if $\Omega$ is a chart of $\mathbb{T} \times\mathbb{R}$, then
	\begin{equation}
		\left( W\inv \circ \tfrac{\partial\;}{\partial \Omega^{i}} \circ W \right) (\phi) (p)
		=
		\tfrac{\partial(\phi\circ w)}{\partial \Omega^{i}}\vert_{ w\inv(p)}
		=
		\tfrac{\partial (\phi\circ w)\circ \Omega\inv}{\partial x^i}\vert_{ (\Omega\circ w\inv)(p)}
		=
		\tfrac{\partial \phi}{\partial (\Omega \circ w\inv)^i}\vert_{p}
		=
		\tfrac{\partial \phi}{\partial \tilde{\Omega}^i}\vert_{p},
	\end{equation}
	where $\tilde{\Omega}:=\Omega\circ w\inv$. We know by a general formula that
	\[	
		\Mat{
			\tfrac{\partial \;}{\partial \tilde{\Omega}^1}
			\\{}\\
			\tfrac{\partial \;}{\partial \tilde{\Omega}^{2}}
		}
		=
		\Mat{
			\tfrac{\partial {\Omega}^1}{\partial \tilde{\Omega}^1} & \tfrac{\partial \Omega^{2}}{\partial \tilde{\Omega}^1}
			\\{}&{}\\
			\tfrac{\partial \Omega^1}{\partial \tilde{\Omega}^{2}} & \tfrac{\partial \Omega^{2}}{\partial \tilde{\Omega}^{2}}
		}
		\,
		\Mat{
				\tfrac{\partial\;}{\partial \Omega^{1}}
				\\{}\\
				\tfrac{\partial\;}{\partial \Omega^{2}}
			}
		.
	\] 
	As	$
		\Omega\circ \tilde{\Omega}\inv (x,r)
		=
		\left( x + { b } r + n(x,r), - r\right)
	$ for some locally constant, $\mathbb{Z} $-valued function $n(x,r)$, 
	we have
	\[
		\tfrac{\partial \Omega^1}{\partial \tilde{\Omega}^i}\vert_{ p}
		=
		\tfrac{\partial (x^1\circ\Omega\circ \tilde{\Omega}\inv)}{\partial x^i}{\vert_{\tilde{\Omega}(p)}}
		=
		\left\{
			\begin{array}{ll}
				1 & \mbox{if } i=1, \\
				 { b } & \mbox{if } i=2,
			\end{array}
		\right.
	\]
	and
	\[
		\tfrac{\partial \Omega^{2}}{\partial \tilde{\Omega}^i}\vert_{ p}
		=
		\tfrac{\partial (x^{2}\circ\Omega\circ \tilde{\Omega}\inv)}{\partial x^i}\vert_{\tilde{\Omega}(p)}
		=
		\left\{
			\begin{array}{ll}
				0 & \mbox{if } i=1, \\
				-1 &\mbox{if } i=2,
			\end{array}
		\right.
	\]
	so that	
	\begin{equation}
		W\inv
		\circ \Mat{
			\tfrac{\partial\;}{\partial \Theta}
			\\{}\\
			\tfrac{\partial\;}{\partial r}
		}
		\circ W
		=
		\Mat{
					\tfrac{\partial \;}{\partial \tilde{\Omega}^1}
					\\{}\\
					\tfrac{\partial \;}{\partial \tilde{\Omega}^{2}}
				}
				=
				\Mat{
					\tfrac{\partial {\Omega}^1}{\partial \tilde{\Omega}^1} & \tfrac{\partial \Omega^2}{\partial \tilde{\Omega}^1}
					\\{}&{}\\
					\tfrac{\partial \Omega^1}{\partial \tilde{\Omega}^{2}} & \tfrac{\partial \Omega^2}{\partial \tilde{\Omega}^{2}}
				}
				\,
				\Mat{
						\tfrac{\partial\;}{\partial \Omega^{1}}
						\\{}\\
						\tfrac{\partial\;}{\partial \Omega^{2}}
					}
		=
		\Mat{
			1 & 0 
			\\{}\\
			 { b } & -1
		}
		\Mat{
			\tfrac{\partial\;}{\partial \Theta}
			\\{}\\
			\tfrac{\partial\;}{\partial r}
		}
		.
	\end{equation}	
	Moreover,
	\[
		W\inv
	 	\circ
	 	\mathsf{M}
		\circ 
		W 
		=
		-
		\mathsf{M}
		.
	\]
	We conclude
	\begin{align}
	\begin{split}\label{eq:fromdHtodN}
		W\inv \circ b\cdot d_{{N}_{ b },+} \circ W
		&\stackrel{\text{(\ref{def:dN})}}{=}
		W\inv \circ
		\left(
			-
			\tfrac{\partial\;}{\partial r}
			+
			b\,\tfrac{\partial\;}{\partial \Theta}
			-2\pi b\,\mathsf{M}
		\right)
		 \circ W
		\\&
		=
		-
		\left(
		 { b } \tfrac{\partial\;}{\partial \Theta} - \tfrac{\partial\;}{\partial r}
		\right)
		+
		b\,\tfrac{\partial\;}{\partial \Theta}
		+ 2\pi b\,\mathsf{M}
		\\&
		=
		2\pi b\, \mathsf{M}
		+ \tfrac{\partial\;}{\partial r}
		=
		d_{ 2\pi b ,+}
	\end{split}
	\end{align}
	as claimed.
	
	Since we have proved $(H_{b},d_{\lambda})$ to be an unbounded cycle for any $\lambda\in\mathbb{R}^{\times}$ (see Theorem~\ref{thm:HdH-is-cycle}), it follows that $({N}_{ b }, b\cdot d_{{N}_{ b }})$ and hence $({N}_{ b }, d_{{N}_{ b }})$ are cycles also.
\end{proof}

We will next use a well-known recipe due to Kucerovsky how to determine that a given unbounded $\KK$-cycle is the Kasparov product of two other cycles.

\begin{theorem}\label{thm:j-N-d-N=TheProduct}
	For $({N}_{ b }, d_{{N}_{ b }})$ as defined in Theorem \ref{thm:N-d-N=Cycle} and
	\[
		j\!:\,
		KK^\mathbb{Z}_0 \left({C(\mathbb{T}_{})},	{C(\mathbb{T}_{})}\right)
		\longrightarrow 
		KK_0 ( A_{\theta} , A_{\theta} )
	\]
	the descent map, the cycle $j([({N}_{ b }, d_{{N}_{ b }})]) = (\mathcal{N}_{b} , D_{\mathcal{N}_{b}})$, satisfies all properties in \cite{Kuc:unbddKKprod}, Theorem~\textup{13}, so that $(\mathcal{N}_{b} , D_{\mathcal{N}_{b}})$ represents the Kasparov product 
	$
		(1_{A_{\theta}} \otimes [\module_{b}])
		\otimes_{A_{\theta}^{\otimes 3}}
		(\Delta_{\theta} \otimes 1_{A_{\theta}})
		.
	$
\end{theorem}

We have already found that the module ${N_{ b }}$ descends to $\mathcal{N}_{ b } = \mathcal{N}_{ b }^{0}\oplus\mathcal{N}_{ b }^{0}$ - in fact, this is where the formulas that we used to define ${N_{ b }}$ came from, see Formula \eqref{formula:ZZ-and-left-action-before-descent},
Formula \eqref{formula:right-action-before-descent}, and Formula \eqref{formula:inner-prod-before-descent}.
Furthermore, we have seen that $\mathcal{N}_{ b }^{\pm} $ can be regarded as $( A_{\theta} \otimes \module_{g})
	 \otimes_{A_{\theta}^{\otimes 3}} 
	( L^{2} \otimes A_{\theta} )$, two copies of which make up the module underlying $(1_{A_{\theta}} \otimes [\module_{b}]) \otimes_{A_{\theta}^{\otimes 3}} (\Delta_{\theta} \otimes 1_{A_{\theta}}) $, via Equations \ref{iso:ALgL2AtoLgL2} and \ref{formula:general:scr-h-in-EL}. The identification can be summed up as follows:
\begin{equation}\label{iso:H1otH2=H}
\underbrace{
	\left( V^{l_{1}} U^{k_{1}} \otimes \Phi \right) \otimes_B
	\left(
		(z^{l_{2}}\otimes \varepsilon_{k_{2}})\otimes V^{l_{3}} U^{k_{3}}
	\right)}_{
	\in \,( A_{\theta} \otimes \module_{g})
		 \otimes_{A_{\theta}^{\otimes 3}} 
		( L^{2} \otimes A_{\theta} )}
	\triangleq
	\lambda^{l_{1}(k_{1}+k_{2})}\,
	\underbrace{
		\Phi ._{ \module_{b} } (V^{l_{1}+l_{2}}U^{-(k_{1}+k_{2})}\otimes V^{l_{3}} U^{k_{3}})
	}_{\in\,\mathcal{N}_{ b }^{\pm}}
\end{equation}
We have also already proved that $(\mathcal{N}_{ b },D_{\mathcal{N}_{b}})$ is indeed in $\Psi(A, C)$.
Therefore, we now only need to prove the following:
\begin{lemma}\label{lem:KaspProd:Bdd}
	For any $x\in C_{c}^{\infty}(\mathcal{A}_{\theta})\odot C_{c}^{\infty}(\mathcal{Z}_{b})\subseteq A_{\theta} \otimes \module_{b} $, the operator 
	\[
		\comm{
			\mat{
				D_{\mathcal{N}_{b}} & 0\vphantom{T_x}
				\\
				0\vphantom{T_x^*} & d_{\Delta}\otimes 1
			}
		}
		{
			\mat{
				0 & T_x
				\\
				T_x^{*} & 0
				}
		}
	\]
	extends to a bounded operator.
\end{lemma}

 We note that $C_{c}^{\infty}(\mathcal{A}_{\theta})\odot C_{c}^{\infty}(\mathcal{Z}_{b})$ is dense in $A_{\theta}\otimes \module_{b}$ by the following:
\begin{lemma}\label{lem:sup-conv-implies-Lb-convergence}\label{lem:sup-conv-implies-mcHb-convergence}
	Suppose $\Phi_{n}\in C_{c}(\mathbb{Z}\times\mathbb{T}\times\mathbb{R})$ are such that $\norm{\Phi_{n}}_{\infty} \stackrel{n\to\infty}{\longrightarrow } 0$ and that, for all $n$, the support of $\Phi_{n}$ is contained in some compact set. Then $\Phi_{n} \stackrel{n\to\infty}{\longrightarrow } 0$ both in $\module_{b}$ and in $\mathcal{H}_{-b}^{\pm}$.
\end{lemma}
The proof employs a ``standard trick'' that was used in \cite[Proof of Thm.~2.8]{MRW:Grpd}: the inductive limit topology on $C_{c} (\mathcal{G})$ for $\mathcal{G}$ a second countable locally compact Hausdorff {\'e}tale groupoid is finer than the topology given by the C*-norm (see 
	\cite[Chapter II, Prop.\ 1.4(i)]{Renault:gpd-approach}
	).
	
\begin{corollary}\label{cor:smooth-dense-in-Lb}\label{cor:smooth-dense-in-mcHb}
	If $c_{00}$ denotes the space of bi-infinite sequences which are eventually zero, then the subspace $c_{00} \odot \mathrm{span}\{ z^{n} \,\vert\, n\in\mathbb{Z}\} \odot C_{c}^{\infty}(\mathbb{R})$ is dense in both $\module_{b}$ and $\mathcal{H}_{-b}^{\pm}.$
\end{corollary}

\begin{remark}
	The statement in Lemma~\ref{lem:KaspProd:Bdd} implicitly makes use of the identification in Equation~\eqref{iso:H1otH2=H}. In other words, our claim (for the creation part) can be rephrased to saying that the following diagram is commutative up to adjointable operators, where $\mathcal{N}^{1} := A_{\theta} \otimes [\module_{ b }]$ and $\mathcal{N}^{2} := ( L^{2} \oplus L^{2} )\otimes A_{\theta}$:
	\begin{equation}
		\begin{tikzcd}[row sep = tiny]\label{diag:meaning-of-lem:KaspProd:Bdd}
			& \mathcal{N}^{1} \otimes_{ A_{\theta}^{3\otimes} } \mathcal{N}^{2} \ar[r, "\text{Eq.~\eqref{iso:H1otH2=H}}"] & \mathcal{N}_{ b } \ar[rd ,"D"]	& \\[-6pt]
			\mathcal{N}^{2}\ar[ur, "T_{x}", end anchor={[yshift=1ex]}]\ar[dr, " d_{\Delta_{\theta}}\otimes 1"']	&				&					& \mathcal{N}_{ b }\\[-4pt]
			& \mathcal{N}^{2} \ar[r, "T_{x}"'] & \mathcal{N}^{1} \otimes_{ A_{\theta}^{3\otimes} } \mathcal{N}^{2} \ar[ru,  "\text{Eq.~\eqref{iso:H1otH2=H}}"', near start]& 
		\end{tikzcd}
	\end{equation}
\end{remark}

We observe that only the creation-part has to be shown:
\begin{lemma}\label{lem:CreationIsEnough}
	Suppose $D\!:\, \dom(D) \to \mathcal{N}$ and $D'\!:\, \dom (D') \to \mathcal{N}'$ are two self-adjoint, densely defined unbounded operators on right-Hilbert $C^*$-modules $\mathcal{N}$ resp.\ $\mathcal{N}'$ over some $C^*$-algebra $C$. If $T\in \mathcal{L} (\mathcal{N}', \mathcal{N})$ is such that $\dom(D T)\cap \dom(D')$ is also dense, and if the operator $D T + T D'$ \textup(or $D T - T D'$\textup) extends to a bounded operator, then so does $T^{*} D + D' T^{*}$ \textup(resp.\ $D T - T D'$\textup).
\end{lemma}

\begin{proof}
	Let $S:= T^{*} D \pm D' T^{*}$ and $R:= D T \pm T D'$, so that $\dom(R) = \dom(D T)\cap \dom (D')$ is dense by assumption. We claim that $R^{*}$ extends $S$.
	
	Let $\xi\in \mathcal{N}$ be an element of $\dom(S)$, that is $\xi\in\dom(D)$ and $T^{*}\xi \in \dom (D')$. In order for $\xi$ to be in $\dom(R^{*})$, we need that the map
	\[
		\mathcal{N}' \supseteq \dom(R) \ni\; \zeta \longmapsto \inner{R\zeta}{\xi}^{\mathcal{N}}_{C}\;\in C
	\]
	is bounded.
	We compute for $\zeta\in \dom(R)$
	\begin{align*}
		\inner{R\zeta}{\xi}^{\mathcal{N}}_{C}
		&=
		\inner{D T\zeta}{\xi}^{\mathcal{N}}_{C}
		\pm
		\inner{T D'\zeta}{\xi}^{\mathcal{N}}_{C}
		=
		\inner{T\zeta}{D\xi}^{\mathcal{N}}_{C}
		\pm
		\inner{D'\zeta}{T^{*} \xi}^{\mathcal{N}'}_{C}
		\\
		&=
		\inner{\zeta}{T^{*}D\xi}^{\mathcal{N}'}_{C}
		\pm
		\inner{\zeta}{D' T^{*} \xi}^{\mathcal{N}'}_{C}
		=
		\inner{\zeta}{S \xi}^{\mathcal{N}'}_{C}.
		\notag
	\end{align*}
	As $S\xi$ is a fixed element of $\mathcal{N}'$, the map $\zeta \mapsto \inner{R\zeta}{\xi}^{\mathcal{N}}_{C} = \inner{\zeta}{S \xi}^{\mathcal{N}'}_{C}$ is bounded. We have shown $\dom(S)\subseteq \dom(R^{*})$ and also that for any $\zeta\in\dom(R)$ and $\xi\in\dom(S)$,
	\[
		\inner{\zeta}{S \xi}^{\mathcal{N}'}_{C} = \inner{R\zeta}{\xi}^{\mathcal{N}}_{C} = \inner{\zeta}{R^{*} \xi}^{\mathcal{N}'}_{C}
		.
	\]
	We know that this property uniquely defines $R^{*}\xi$ since $\dom(R)$ is dense, and hence $R^{*} \xi = S\xi$ on $\dom(S)$. In other words, $R^{*}$ extends $S$.
	
	We now only need to see that $R^{*}$ is defined everywhere (which then makes it a bounded operator), so that $S$ indeed has a bounded extension: let $\overline{R}$ be the assumed bounded extension of $R$. Then for $\xi\in\mathcal{N}$ and $\zeta\in\dom(R)$, we have
	\[
		\norm{
			\inner{R\zeta}{\xi}^{\mathcal{N}}_{C}
		}
		\leq 
		\norm{R\zeta}\cdot\norm{\xi}
		\leq
		\norm{\overline{R}}\cdot\norm{\zeta}\cdot\norm{\xi}.
	\]
	Therefore, the map
	\[
		\dom(R) \ni\; \zeta \longmapsto \inner{R\zeta}{\xi}^{\mathcal{N}}_{C}\;\in C
	\]
	is a bounded operator for \textit{any} $\xi\in\mathcal{N}$, which means $\mathcal{N}\subseteq \dom(R^{*})$, so $R^{*}$ is defined everywhere.
\end{proof}

\begin{proof}[Proof of Lemma~\ref{lem:KaspProd:Bdd}]
	By linearity, it suffices to prove the claim for an elementary tensor $x=a\otimes \Phi$ in $ C_{c}^{\infty}(\mathcal{A}_{\theta}) \odot C_{c}^{\infty}(\mathcal{Z}_{b})$.

	Let us untangle Diagram~\ref{diag:meaning-of-lem:KaspProd:Bdd} and be precise: Instead of working with $D_{\mathcal{N}_{b}}$ on $\mathcal{N}_{b}$, we will work with the corresponding operator $\tilde{D}$ on the \emph{actual} space $\mathcal{N}^{1}\otimes_{ A_{\theta}^{3\otimes} }\mathcal{N}^{2}$ (using Equation~\eqref{iso:H1otH2=H} to figure out $\tilde{D}$). Unfortunately, $\tilde{D}$ is going to be very unwieldy, which is the reason we instead chose to define $D$'s lift in Theorem~\ref{thm:N-d-N=Cycle}. The upshot is that $\tilde{D}T_x - T_x (d_{\Delta}\otimes 1)$ will turn out to be a creation operator, so that it is clearly adjointable.
	
	Note that, since $C_{c}^{\infty} (\mathcal{Z}_{b})$ is a subspace of $\dom(D)$ which contains $C_{c}^{\infty} (\mathcal{Z}_{b}) ._{\module_{b}} ( C_{c}^{\infty}(\mathcal{A}_{\theta}) )^{2\odot}$, the map in Equation~\eqref{iso:H1otH2=H} shows that
	\begin{equation}\label{eq:dom-tilde-D}
		\bigl(
			C_{c}^{\infty}(\mathcal{A}_{\theta}) \odot C_{c}^{\infty} (\mathcal{Z}_{b})
		\bigr)
			\odot_{\mathfrak{A}^{3\odot}}
		\bigl(
			C_{c}^{\infty}(\mathcal{A}_{\theta}) \odot C_{c}^{\infty}(\mathcal{A}_{\theta}) 
		\bigr)
		\subseteq
		\dom(\tilde{D}_{\pm})
		.
	\end{equation}
	
	For $a\in C_{c}^{\infty} ( \mathcal{A}_{\theta})\subseteq A_{\theta}$ and $f\in C_{c}^{\infty} ( \mathcal{A}_{\theta})\subseteq (L^2 \oplus L^2)^{\pm}$, define the function $\psi(a,f)\in C_{c}^{\infty} ( \mathcal{A}_{\theta})$ by
	\[
		\psi(a,f) \, ([x],k)
		:=
		\sum_{n\in\mathbb{Z}} a([x-k\theta],-n)\, f([x],n-k),
	\]
	then for $a= V^{l_{1}}U^{k_{1}}$ and $f=z^{l_{2}}\otimes \varepsilon_{k_{2}}$, we recover $
		\psi(a,f)
		=
		\lambda^{l_{1}(k_{1}+k_{2})}\,	
		V^{l_{1}+l_{2}}U^{-(k_{1}+k_{2})}.	
	$
	This shows that, for $c\in C_{c}^{\infty} ( \mathcal{A}_{\theta})$ and $\Phi \in C_{c}^{\infty} (\mathcal{Z}_{b})$, the map in Equation~\eqref{iso:H1otH2=H} identifies
	\begin{alignat}{3}\label{iso:H1otH2=H-2nd-version}
	\mathcal{N}^{1} \otimes_{ A_{\theta}^{3\otimes} } (\mathcal{N}^{2})^{\pm}
		&&\longleftrightarrow&&&
		\mathcal{N}_{ b }^{\pm}
				\nonumber
		\\
		(
			a\otimes \Phi
		)
		\otimes_{ A_{\theta}^{3\otimes} }
		(
			f \otimes c
		)
		&&\stackrel{\hphantom{\longleftrightarrow}}{\triangleq}&&&
			\Phi ._{\module_{ b } } \bigl(\psi(a,f)\otimes c\bigr) 
	\end{alignat}

	To find $\tilde{D}_{\pm} \bigl( (a\otimes \Phi) \otimes_{ A_{\theta}^{3\otimes} } (f\otimes c)\bigr)$, we see from Equation~\eqref{iso:H1otH2=H-2nd-version} that we first need to compute $D_{\mathcal{N}_{b},\pm} \Bigl( \Phi._{\module_{b}} \bigl(\psi(a,f)\otimes c\bigr) \Bigr)$. Note that, if $\xi := \psi(a,f)$ and $\Psi := \Phi._{\module_{b}} \xi$, then Equation~\eqref{eq:rightActionOnE0-general} (the formula for the right action on $\module_{b}$) reveals that
	\begin{align*}
		\frac{\partial \Psi}{\partial r}
		&\,=\,
		\left(\frac{\partial \Phi}{\partial r}\right)._{\module_{b}} \xi
		+
		b\, \Phi._{\module_{b}} \left( \psi(a,f) \otimes \frac{\partial c}{\partial \Theta}\right),
		\\
		\frac{\partial \Psi}{\partial \Theta}
		&\,=\,
		\left(\frac{\partial \Phi}{\partial \Theta}\right)._{\module_{b}} \xi
		+
		\Phi._{\module_{b}} \left( \psi(a,f) \otimes \frac{\partial c}{\partial \Theta} + 
		\frac{\partial \psi(a,f)}{\partial \Theta}
		\otimes c\right),
		\and 
		\\
		\mathsf{M}^{\mathbb{R}} \Psi
		&\,=\,
		\left(\mathsf{M}^{\mathbb{R}}\Phi\right)._{\module_{b}} \xi
		+
		\Phi._{\module_{b}} \left(\left( \mathsf{M}^{\mathbb{Z}}\psi(a,f)\right) \otimes c\right),
	\end{align*}
	where $\mathsf{M}^{\mathbb{R}}$ resp.\ $\mathsf{M}^{\mathbb{Z}}$ denotes the operator that multiplies by the input of the $\mathbb{R}$- resp.\ the $\mathbb{Z}$-component, and $\frac{\partial\;}{\partial r}$ resp.\ $\frac{\partial\;}{\partial \Theta}$ refers to differentiation with respect to the $\mathbb{R}$- resp.\ $\mathbb{T}$-component.
	
	Therefore, applying the operator $D:= D_{\mathcal{N}_{b}}$ on $\mathcal{N}_{b}$ --built out of $d_{N_{b}}$ (see Definition~\ref{def:dN}) via descent-- to $\Psi$ yields
	\begin{align*}
		D_{\pm} (\Psi)
		\;&=\;
		\left( \mp \frac{1}{b}\frac{\partial\;}{\partial r} \pm \frac{\partial\;}{\partial \Theta} - 2\pi \mathsf{M}^{\mathbb{R}}\right) (\Psi) 
		\\
		&=\;
		\left[
			\mp \frac{1}{b}\frac{\partial\;}{\partial r} \pm \frac{\partial\;}{\partial \Theta}
			- 2\pi \mathsf{M}^{\mathbb{R}}
		\right]
		(\Psi)
		._{\module_{b}} \xi
		+
		\Phi._{\module_{b}} 
		\left[
			\left(
				\pm \frac{\partial\;}{\partial \Theta} 
				-
				2\pi\, \mathsf{M}^{\mathbb{Z}}
			\right)
			\bigl(\psi(a,f)\bigr)
			\otimes c
		\right].
	\end{align*}
	Using the definition of $\psi(a,f)$, we compute
	\[
		-\mathsf{M}^{\mathbb{Z}} \psi(a,f)
		=
		\psi( \mathsf{M}^{\mathbb{Z}} s, f)
		+
		\psi(a, \mathsf{M}^{\mathbb{Z}} f)
		\and 
		\frac{\partial \psi(a,f)}{\partial \Theta}
		=
		\psi\left( \frac{\partial a}{\partial \Theta},f\right)
		+
		\psi\left( a, \frac{\partial f}{\partial \Theta}\right),
	\]
	so that all in all:
	\begin{align*}
		D_{\pm} (\Psi)
		\;&=\;
		\left[
			\mp \frac{1}{b} \frac{\partial\;}{\partial r}
			\pm \frac{\partial\;}{\partial \Theta}
			-
			2\pi\, \mathsf{M}^{\mathbb{R}}
		\right]
		(\Phi)._{\module_{b}} \xi
		\\
		&\qquad
		+
		\Phi._{\module_{b}}
		\left[
			\psi
			\left(
				\pm \frac{\partial a}{\partial \Theta}
				+ 2\pi\, \mathsf{M}^{\mathbb{Z}}a,
				f
			\right)
			+
			\psi
			\left(
				a, 
				\pm \frac{\partial f}{\partial \Theta}
				+ 2\pi\, \mathsf{M}^{\mathbb{Z}}f,
			\right)
		\right]\otimes c.
	\end{align*}
	This element corresponds via Equation~\eqref{iso:H1otH2=H-2nd-version} to the following element in $\mathcal{N}^{1} \otimes_{ A_{\theta}^{3\otimes} } (\mathcal{N}^{2})^{\pm}$:
	\begin{align*}
	&
			\tilde{D}_{\pm} \bigl( \left(a\otimes \Phi\right) \otimes_{ A_{\theta}^{3\otimes} } \left(f\otimes c\right)\bigr)
			\\
			&\quad:=
			\left(
				 a\otimes \left[-2\pi\,\mathsf{M}^{\mathbb{R}}\mp\frac{1}{b}\frac{\partial\;}{\partial r} \pm \frac{\partial\;}{\partial \Theta}\right](\Phi)
				+
				\left[2\pi\, \mathsf{M}^{\mathbb{Z}} \pm\frac{\partial\;}{\partial \Theta}\right](a)\otimes \Phi
			\right) \otimes_{ A_{\theta}^{3\otimes} } \left(f\otimes c\right)
			\\
			&\qquad
			+
			\left(a\otimes \Phi\right) \otimes_{ A_{\theta}^{3\otimes} }
			\left(\left[
				2\pi \, \mathsf{M}^{\mathbb{Z}}
				\pm \, \frac{\partial\;}{\partial \Theta}
			\right] (f)
			\otimes c
			\right)
	\end{align*}

	Notice that, since $\Dirac_{\mathbb{T}} = -\imaginary\frac{\partial\;}{\partial \Theta}$ and $\Dirac_{\mathbb{Z}} = 2\pi\, \mathsf{M}^{\mathbb{Z}}$ (defined in Lemma~\eqref{def:Delta}), we get
	\[
		\left[
			2\pi \, \mathsf{M}^{\mathbb{Z}}
			\pm \, \frac{\partial\;}{\partial \Theta}
		\right] (f) \otimes x
		=		
		(1\otimes \Dirac_{\mathbb{Z}} \pm \imaginary\ \Dirac_{\mathbb{T}}\otimes 1)(f) \otimes c
		=
		\bigl(d_{ \Delta_{\theta}, \pm}\otimes 1 \bigr)
		\left(
			f\otimes c
		\right)
		=:
		D_{2, \pm}
		\left(
			f\otimes c
		\right)
		.
	\]
	
	Thus, if we define for $x=a\otimes \Phi$:
	\begin{align*}
		X_{\pm} (x) :=
		a\otimes \left[-2\pi\,\mathsf{M}^{\mathbb{R}}\mp\frac{1}{b}\frac{\partial\;}{\partial r} \pm \frac{\partial\;}{\partial \Theta}\right](\Phi)
		+
		\left[2\pi\, \mathsf{M}^{\mathbb{Z}} \pm\frac{\partial\;}{\partial \Theta}\right](a)\otimes \Phi
		 \ \in \
		 C_{c}^{\infty}(\mathcal{A}_{\theta}) \odot C_{c}^{\infty}(\mathcal{Z}_{b}),
	\end{align*}	
	then this shows that
	\begin{align*}
		\tilde{D}_{\pm} \bigl( T_{x}(f\otimes c)\bigr)
		=& 
		T_{X_{\pm} (x)}(f\otimes c)		
		+
		T_{x} D_{2,\pm} 
		\left(
			f\otimes c
		\right)	\bigr)
		.
	\end{align*}
	We conclude that $\tilde{D}_{\pm} T_{x} - T_{x} D_{2,\pm} = T_{X_{\pm} (x)}$ is an adjointable operator.
				
	If we can invoke Lemma~\ref{lem:CreationIsEnough}, then we do not have to deal with the annihilation part. The only thing we need to check is that the set $\dom(\tilde{D} T_x)\cap \dom(D_{2})$ is dense. The domain of $D_{2,\pm} = d_{\Delta_{\theta},\pm}\otimes 1$ contains $\mathfrak{A} \odot \mathfrak{A} \subseteq (L^{2}\oplus L^{2})^{\pm} \otimes A_{\theta}$. In particular, if $f\in C_{c}^{\infty}(\mathcal{A}_{\theta}) \subseteq \mathfrak{A} \subseteq (L^{2}\oplus L^{2})^{\pm}$ and $c\in C_{c}^{\infty}(\mathcal{A}_{\theta})\subseteq A_{\theta}$, then for $x=a\otimes \Phi \in C_{c}^{\infty}(\mathcal{A}_{\theta}) \odot C_{c}^{\infty}(\mathcal{Z}_{b})$, Equation~\eqref{eq:dom-tilde-D} shows that $
	T_{x} (f\otimes c)
	\in \dom (\tilde{D})$ always. Thus, $\dom(\tilde{D} T_x)\cap \dom(D_{2})$ contains $C_{c}^{\infty}(\mathcal{A}_{\theta})\odot C_{c}^{\infty}(\mathcal{A}_{\theta})$, which is indeed dense in $(L^{2}\oplus L^{2})^{\pm} \otimes A_{\theta}$.
\end{proof}

\section{A spectral cycle representative of the unit}

As a result of the previous sections, we have obtained the following, where we use that $\twist_{b}\inv = \twist_{-b}$ by Theorem~\ref{theorem:thetwistedindextheorem}.

\begin{theorem}
	Let \(g = \begin{bsmallmatrix}1 & b\\0 & 1\end{bsmallmatrix}\) \emph{for $b>0$} and \(\module_b:= \module_g\) \textup(Definition 
	\ref{def:Lg}\textup). Then 
	\begin{equation*}
		(1_{A_{\theta}} \otimes \twist_{-b})_*([\module_{ b }])
		\otimes_{A_{\theta}^{\otimes 3}}
		(\Delta_{\theta} \otimes 1_{A_{\theta}}) = 1_{A_\theta}
		,	\end{equation*}
		where \(\twist_b\in \KK_0(A_\theta, A_\theta)\) is the \(b\)-twist 
		\textup(Definition \ref{definition:btwist}\textup). In particular, the class 
		\[ \widehat{\Delta}_\theta := (1_{A_\theta}\otimes \twist_{-b})_*(\module_b) \]
		together with Connes' class \(\Delta_{\theta}\), satisfies the first zig-zag equation. The 
		classes \(\Delta_\theta, \widehat{\Delta}_\theta\) are the co-unit and unit, respectively,
		of a self-duality for \(A_\theta\). 
\end{theorem}

The co-unit of Connes' duality is of course 
represented by a spectral triple, and in this section we 
describe a spectral cycle representative for $\Dudelta_{\theta}$. 

First, recall that $\twist_{ -b }$ can be described as the descended version of the cycle $(H_{ -b }, d_{1})$, \emph{i.e.}\ 
\[
	\twist_{-b}
	=
	j([ (H_{ -b }, d_{1}) ])
	=:
	[(\mathcal{H}_{-b}, D_{1})],
\]
Thus, its module is a completion of $C_{c} (\mathbb{Z}\times \mathbb{T}\times\mathbb{R})$, described explicitly in Lemma~\ref{lem:description-of-mcHb} below, and its operator $D_{1 } = \mat{0 & D_{1,-} \\ D_{1,+} & 0}$ is given by 
\[
	D_{1,\pm} = \mathsf{M} \pm \tfrac{\partial \;}{\partial r},
\]
where $\mathsf{M}$ still denotes multiplication by the input of the $\mathbb{R}$-component.
Recall from Remark~\ref{rmk:Hb-db-replaced-by-Hb-dlambda} that we can replace $D_{1}$ by $\frac{1}{2\pi}\cdot D_{\lambda}$ for any $\lambda>0$, so for the best final results, we will choose $\lambda = 2\pi b>0$:
\begin{equation}\label{def:DmcH}
	D_{\mathcal{H}}
	:=
	\frac{1}{2\pi}\cdot D_{2\pi b}
	=
	\Mat{0 & b\, \mathsf{M} - \frac{1}{2\pi}\tfrac{\partial \;}{\partial r} \\ b\, \mathsf{M} + \frac{1}{2\pi}\tfrac{\partial \;}{\partial r} & 0}
	.
\end{equation}

Before we can state the main theorem of this section, we need some notation.
\begin{definition}\label{def:Schw-n}
For a smooth function $F$ on $\mathbb{Z}\times\mathbb{T}\times\mathbb{R}^{n}$, any $N\in\mathbb{N}_{0}$, and $\alpha$ an $n$-multi-index, define the semi-norm
\[
	\normSn{F}{N}{\alpha} :=
	\sup{}{
	\left\{
		\left(\norm{(k,\vec{x})}_{\ell^{N}}^{N} + 1\right) \abs{ \frac{\del^{\alpha} \Phi}{\del \vec{x}^{\alpha}} (k,[v],\vec{x})}
		\,:\,
		(k,[v],\vec{x})\in \mathbb{Z}\times\mathbb{T}\times\mathbb{R}^{n}
	\right\}}
	,
\]
where $\frac{\del^{\alpha} \;}{\del \vec{x}^{\alpha}}$ is differentiation with respect to the $\mathbb{R}$-components. If $\normSn{F}{N}{\alpha}$ is finite for every choice of $N$ and $\alpha$, then $F$ is called a \emph{Schwartz--Bruhat function}. We will denote the locally convex space consisting of such $F$ by $\Schw_{n}$. 
\end{definition}	

\begin{remark}
While it is possible to define a larger family of semi-norms by including differentiation in the $\mathbb{T}$-direction, the above seminorms are sufficient for our goals.
\end{remark}

\begin{definition}\label{def:ops-on-Schw-n}
	For functions on $\mathbb{Z}\times\mathbb{T}\times\mathbb{R}^{n}$, let $\mathsf{M}^{\mathbb{R}}_{i}$ be the operator of multiplication by the input of the $i^{\text{th}}$ $\mathbb{R}$-component, and $\partial_{i}$ differentiation with respect to the $i^{\text{th}}$ $\mathbb{R}$-component. Let $\mathsf{M}^{\mathbb{Z}}$ be the operator of multiplication by the input of the $\mathbb{Z}$-component.
\end{definition}

Note that all of these operators map $\Schw_{n}$ back into itself. We can now state the theorem:
	\begin{theorem}\label{thm:cycle-rep-Dudelta}
		Let $\mathcal{R}^{\pm}$ be the completion of the right-$\mathfrak{A}\odot \mathfrak{A}$ pre-Hilbert module $R^{\infty}:=\Schw_{2}$ whose structure is defined by:
		\begin{alignat}{2}
			\begin{split}\label{eq:right action-on-mcR}
		 	&\bigl( F ._{\mathcal{R}} (V^{l_{1}}U^{k_{1}} \otimes V^{l_{2}}U^{k_{2}}) \bigr) (k,[x],r,s)
		 	\\	
		 	&=
		 	\lambda^{l_{1}(k+k_{1})+ l_{2}k_{2}}\,\e{x(l_{1}+l_{2})}\e{ (l_{2} r - k_{2}s)}
		 	F(k-k_{2} +k_{1}, [x+k_{2}\theta], r, s)
			\end{split}
		\end{alignat}
		and
		\begin{alignat}{2}
		&\inner{F_{1}}{F_{2}}^{\mathcal{R}}
			( l_{1},[v],l_{2}, [w])
			\notag	\\
				\begin{split}\label{eq:inner-product-on-mcR}
				&\quad= \sum_{k_{1}, k_{2}\in\mathbb{Z}}
				\int_{t}
				\e{tl_{2}}\overline{F_{1}}(k_{1}, [v-k_{1}\theta], k_{2}+k_{1}\theta - v + w, t)
			\\
			&\quad\hphantom{=\sum_{k_{1}, k_{2}\in\mathbb{Z}}\int_{t}}
				F_{2}(k_{1}	+l_{2} - l_{1}, [v-(k_{1}	+l_{2})\theta], k_{2}+k_{1}\theta - v + w, t) 
				\,\mathrm{d}\, t.
				\end{split}
		\end{alignat}
		Let $\mathcal{R} := \mathcal{R}^{+}\oplus \mathcal{R}^{-}$ be standard evenly graded and define		
		\begin{alignat}{3}\label{eq:mau-on-mcR}
			d_{\mathcal{R}} := \mat{ 0 & d_{\mathcal{R},-} \\ d_{\mathcal{R},+} & 0}
			\text{ where }
			d_{\mathcal{R},\pm} := &
			\mathsf{M}^{\mathbb{R}}_{1} \mp \imaginary \mathsf{M}^{\mathbb{R}}_{2}
			\text{ with }
			\dom (d_{\mathcal{R},\pm}) :=
			\Schw_{2}
			.
		\end{alignat}
		Then $(\mathcal{R}, d_{\mathcal{R}})$ is a Kasparov cycle and represents $\Dudelta_{\theta}$. In particular, $\Dudelta_{\theta}$ does not depend on the choice of $b\in\mathbb{Z}^\times$.
	\end{theorem}
	
	To prove this, we will make use of the following:
	\begin{theorem}[special case of {\cite[Theorem 7.4]{LeMe:sums-reg-sa-ops}}]\label{thm:LeMe:sums-reg-sa-ops-Thm7.4}
		Let $\mathcal{E}_{b}:= \module_{b} \otimes_{ A_{\theta}^{\otimes 2} } (A_{\theta} \otimes \mathcal{H}_{-b})$ for $b>0$, and suppose we have
		\begin{enumerate}[label=\textup{(\arabic*)}]
			\item\label{item:LeMe-s.a.+reg}
			an odd, self-adjoint, regular operator $D_{\mathcal{E}}\colon \dom(D_{\mathcal{E}}) \to \mathcal{E}_{b}$ so that
			\item\label{item:LeMe-WAP}
			$(0,D_{\mathcal{E}})$ is a weakly anticommuting pair, and
			\item\label{item:LeMe-core}
			a dense $\mathfrak{A}\odot\mathfrak{A}$-submodule $\mathcal{X}\subseteq \module_{b}$ for which the algebraic tensor product \mbox{$\mathcal{X}\odot_{\mathfrak{A}\odot \mathfrak{A}} \dom( 1_{A_{\theta}}\otimes D_{\mathcal{H}})$} is a core for $D_{\mathcal{E}}$ such that
			\item\label{item:LeMe-creation}
			for all $\Phi\in \mathcal{X}$, both operators $\eta\mapsto D_{\mathcal{E},\pm} (\Phi\otimes \eta) - \Phi \otimes (1_{A_{\theta}}\otimes D_{\mathcal{H},\pm})(\eta)$ with domain $\dom(1_{A_{\theta}}\otimes D_{\mathcal{H},\pm})$ extend to adjointable operators $A_{\theta}\otimes \mathcal{H}_{ -b }^{\pm} \to\mathcal{E}_{b}^{\pm}$. 
		\end{enumerate}
		Then $(\mathcal{E}_{b},D_{\mathcal{E}})$ is a Kasparov cycle and represents $\Dudelta_{\theta}$.
	\end{theorem}
	
	Note that Item~\ref{item:LeMe-WAP} is actually true no matter what self-adjoint regular operator $D_{\mathcal{E}}$ is chosen.

	The remainder of this section is structured as follows: First, we find a description of $\mathcal{E}_{b}^{\pm}$ as a completion, called $\mathcal{P}_{b}^{\pm}$, of $C_{c}(\mathbb{Z}\times\mathbb{T}\times\mathbb{R}^{2})$. We will then prove that $\mathcal{E}_{b}^{\pm}$ contains $\Schw_{2}$, Schwartz--Bruhat functions on $\mathbb{Z}\times\mathbb{T}\times\mathbb{R}^{2}$, and explicitly describe the module structure of this subspace. Using a unitary operator, we simplify $\mathcal{E}_{b}$ to the module $\mathcal{R}$ from Theorem~\ref{thm:cycle-rep-Dudelta}. On this easier module, we study the two unbounded operators $d_{\mathcal{R},\pm}\colon \mathcal{R}^{\pm}\to\mathcal{R}^{\mp}$ to then induce them to unbounded operators $D_{\mathcal{E},\pm}\colon \mathcal{E}_{b}^{\pm} \to \mathcal{E}_{b}^{\mp}$. Finally, we will show that the off-diagonal operator $D_{\mathcal{E}}$, built in the usual way out of $D_{\mathcal{E},\pm}$, makes $\mathcal{E}_{b}$ a representative of $\Dudelta_{\theta}$. This will prove Theorem~\ref{thm:cycle-rep-Dudelta}.

		\begin{proposition}[The balancing]\label{lem:mcE-simplified-bumpy-version}
			The module $\mathcal{E}_{b}^{\pm}=\module_{b} \otimes_{A_{\theta}^{\otimes 2}} (A_{\theta} \otimes \mathcal{H}_{ -b }^{\pm})$ underlying $\Dudelta_{\theta}$ has a copy of $P^{\infty}:=C_{c}^{\infty} (\mathbb{Z}\times\mathbb{T}\times\mathbb{R})\odot C_{c}^{\infty} (\mathbb{R})$ as a dense subspace via the following map:
			\begin{equation*}
				\begin{tikzcd}[column sep = small]
					\iota_{0}\colon\quad 
					P^{\infty} \ar[r]
					& \module_{b} \otimes_{A_{\theta}^{\otimes 2}} (A_{\theta} \otimes \mathcal{H}_{ -b }^{\pm})
					\\[-20pt]
					\Phi
					\odot \psi
					\ar[r, mapsto]
					&
					\Phi
					\otimes
					(1_{A_{\theta}} \otimes 
					\varepsilon_{0} \otimes z^{0} \otimes \psi)	
				\end{tikzcd}	
			\end{equation*}
		\end{proposition}

		\begin{proof}
			Any elementary tensor in $ \module_{b} \otimes_{A_{\theta}^{\otimes 2}} (A_{\theta} \otimes \mathcal{H}_{ -b }^{\pm})$ can be written as $\Phi\otimes (1_{A_{\theta}}\otimes \Psi)$ for $\Phi\in \module_{b}$ and $\Psi\in\mathcal{H}_{ -b }^{\pm}$ due to the balancing.
			Moreover, if $k,l\in\mathbb{Z}$, the balancing further gives the following identity in $ \module_{b} \otimes_{A_{\theta}^{\otimes 2}} (A_{\theta} \otimes \mathcal{H}_{ -b }^{\pm})$: 
			\begin{equation}\label{eq:balancing-in-Dudelta}
				\bigl(\Phi ._{ \module_{b}} (1\otimes V^{l} U^{k})\bigr) \otimes (1_{A_{\theta}} \otimes \Psi)
				=
				\Phi \otimes \bigl(1_{A_{\theta}} \otimes ( V^{l} U^{k} ._{\mathcal{H}_{ -b }} \Psi)\bigr).
			\end{equation}

			If we take $\Psi=\varepsilon_{p}\otimes z^{q} \otimes \psi \in C_{c}^{\infty}(\mathbb{Z}\times\mathbb{T}\times\mathbb{R}) \subseteq \mathcal{H}_{ -b }^{\pm}$, then the left action of $V^{l} U^{k}$ on $\mathcal{H}_{-b}$ (explicitly constructed in Equation~\ref{eq:left action-on-mcH}) yields:
			\[
				\bigl(	V^{l} U^{k} ._{\mathcal{H}_{ -b }} \Psi \bigr) (n,[x],r)
				=
				\e{l(x-rb)} \varepsilon_{p} (n-k) \e{q(x-k\theta)} \psi ( r-k ).
			\]
			If we write
			$
				(\kappa_k\psi) (r) := \psi (r-k)
			,$
			then this means
			\[
				V^{l} U^{k} ._{\mathcal{H}_{ -b }} (\varepsilon_{p}\otimes z^{q} \otimes \psi)
				=
				\lambda^{-qk} \varepsilon_{p+k} \otimes z^{l+q} \otimes [\me{lb\,\cdot}\kappa_{k}\psi].
			\]
			In other words, we have the following equality in the even/odd part $ \module_{b} \otimes_{A_{\theta}^{\otimes 2}}	(1_{A_{\theta}}\otimes \mathcal{H}_{ -b }^{\pm})$ of the module $\mathcal{E}_{b}$ of $\Dudelta_{\theta}$ for all $l,k\in\mathbb{Z}$:
			\begin{align*}
				&\bigl(\Phi ._{ \module_{b}} (1\otimes V^{l} U^{k})\bigr) \otimes (1_{A_{\theta}} \otimes \varepsilon_{p}\otimes z^{q} \otimes \psi)
				=
				\Phi \otimes \bigl(1_{A_{\theta}} \otimes \lambda^{-qk} \varepsilon_{p+k} \otimes z^{l+q} \otimes [\me{lb\,\cdot}\kappa_{k}\psi]\bigr)
				.
			\end{align*}
			If we choose $p=q=0$ (which is equivalent to choosing $k,l$), then we have no degree of freedom left. Moreover, replacing $\psi$ by $\e{lb\,\cdot}\kappa_{-k}\psi$ (so that $\me{lb\,\cdot}\kappa_{k}\psi$ becomes $\psi$), we can rephrase the balancing to:
			\begin{align*}
				\Phi \otimes \bigl(1_{A_{\theta}} \otimes \left(\varepsilon_{k} \otimes z^{l} \otimes \psi\right)\bigr)
				=
				&\bigl(\Phi ._{ \module_{b}} (1\otimes V^{l} U^{k})\bigr) \otimes \left(1_{A_{\theta}} \otimes \left(\varepsilon_{0}\otimes z^{0} \otimes \left[\e{lb\,\cdot}\kappa_{-k}\psi\right]\right)\right)
				.
			\end{align*}
			More generally, an element of the form
			\[
				\sum_{k,l} \Phi_{k,l} \otimes \bigl(1_{A_{\theta}} \otimes \varepsilon_{k} \otimes z^{l} \otimes \psi_{k,l}\bigr)
				\quad \in\quad \module_{b} \otimes_{A_{\theta}^{\otimes 2}}	(1_{A_{\theta}}\otimes \mathcal{H}_{ -b }^{\pm}) = \mathcal{E}_{b}
			\]
			equals
			\begin{equation}\label{eq:balancing-yielding-zeroes}
				\sum_{k,l} \bigl((\Phi_{k,l}) ._{ \module_{b}} (1\otimes V^{l} U^{k})\bigr) \otimes \left(1_{A_{\theta}} \otimes \varepsilon_{0}\otimes z^{0} \otimes \left[\e{lb\,\cdot}\kappa_{-k}(\psi_{k,l})\right]\right).
			\end{equation}
			Lastly, recall that by Corollary~\ref{cor:smooth-dense-in-mcHb}, $C_{c}^{\infty}(\mathbb{Z}\times\mathbb{T}\times\mathbb{R})$ is dense in $\mathcal{H}_{-b}^{\pm}$ and $C_{c}^{\infty}(\mathbb{Z}\times\mathbb{T}\times\mathbb{R})$ is dense in $\module_{b}$. Since $(\Phi_{k,l}) ._{ \module_{b}} (1\otimes V^{l} U^{k})$ is in $C_{c}^{\infty}(\mathbb{Z}\times\mathbb{T}\times\mathbb{R})$ if $\Phi_{k,l}$ is, we conclude that indeed, the map $P^{\infty}=C_{c}^{\infty}(\mathbb{Z}\times\mathbb{T}\times\mathbb{R})\odot C_{c}^{\infty}(\mathbb{R})\to \mathcal{E}_{b}$ determined by
			\begin{equation*}
				\begin{tikzcd}[column sep = small, row sep = tiny]
					\Phi 
					\odot \psi
					\ar[r, mapsto]
					&
					\Phi
					\otimes
					(1_{A_{\theta}} \otimes 
					\varepsilon_{0} \otimes z^{0} \otimes \psi)		
				\end{tikzcd}	
			\end{equation*}
			has dense range.
		\end{proof}

		\begin{lemma}\label{lem:P-pre-H-module-structure}
			The space $P^{\infty}=C_{c}^{\infty} (\mathbb{Z}\times\mathbb{T}\times\mathbb{R})\odot C_{c}^{\infty} (\mathbb{R})$ inherits the following structure of a pre-Hilbert right-module from $\mathcal{E}_{b}^{\pm}$ via $\iota_{0}$ \textup(the map in Lemma~\ref{lem:mcE-simplified-bumpy-version}\textup{):} the pre-inner product with values in $C_{c}^{\infty}(\A)$ is given
			for $F_{i}\in P^{\infty}$ by
				\begin{align}
					&\inner{F_{1}}{F_{2}}^{ \mathcal{P} }(l_{1},[v],l_{2},[w])
					\notag
								\\		
					\begin{split}\label{eq:inner-of-P}
						&\quad=
						\sum_{k_{1}, k_{2}\in\mathbb{Z}}
						\int_{\mathbb{R}}
						\overline{F_{1}}
						\left(k_{1}	,[v],\tfrac{k_{2}+k_{1}\theta - v + w }{ b } - r, r\right)
						\\
						&
						\quad\hphantom{\abs{b}	{\sum_{k_{1}, k_{2}\in\mathbb{Z}}}
						\int_{\mathbb{R}}}
						F_{2}
						\left(k_{1}	+l_{2} - l_{1},[v-l_{1}\theta],\tfrac{k_{2}+k_{1}\theta - v + w}{ b } -r + l_{1},r-l_{2} \right) 			
						\,\mathrm{ d }\, r.
					\end{split}
				\end{align}
			The right action of an element $\xi\in \mathfrak{A}\odot\mathfrak{A}$ on $F\in P^{\infty}$ is given by:
			\begin{align}
		 		\begin{split}\label{eq:action-on-P}
		 		\left( F ._{\mathcal{P}} \xi \right) (k,[v],r,s)
		 		=&
		 		\sum_{k_{1},k_{2}}
				 F \left(k-k_{2}+k_{1}, [v+k_{1}\theta], r-k_{1}, s+k_{2}\right)
		 		\\&\hphantom{\sum_{k_{1},k_{2}}}
		 		\cdot
				 \xi \left(k_{1}, [v+k_{1}\theta], k_{2}, [v+b (r+s)+(k_{2}-k)\theta]\right)
				 .
				 \end{split}
		 	\end{align}
		\end{lemma}
		
		The proof is straightforward.
		
		\begin{remark}\label{rmk:def:mcPb}
			If we let $\mathcal{P}_{b}^{\pm}$ be the completion of $P^{\infty}$ with respect to the above inner product, then $\iota_{0}$ extends, by construction, to a unitary $\mathcal{P}_{b}^{\pm}\cong\mathcal{E}_{b}^{\pm}$.
		\end{remark}
		The next goal is to prove the that $\mathcal{E}_{b}$ contains functions of Schwartz decay.
		
		\begin{proposition}\label{prop:incl-Schwartz-into-mcEb} 
			The injective linear map $\iota_{0}\colon P^{\infty}=C_{c}^{\infty} (\mathbb{Z}\times\mathbb{T}\times\mathbb{R})\odot C_{c}^{\infty} (\mathbb{R}) \to \mathcal{E}_{b}^{\pm}$ from Lemma~\ref{lem:mcE-simplified-bumpy-version} extends to an injective linear map $\incl \colon \Schw_{2} \to \mathcal{E}_{b}^{\pm}$. Moreover, 
			the image of $\Schw_{2}$ is a right-$\mathfrak{A}\odot\mathfrak{A}$ pre-Hilbert submodule of $\mathcal{E}_{b}^{\pm}$. The module structure on $\Schw_{2}$ induced by $\incl$ is given by the same formulas as on $P^{\infty}$.
		\end{proposition}
		
		\begin{corollary}\label{cor:mcPb-contains-Schw2}
			The completion $\mathcal{P}_{b}^{\pm}$ of $P^{\infty}$ has $\Schw_{2}$ as a dense subspace.%
		\end{corollary}
		
		The main tool needed for the proof of Proposition~\ref{prop:incl-Schwartz-into-mcEb} (see page~\pageref{pf:prop:incl-Schwartz-into-mcEb}) is the following result, proved using some estimates of quadruple series of rapid decay, and its corollaries:		
		\begin{lemma}\label{lem:I-norm-smaller-than-Snorm} 
			For any integer $N\geq 6$, there exists a finite number $\mu(N) \geq 0$ with the following property: If $ F_{1} , F_{2} \in \Schw_{2}$, then for all $M,N\geq 6$,
			\begin{align*}
				\norm{\inner{ F_{1} }{ F_{2} }}_{I}
				\leq
				\mu(M)
				\cdot
				\normSt{ F_{1} }{M}{0}
				\cdot
				\mu(N)
				\cdot
				\normSt{ F_{2} }{N}{0}
				,
			\end{align*}
			where we define the inner product of two Schwartz--Bruhat functions by the same formula as Equation~\ref{eq:inner-of-P}.
		\end{lemma}
		For a definition of the $I$-norm, see \cite{Renault:gpd-approach}. Note that this, in particular, implies that $\inner{ F_{1} }{ F_{2} }$ is indeed a function on $\mathcal{A}$ (\emph{i.e.}, that it takes finite values). 	With this tool, one proves the following:
		\begin{lemma}\label{lem:mcEb-inner-still-ok-for-Schw2} 
			If $ F_{n}\in P^{\infty} = C_{c}^{\infty} (\mathbb{Z}\times\mathbb{T}\times\mathbb{R})\odot C_{c}^{\infty} (\mathbb{R})$ converges to $ F \in \Schw_{2}$ with respect to $\normSt{\,\cdot\,}{M}{0}$, and $ G_{n}\in P^{\infty}$ to $ G \in \Schw_{2}$ in $\normSt{\,\cdot\,}{N}{0}$ for some $M,N\geq 6$, then $\inner{ F_{n}}{ G_{n}}^{\mathcal{P}}$ converges to $\inner{ F }{ G }^{\Schw}$ in $C^*(\mathcal{A})$. Consequentially, the function $\inner{ F }{ G }^{\Schw}$ is an element of $C^*(\mathcal{A})=A_{\theta}\otimes A_{\theta}$.
		\end{lemma}
		
		Using the fact that the $I$-norm dominates the C*-norm (see \cite[Chap.~II, Prop.~4.2.(ii)]{Renault:gpd-approach}), we conclude:
		\begin{corollary}\label{cor:lem:I-norm-smaller-than-Snorm} 
			For any integers $M,N\geq 6$ and with $\mu(N)$ as in Lemma~\ref{lem:I-norm-smaller-than-Snorm}, we have for all $ F_{j}$ in $\Schw_{2}$:
			\begin{align*}
				\norm{\inner{ F_{1} }{ F_{2} }^{\Schw}}_{C^*(\mathcal{A})}
				\leq
				\mu(M)
				\cdot
				\normSt{ F_{1} }{M}{0}
				\cdot
				\mu(N)
				\cdot
				\normSt{ F_{2} }{N}{0}
				.
			\end{align*}
			In particular, if $F\in P^{\infty}$, then $
				\norm{\iota_{0}( F )}_{\mathcal{E}_{b}}
				\leq
				\mu(N)
				\cdot
				\normSt{ F }{N}{0}
				.$
		\end{corollary}
				
		Using the fact that the C*-norm dominates the uniform-norm (see \cite[Prop.\ 4.1 (i)]{Renault:gpd-approach}), we also conclude:
		\begin{corollary}\label{lem:contractive-mcEb}
			For $F$ in $\Schw_{2}$, we have
			\begin{align*}
					\norm{\inner{ F }{ F }^{\Schw}}_{C^*(\mathcal{A})}
					\geq
					\sup{}{
					\left\{
						\int_{\mathbb{R}}
						\abs{ F }^{2}
						\left(k	,[v], s - r, r \right)		
						\,\mathrm{ d }\, r
					\,:\,
					[v]\in\mathbb{T}, k\in\mathbb{Z}, s\in\mathbb{R}
					\right\}}.
				\end{align*}
		\end{corollary}

		\begin{lemma}\label{lem:right-action-gives-back-Schw2}
			If $ F $ in $\Schw_{2}$ and $ \xi $ in $\mathfrak{A}\odot \mathfrak{A}$, and if $F ._{\Schw_{2}} \xi $ is defined by the same formula as Equation~\ref{eq:action-on-P}, then $F ._{\Schw_{2}} \xi $ is an element of $\Schw_{2}$. Moreover, if $ F_{n}\in P^{\infty} = C_{c}^{\infty} (\mathbb{Z}\times\mathbb{T}\times\mathbb{R})\odot C_{c}^{\infty} (\mathbb{R})$ converges to $ F $ in $\Schw_{2}$, then $ F_{n} ._{\mathcal{P}} \xi $ converges to $ F ._{\Schw_{2}} \xi $ in $\Schw_{2}$.
		\end{lemma}

		\begin{proof}[Proof of Proposition~\ref{prop:incl-Schwartz-into-mcEb}]\label{pf:prop:incl-Schwartz-into-mcEb}
			Take any $ F \in \Schw_{2}$ and let $ F_{n}\in P^{\infty} = C_{c}^{\infty} (\mathbb{Z}\times\mathbb{T}\times\mathbb{R})\odot C_{c}^{\infty} (\mathbb{R}) $ be a sequence which converges to $ F $ in $\Schw_{2}$; in particular, for any $\epsilon>0$ and for $n, m$ sufficiently large,
			\begin{align*}
				\normSt{ F_{n} - F_{m}}{4}{0}
				\leq
				\normSt{ F_{n} - F }{4}{0}
				+
				\normSt{ F - F_{m}}{4}{0}
				&< \epsilon.
			\end{align*}
			By Corollary~\ref{cor:lem:I-norm-smaller-than-Snorm}, the sequence $( \iota_{0}(F_{n}))_{n}$ is therefore Cauchy in $\mathcal{E}_{b}^{\pm}$ and hence converges; let $\incl ( F )$ denote the limit in $\mathcal{E}_{b}^{\pm}$. 
			Note that, if $\lim_{n}^{\Schw} F_{n} = F =0$, then $\lim_{n}^{\mathcal{E}} \iota_{0}(F_{n}) =0$ by the same corollary, so $\incl ( F )$ does not depend on the chosen sequence in $P^{\infty}$ and for $F\in P$, we have $\incl (F) = \iota_{0}(F)$. Using Corollary~\ref{cor:lem:I-norm-smaller-than-Snorm} yet again, we get  
			for any integer $N\geq 6$:
			\begin{align}
				\norm{\inner{\incl ( F )}{\incl ( F )}^{\mathcal{E}}}^{\frac{1}{2}}_{C^*(\mathcal{A})}
				&=
				\norm{\incl ( F )}_{\mathcal{E}_{b}}
				=
				\lim_{n\to\infty} \norm{\iota_{0}( F_{n} )}_{\mathcal{E}_{b}}
				\notag
				\\				
				&\leq
				\lim_{n\to\infty}\left(\normSt{F_{n}}{N}{0}\cdot \mu(N)\right)
				\label{eq:smaller-than-CF}
				=
				\normS{ F }{N}{0}\cdot \mu(N)
				.
			\end{align}
			To check that the extended map $\incl$ is injective, note first that there exists a constant $K$ such that for any $F\in\Schw_{2}$ and any $N\geq 2$:
			\[
				K \cdot \left(\normSt{F}{N}{0}\right)^{2}
				\geq
				\sup{}{
				\left\{
					\int_{\mathbb{R}}
					\abs{ F }^{2}
					\left(k	,[v], s - r, r \right)		
					\,\mathrm{ d }\, r
				\,:\,
				[v]\in\mathbb{T}, k\in\mathbb{Z}, s\in\mathbb{R}
				\right\}}				
				.
			\]
			Using Lemma~\ref{lem:contractive-mcEb}, this implies
			\[
				\norm{\incl( F )}_{\mathcal{E}_{b}}^{2}
				\geq
				\sup{}{
				\left\{
					\int_{\mathbb{R}}
					\abs{ F }^{2}
					\left(k	,[v], s - r, r \right)		
					\,\mathrm{ d }\, r
				\,:\,
				[v]\in\mathbb{T}, k\in\mathbb{Z}, s\in\mathbb{R}
				\right\}}
				,
			\]
			\emph{i.e.\!} if $\norm{\incl( F )}_{\mathcal{E}_{b}}=0$, then $F\equiv 0$, so $\incl$ is injective.
			Some more estimates with Lemma~\ref{lem:mcEb-inner-still-ok-for-Schw2} and Corollary~\ref{cor:lem:I-norm-smaller-than-Snorm} show
			\[
				\incl (F) ._{\mathcal{E}_{b}} \xi
				=
				\incl ( F._{\Schw_{2}} \xi )
		\text{ and }
				\inner{F}{G}^{\Schw}
				=
				\inner{\incl (F)}{\incl (G)}^{\mathcal{E}}
			\]
			where $\xi\in\mathfrak{A}\odot\mathfrak{A}$, which concludes our proof.
		\end{proof}

		\begin{remark}\label{rmk:Schw-in-both-Lb-and-Hb-to-mcEb}
			One proves \emph{mutatis mutandis} that the inclusions $C_{c}^{\infty}(\mathbb{Z}\times\mathbb{T}\times\mathbb{R}) \subseteq \module_{b}$ and $C_{c}^{\infty}(\mathbb{Z}\times\mathbb{T}\times\mathbb{R}) \subseteq \mathcal{H}_{b}^{\pm}$ (which are dense by Corollary~\ref{cor:smooth-dense-in-Lb}) extend to injective linear maps $\Schw_{1} \to \module_{b}$ resp.\ $\Schw_{1} \to \mathcal{H}_{b}^{\pm}$, and that the respective right pre-Hilbert module formulas on $C_{c}^{\infty}(\mathbb{Z}\times\mathbb{T}\times\mathbb{R})$ are still valid for elements in $\Schw_{1}$. 
			Fully analogously to the map $\iota_{0}\colon P^{\infty}=C_{c}^{\infty} (\mathbb{Z}\times\mathbb{T}\times\mathbb{R})\odot C_{c}^{\infty} (\mathbb{R}) \to \mathcal{E}_{b}^{\pm}$ from Lemma~\ref{lem:mcE-simplified-bumpy-version}, we could therefore have defined the map	
			\begin{equation*}
				\begin{tikzcd}[column sep = small]
					\iota'\colon\quad 
					\Schw_{1}\odot\mathcal{S}(\mathbb{R}) \ar[r]
					& \module_{b} \otimes_{A_{\theta}^{\otimes 2}} (A_{\theta} \otimes \mathcal{H}_{ -b }^{\pm})
					=
					\mathcal{E}_{b}^{\pm}
					\\[-20pt]
					\Phi
					\odot \psi
					\ar[r, mapsto]
					&
					\Phi
					\otimes
					(1_{A_{\theta}} \otimes 
					\varepsilon_{0} \otimes z^{0} \otimes \psi)
					,
				\end{tikzcd}	
			\end{equation*}
			which clearly also has dense image.
			By construction, $\iota'$ and $\iota_{0}$ give rise to the same extension, namely the injective linear map $\incl\colon \Schw_{2} \to \mathcal{E}_{b}^{\pm}$ from Proposition~\ref{prop:incl-Schwartz-into-mcEb}.  
		\end{remark}

	Now that we have simplified $\mathcal{E}_{b}$, we would like to show that it is unitarily equivalent to the module $\mathcal{R}$ from Theorem~\ref{thm:cycle-rep-Dudelta}.		
		\begin{definition}\label{def:Lambda}
			Let
			\begin{align*}
					\chi\colon \Schw_{2} \to \Schw_{2} ,
					\quad
					\chi (F) (k, [x], r, s) :=& \int_{t} F(k, [x], r, t) \me{ts}\mathrm{d}\, t,
			\\
			\text{and} \quad
					\Gamma\colon \Schw_{2} \to \Schw_{2},
					\quad
					\Gamma (F) (k, [x], r, s) :=& 
					F(k, [x-k\theta], b(r + s + k), s)			
			\end{align*}
			with inverses given by 
			\begin{align*}
				\chi\inv(F) (k, [x], r, s) :=& \int_{q} F(k, [x], r, q) \e{qs}\mathrm{d}\, q
				\\
				\text{and}\quad
				\Gamma\inv(F) (k, [x], r, s) :=& 
				F(k, [x+k\theta], \tfrac{r}{b} - s - k, s).
			\end{align*}
			And define
			\[\begin{tikzcd}[
					 ,/tikz/column 2/.append style={anchor=base east}
					 ,/tikz/column 3/.append style={anchor=base west}
					 , column sep = small]
				\Xi:= \Gamma \circ \chi \colon & \Schw_{2} \ar[r] & \Schw_{2} \\[-20pt]
					&\Xi ( F ) (k,[x],r,s)
					\ar[r, equal] &
					 \int_{t} F (k,[x-k\theta],b(r+s+k),t)\me{ts}\mathrm{d}\; t
					.
				\end{tikzcd}
			\]
			with inverse
			\[	
			\begin{tikzcd}
				&\Xi\inv ( F ) (k,[x],r,s) =
				\int_{q} F (k,[x+k\theta],\tfrac{r}{b}-q - k,q)\e{qs}\mathrm{d}\; q
				.
			\end{tikzcd}
			\]
		\end{definition}

		\begin{theorem}
			\label{thm:module-mcR}
				The map $\Xi$ extends to a unitary from $\mathcal{R}^{\pm}$, the completion of the pre-Hilbert module $R^{\infty}$ defined in Theorem~\ref{thm:cycle-rep-Dudelta}, to $\mathcal{P}_{b}^{\pm}$, the completion of the pre-Hilbert module $P^{\infty}$ defined in Lemma~\ref{lem:P-pre-H-module-structure}.
		\end{theorem}
		
		\begin{proof}
			A direct computation shows that the linear map $\Xi \colon R^{\infty}\to \Schw_{2}\subseteq \mathcal{P}_{b}^{\pm}$ preserves the pre-inner product and right $\mathfrak{A}\odot\mathfrak{A}$-module structure on $\Schw_{2}=R^{\infty}\subseteq \mathcal{R}^{\pm}$. As $\Xi$ is a bijection $\Schw_{2}\to\Schw_{2}$, and as $\Schw_{2}$ is dense in both $\mathcal{R}^{\pm}$ and $\mathcal{P}_{b}^{\pm}$ by definition, $\Xi$ extends to a unitary $\mathcal{R}^{\pm}\cong \mathcal{P}_{b}^{\pm}$.
		\end{proof}
		
		\begin{corollary}\label{cor:thm:module-mcR}
			The map $\incl\circ\Xi:R^{\infty}\to\mathcal{E}_{b}^{\pm} = \module_{b} \otimes_{ A_{\theta}^{\otimes 2} } (A_{\theta} \otimes \mathcal{H}_{b}^{\pm})$ extends to a unitary
			$\mathcal{R}^{\pm}\cong\mathcal{E}_{b}^{\pm}$, where $\incl$ is the injective linear map from Proposition~\ref{prop:incl-Schwartz-into-mcEb}.
		\end{corollary}

		We now turn to the operator.

		\begin{lemma}\label{prop:mcR-sa-and-reg}
			The closure of the operator $d_{\mathcal{R}}$ from Equation~\eqref{eq:mau-on-mcR} 
			is self-adjoint and regular.
		\end{lemma}
		
		\begin{proof}
			Because $\mathsf{M}^{\mathbb{R}}_{1}$ and $\mathsf{M}^{\mathbb{R}}_{2}$ are obviously symmetric in view of the inner product defined on $\mathcal{R}^{\pm}$ (see Equation~\ref{eq:inner-product-on-mcR}), so is $d_{\mathcal{R}}$. Since the domain of $d_{\mathcal{R}}$ is the dense set $\Schw_{2}$, it thus suffices to check that $d_{\mathcal{R}} \pm \imaginary$ has dense range.
			For any given $\psi_{1},\psi_{2}\in \Schw_{2}$, define
			\begin{align*}
				\phi_{1} (k,[x],r,s)
				:=&
				\frac{(r+\imaginary s)\cdot\psi_{2}(k,[x],r,s) \mp \imaginary \psi_{1}(k,[x],r,s)}{1 + s^{2} + r^{2} },
				\text{ and}
				\\{}\\
				\phi_{2} (k,[x],r,s)
				:=&
				\frac{(r-\imaginary s)\cdot\psi_{1}(k,[x],r,s) \mp \imaginary \psi_{2}(k,[x],r,s)}{1 + s^{2} + r^{2} }.
			\end{align*}
			These functions lie in the domain of our operator $d_{\mathcal{R}}$ and satisfy $(d_{\mathcal{R}}\pm \imaginary) (\phi_{1}\oplus \phi_{2}) = \psi_{1}\oplus \psi_{2}$, so the range of $d_{\mathcal{R}} \pm \imaginary$ contains $\Schw_{2}^{\oplus 2}$ and is hence dense.
		\end{proof}
			
		\begin{corollary}[using Theorem~\ref{thm:module-mcR}]\label{cor:d-mcPb-sa-reg}
			On $\mathcal{P}_{b}$, the closure of the operator 
			\begin{alignat}{3}\label{def:mau-on-mcPb}
				d_{\mathcal{P}} := \mat{ 0 & d_{\mathcal{P},-} \\ d_{\mathcal{P},+} & 0}
				\text{ where }
				d_{\mathcal{P},\pm} := &\,
				\Xi \circ d_{\mathcal{R},\pm} \circ \Xi\inv
				\text{ with }
				\dom (d_{\mathcal{P},\pm}) :=
				\Schw_{2} 
				\subseteq
				\mathcal{P}_{b}^{\pm},
			\end{alignat}
			is self-adjoint and regular.
		\end{corollary}
		Note that the definition of $d_{\mathcal{P},\pm}$ indeed makes sense since $d_{\mathcal{R},\pm}$ maps its domain $\Schw_{2}$ back into itself.
		A direct computation shows:
		\begin{lemma}\label{lem:d-mcPb-computed}
			We have
			\[
				\Xi \circ \mathsf{M}^{\mathbb{R}}_{1} \circ \Xi\inv
				=
				b \left(\mathsf{M}^{\mathbb{R}}_{1} +\mathsf{M}^{\mathbb{Z}}+ \mathsf{M}^{\mathbb{R}}_{2}\right)
				\quad	\text{ and } \quad
				\Xi \circ \mathsf{M}^{\mathbb{R}}_{2} \circ \Xi\inv
				=
				\frac{\imaginary}{2\pi}\left(\partial_{2} - \partial_{1} \right)
				,
			\]
			where $\Xi\colon \Schw_{2}\to\Schw_{2}$ is the map defined in Definition~\ref{def:Lambda}. In particular,
			\[
				d_{\mathcal{P},\pm}
				=
				\left[b (\mathsf{M}^{\mathbb{R}}_{1} + \mathsf{M}^{\mathbb{Z}}) \mp \frac{1}{2\pi}\,\partial_{1} \right]
				+			
				\left[ b\, \mathsf{M}^{\mathbb{R}}_{2} \pm \frac{1}{2\pi}\,\partial_{2} \right]
				.
			\]
		\end{lemma}

		We should remark that we have written $d_{\mathcal{P},\pm}$ in such a way because the $\mathbb{Z}$- and the first $\mathbb{R}$-component both arose from the copy of $\module_{b}$ inside of $\mathcal{E}_{b}^{\pm}$, while the second $\mathbb{R}$-component arose from the copy of $\mathcal{H}_{-b}^{\pm}$; \emph{cf.\!} the map $\iota_{0}$ in Lemma~\ref{lem:mcE-simplified-bumpy-version} with extension $\incl$ constructed in Proposition~\ref{prop:incl-Schwartz-into-mcEb}.
		
	As $\incl$ is an injective map and as $d_{\mathcal{P},\pm}$ maps its domain $\Schw_{2}$ back into itself, it makes sense to define the following operator on $\mathcal{E}_{b}$:
	\begin{alignat}{3}\label{def:mau-on-mcEb}
		D_{\mathcal{E}} := \mat{ 0 & D_{\mathcal{E},-} \\ D_{\mathcal{E},+} & 0}
		\text{ where }
		D_{\mathcal{E},\pm} := &
		\incl \circ d_{\mathcal{P},\pm} \circ \incl\inv
		\text{ with }
		\dom (D_{\mathcal{E},\pm}) :=
		\mathrm{ran}(\incl)
		\subseteq
		\mathcal{E}_{b}^{\pm}.
	\end{alignat}
	Note that $D_{\mathcal{E}}$ is densely defined according to Lemma~\ref{lem:mcE-simplified-bumpy-version}. Moreover, its closure is self-adjoint and regular\label{cor:mcP-sa-and-reg} because the closure of $d_{\mathcal{P}}$ is by Corollary~\ref{cor:d-mcPb-sa-reg}.
	
	Recall that we chose
	\[
				D_{\mathcal{H},\pm}
				=
				b \, 
				\mathsf{M}^{\mathbb{R}}
				\pm 
				\frac{1}{2\pi}
				\frac{\partial\;}{\partial r}
	\]
	in Equation~\eqref{def:DmcH}. Its domain can be chosen to be $
		\dom (D_{\mathcal{H},\pm})
		:=
		\Schw_{1}
		\subseteq
		\mathcal{H}_{-b}^{\pm}$
	thanks to Remark~\ref{rmk:Schw-in-both-Lb-and-Hb-to-mcEb}.
	\begin{lemma}
		Let
		\begin{alignat*}{3}
			D_{\module,\pm}
			:=&
				b  \left(\mathsf{M}^{\mathbb{R}} + \mathsf{M}^{\mathbb{Z}}\right) \mp \frac{1}{2\pi}\frac{\partial\;}{\partial r}
			&&&\text{ with }&
			\dom (D_{\module,\pm})
			:=
			\Schw_{1}
			\subseteq
			\module_{b}.
		\end{alignat*}
		On the image under $\incl$ of the subspace $\Schw_{1}\odot \mathcal{S}(\mathbb{R})$ of $\Schw_{2}$, we have
		\[
			D_{\mathcal{E},\pm}
			=
			D_{\module,\pm}
			\otimes_{A_{\theta}^{\otimes 2}}
			\left(1_{A_{\theta}\otimes\mathcal{H}_{-b}^{\pm}}\right)
			+
			1_{\module_{b}}
			\otimes_{A_{\theta}^{\otimes 2}}
			\left(
			1_{A_{\theta}}
			\otimes
			D_{\mathcal{H},\pm}
			\right)
			.
		\]
	\end{lemma}
	
	In the above, we have written $\otimes_{A_{\theta}^{\otimes 2}}$ (instead of the more customary $\otimes$) to emphasize that $\mathcal{E}_{b}^{\pm} = \module_{b} \otimes_{ A_{\theta}^{\otimes 2} } (A_{\theta}\otimes\mathcal{H}_{-b}^{\pm})$ is the \emph{balanced} tensor product (so it is not obvious \emph{a priori} that the above operator is well-defined).
	
	\begin{proof}
		Recall that $\incl$ is the extension of the map 
		\begin{equation*}
			\begin{tikzcd}[column sep = small, row sep = tiny]
				\iota_{0}\colon\quad 
				C_{c}^{\infty} (\mathbb{Z}\times\mathbb{T}\times\mathbb{R})\odot C_{c}^{\infty} (\mathbb{R}) \ar[r]
				& \module_{b} \otimes_{A_{\theta}^{\otimes 2}} (A_{\theta} \otimes \mathcal{H}_{ -b }^{\pm})
				=
				\mathcal{E}_{b}^{\pm} 
			\end{tikzcd}	
		\end{equation*}
		from Lemma~\ref{lem:mcE-simplified-bumpy-version} to all of $\Schw_{2}$. In particular, on the subspace $\Schw_{1}\odot \mathcal{S}(\mathbb{R})$, $\incl$ is given by the exact same formula as $\iota_{0}$, namely
		\[
			\incl(\Phi\odot \psi)
			=
			\Phi
			\otimes
			(1_{A_{\theta}} \otimes 
			\varepsilon_{0} \otimes z^{0} \otimes \psi)	
			.
		\]
		It is then obvious that $D_{\mathcal{E},\pm}$, defined as $\incl \circ d_{\mathcal{P},\pm} \circ \incl\inv$ with
		\[
			d_{\mathcal{P},\pm}
			=
			\left[b (\mathsf{M}^{\mathbb{R}}_{1} + \mathsf{M}^{\mathbb{Z}}) \mp \frac{1}{2\pi}\,\partial_{1} \right]
			+			
			\left[ b\, \mathsf{M}^{\mathbb{R}}_{2} \pm \frac{1}{2\pi}\,\partial_{2} \right]
		\]
		computed in Lemma~\ref{lem:d-mcPb-computed}, is indeed as claimed.
	\end{proof}
	
	We point out that $D_{\mathcal{E},\pm}$ is indeed well-defined, despite the balancing. To be more precise, one can show for all $a\in \mathfrak{A}$ and $\Phi \in \Schw_{2}\subseteq \module_{b}$ that
	\[
		D_{\module, \pm} \big(\Phi ._{ \module_{b}} (1\otimes a)\bigr)
		=
			\bigl[D_{\module, \pm}\Phi\bigr] ._{ \module_{b}} (1\otimes a)
			+
			b\cdot
			\Phi._{ \module_{b}} \left[1\otimes \left(\mathsf{M}^{\mathbb{Z}}\mp \frac{1}{2\pi} \frac{\partial\;}{\partial \Theta}\right) (a)\right]
	\]
	and for $\Psi\in \Schw_{2}\subseteq \mathcal{H}_{-b}^{\pm}$ that
	\[
		D_{\mathcal{H}, \pm} \big(a._{\mathcal{H}_{-b}} \Psi\bigr)
		=
		a._{\mathcal{H}_{-b}} \bigl[D_{\mathcal{H}, \pm} \big(\Psi\bigr)\bigr]
		+ b\cdot
		\left( \mathsf{M}^{\mathbb{Z}}\mp \frac{1}{2\pi} \frac{\partial\;}{\partial \Theta}\right) (a)._{\mathcal{H}_{-b}} \Psi.
	\]
	This implies that, for all $a', c\in \mathfrak{A}$:
	\begin{equation*}
	\left[D_{\module,\pm}\bigl(\Phi ._{ \module_{b}} (a'\otimes a)\bigr)\right] \otimes (c \otimes \Psi)
		+
		\bigl(\Phi ._{ \module_{b}} (a'\otimes a)\bigr) \otimes (c \otimes \left[D_{\mathcal{H},\pm} \Psi\right])
	\end{equation*}
	equals 
	\begin{equation*}
		\left[ D_{\module,\pm}\Phi\right] \otimes \bigl(a'c \otimes ( a ._{\mathcal{H}_{ -b }} \Psi)\bigr)
		+
		\Phi \otimes \bigl(a'c \otimes  D_{\mathcal{H},\pm}\left[ a ._{\mathcal{H}_{ -b }} \Psi\right]\bigr)
		,
	\end{equation*}
	so that $D_{\mathcal{E}}$ is indeed well-defined.
	
	\begin{lemma}
		The operator $d_{\mathcal{P},\pm}$ leaves the subspace $\Schw_{1}\odot \mathcal{S}(\mathbb{R})$ of $\Schw_{2}$ invariant. Moreover, $\Schw_{1}\odot \mathcal{S}(\mathbb{R})$ is a core for $d_{\mathcal{P},\pm}$.
	\end{lemma}
	
	The invariance is obvious, and the proof regarding the core requires only an application of Corollary~\ref{cor:lem:I-norm-smaller-than-Snorm} (in fact, one proves that any subspace of $\Schw_{2}=\dom(d_{\mathcal{P},\pm})$ which is dense with respect to the family of seminorms on $\Schw_{2}$, is a core for $d_{\mathcal{P},\pm}$). Since $D_{\mathcal{E},\pm} := 
			\incl \circ d_{\mathcal{P},\pm} \circ \incl\inv$ (see Equation~\eqref{def:mau-on-mcEb}), a consequence is that Item~\ref{item:LeMe-core} holds for $(\mathcal{E}_{b}, D_{\mathcal{E}})$:
	\begin{corollary}\label{lem:D-item:LeMe-core}
		The dense $\mathfrak{A}\odot\mathfrak{A}$-submodule $\dom (D_{\module}) =  \Schw_{1}\subseteq \module_{b}$ makes \mbox{$\Schw_{1}\odot_{\mathfrak{A}\odot \mathfrak{A}} \dom( 1_{A_{\theta}}\otimes D_{\mathcal{H}})$} a core for $D_{\mathcal{E}}$.
	\end{corollary}

	Item~\ref{item:LeMe-creation} holds as well for $(\mathcal{E}_{b}, D_{\mathcal{E}})$:
	\begin{lemma}\label{lem:D-item:LeMe-creation}
		For all $\Phi\in \dom (D_{\module})=\Schw_{1}\subseteq \module_{b}$, both operators $\eta\mapsto D_{\mathcal{E},\pm} (\Phi\otimes \eta) - \Phi \otimes (1_{A_{\theta}}\otimes D_{\mathcal{H},\pm})(\eta)$ with domain $\dom(1_{A_{\theta}}\otimes D_{\mathcal{H},\pm})$ extend to adjointable operators $A_{\theta}\otimes \mathcal{H}_{ -b }^{\pm} \to\mathcal{E}_{b}^{\pm}$.
	\end{lemma}

	\begin{proof}
		For $\eta = a \otimes \Psi$ for $a\in \mathfrak{A}$ and $\Psi\in \Schw_{1} \subseteq \mathcal{H}_{-b}^{\pm}$ and $\Phi\in \Schw_{1} \subseteq \module_{b}$:
		\begin{align*}
			&D_{\mathcal{E},\pm} (\Phi\otimes_{A_{\theta}^{\otimes 2}} \eta) - \Phi \otimes_{A_{\theta}^{\otimes 2}} (1_{A_{\theta}}\otimes D_{\mathcal{H},\pm}) (\eta)
			\\&\quad=
			(D_{\module,\pm}\Phi)\otimes_{A_{\theta}^{\otimes 2}} (a \otimes \Psi)
			+
			\Phi\otimes_{A_{\theta}^{\otimes 2}} (a \otimes  D_{\mathcal{H},\pm} (\Psi)) - \Phi \otimes_{A_{\theta}^{\otimes 2}} (1_{A_{\theta}}\otimes D_{\mathcal{H},\pm}) (\eta)
			\\&\quad=
			(D_{\module,\pm}\Phi)\otimes_{A_{\theta}^{\otimes 2}} \eta
			=
			T_{D_{\module,\pm}\Phi}(\eta).
		\end{align*}		
		We conclude for general $\eta$ that $D_{\mathcal{E},\pm} (\Phi\otimes_{A_{\theta}^{\otimes 2}} \eta) - \Phi \otimes_{A_{\theta}^{\otimes 2}} (1_{A_{\theta}}\otimes D_{\mathcal{H},\pm}) (\eta) = T_{D_{\module,\pm}\Phi}(\eta)$, so
		we have shown that the operator in question is a creation operator, which is clearly adjointable.
	\end{proof}

	\begin{proposition}\label{prop:cycle-rep-Dudelta-uglier}
		The pair $(\mathcal{E}_{b},D_{\mathcal{E}})$ is a Kasparov cycle and represents $\Dudelta_{\theta}$.
	\end{proposition}

	\begin{proof}
		Recall that $\Dudelta_{\theta}$ was defined as $\left[ \module_{b} \right] \otimes_{A_{\theta}^{\otimes 2}} (1_{A_{\theta}}\otimes x_{ -b })$. We have checked that the items in Theorem~\ref{thm:LeMe:sums-reg-sa-ops-Thm7.4} are all satisfied:
		
		As explained on page \pageref{cor:mcP-sa-and-reg}, the closure of $D_{\mathcal{E}}$ is self-adjoint and regular (\emph{i.e.\!} Item~\ref{item:LeMe-s.a.+reg} holds) because $D_{\mathcal{E}}$ is unitarily equivalent to the operator $d_{\mathcal{P}}$, whose closure is self-adjoint and regular by Corollary~\ref{cor:d-mcPb-sa-reg}.
		We explained that $(0, D_{\mathcal{E}})$ is a weakly anticommuting pair (\emph{i.e.\!} Item~\ref{item:LeMe-WAP} holds), and in Lemma~\ref{lem:D-item:LeMe-core}, we have proved that for the dense submodule $\mathcal{X}:=\Schw_{1}$ of $\module_{b}$, the algebraic tensor product $\Schw_{1}\odot \dom (1_{A_{\theta}}\otimes D_{\mathcal{H}})$ is a core for $D_{\mathcal{E}}$ (\emph{i.e.\!} Item~\ref{item:LeMe-core} holds). Lastly, in Lemma~\ref{lem:D-item:LeMe-creation} we have shown that, for $\Phi\in \mathcal{X}$, the operator $D_{\mathcal{E}} T_{\Phi} - T_{\Phi} (1_{A_{\theta}}\otimes D_{\mathcal{H}})$ has adjointable extension, \emph{i.e.\!} Item~\ref{item:LeMe-creation} holds as well.
	\end{proof}
	
	\begin{proof}[Proof of Theorem~\ref{thm:cycle-rep-Dudelta}]
		We have shown in Corollary~\ref{cor:thm:module-mcR} that $\mathcal{R}$ is unitarily equivalent to $\mathcal{E}_{b}$, and we have defined $D_{\mathcal{E}}$ exactly so that the unitary equivalence turns it into $d_{\mathcal{R}}$. The claim now follows from Proposition~\ref{prop:cycle-rep-Dudelta-uglier}.
	\end{proof}

\


\end{document}